\documentclass{amsart} 

\usepackage{luatex85}
\usepackage[normalem]{ulem}
\usepackage{tikz-cd}
\numberwithin{equation}{section}
\usepackage[hyphens]{url} 
\usepackage{breakurl}

\usepackage{longtable}
\usepackage[all]{xy}
\usepackage{amsmath}
\usepackage{amssymb}
\usepackage{comment}
\usepackage{stmaryrd}
\usepackage{xcolor}
\usepackage{mathtools}
\usepackage{enumitem}
\usepackage[T1]{fontenc}

\definecolor{Vino}{rgb}{0.356,0,0}
\definecolor{Vino2}{rgb}{0.5,0,0}
\definecolor{Ruta}{rgb}{0.309, 0.459, 0.208}
\usepackage[final,colorlinks,linkcolor=Ruta,citecolor=Vino,urlcolor=black,breaklinks]{hyperref}
\let\cal\mathcal
\def\Ascr{{\cal A}}

\def\Dscr{{\cal D}}

\def\Fscr{{\cal F}}

\def\Hscr{{\cal H}}

\def\Kscr{{\cal K}}
\def\Lscr{{\cal L}}

\def\Oscr{{\cal O}}
\def\Pscr{{\cal P}}

\def\Tscr{{\cal T}}

\def\Wscr{{\cal W}}

\let\blb\mathbb
\def\CC{{\blb C}}

\def\QQ{{\blb Q}}

\def \PP{{\blb P}}
\def \AA{{\blb A}}
\def \ZZ{{\blb Z}}
\def \TT{{\blb T}}
\def \NN{{\blb N}}
\def \RR{{\blb R}}

\def\id{\text{id}}
\def\Id{\operatorname{id}}

\def\quot{/\!\!/}

\def\mod{\operatorname{mod}}

\def\length{\mathop{\text{length}}}

\def\coh{\mathop{\text{\upshape{coh}}}}

\def\Spec{\operatorname {Spec}}

\def\Ext{\operatorname {Ext}}
\def\Hom{\operatorname {Hom}}

\def\RHom{\operatorname {RHom}}

\def\relint{\operatorname {relint}}

\def\coker{\operatorname {coker}}
\def\ker{\operatorname {ker}}

\def\id{{\operatorname {id}}}

\def\rk{\operatorname {rk}}

\def\r{\rightarrow}

\DeclareMathOperator{\Perf}{Perf}

\let\dirlim\injlim

\newtheorem{lemma}{Lemma}[section]
\newtheorem{proposition}[lemma]{Proposition}
\newtheorem{theorem}[lemma]{Theorem}
\newtheorem{corollary}[lemma]{Corollary}

\newtheorem{convention}[lemma]{Convention}

\theoremstyle{definition}

\newtheorem{example}[lemma]{Example}
\newtheorem{definition}[lemma]{Definition}
\newtheorem{conjecture}[lemma]{Conjecture}

{

\newtheorem{hypothesis}{Hypothesis}
}

\theoremstyle{remark}

\newtheorem{remark}[lemma]{Remark}

\newtheorem{notation}[lemma]{Notation}

\newdimen\uboxsep \uboxsep=1ex
\def\uboxn#1{\vtop to 0pt{\hrule height 0pt depth 0pt\vskip\uboxsep
\hbox to 0pt{\hss #1\hss}\vss}}

\def\uboxs#1{\vbox to 0pt{\vss\hbox to 0pt{\hss #1\hss}
\vskip\uboxsep\hrule height 0pt depth 0pt}}

\let\oldmarginpar\marginpar
\long\def\marginpar#1{\oldmarginpar{\raggedright \tiny\baselineskip 0pt \lineskip 0pt #1}}

\def\TT{\mathbb{T}}

\def\Re{\operatorname{Re}}

\def\Im{\operatorname{Im}}
\def\Re{\operatorname{Re}}

\def\WF{\Wscr\Fscr}
\def\Log{\operatorname{Log}}
\def\Tw{\Dscr}
\def\Two{\operatorname{Tw}}
\def\Ho{\operatorname{Ho}}
\def\grad{\operatorname{grad}}
\def\tl{\operatorname{tl}}

\RequirePackage{luatex85}
\usepackage{amssymb}
\usepackage{amsmath}
\usepackage{graphicx}
\usepackage[all]{xy}
\usepackage{tikz}
\usepackage{pgfplots}
\pgfplotsset{compat=1.16}
\usetikzlibrary{arrows.meta,bending,cd,babel,shadings,shadows,snakes,external}
\providecommand{\charclass}{}
\renewcommand{\charclass}[1][1]{%
\tikzsetnextfilename{charclass_cached_#1}
\begin{tikzpicture}[baseline=(current bounding box.center),scale=#1]
\begin{scope}[cm={1.5,0.0,0.0,1.5,(0.0,0.0)},,,anchor=center,inner sep=0,outer sep=0]
\clip(0,0) rectangle (8,5);
\path[draw,fill=gray!10,draw=white]  
   (8,0) -- (5.2,1.75) -- (5.2,3.25) -- (8,5) -- (4,5) -- (4,-5) -- cycle;
\path[draw,dashed,fill={rgb,1:red,0.826;green,0.868;blue,0.742}]  
   (8,-0.833333) .. controls (7.446511,-0.445776)  and (3.455069,1.775372)
    .. (3.475,2.5) .. controls (3.494489,3.208538)  and (7.464249,5.249863)
    .. (8,5.625);
\path[draw,color=gray!20]  
   (4,0) -- (4,5);
\path[draw,color=gray!20]  
   (4.133333,0) -- (4.133333,5);
\path[draw,color=gray!20]  
   (4.266667,0) -- (4.266667,5);
\path[draw,color=gray!20]  
   (4.4,0) -- (4.4,5);
\path[draw,color=gray!20]  
   (4.533333,0) -- (4.533333,5);
\path[draw,color=gray!20]  
   (4.666667,0) -- (4.666667,5);
\path[draw,color=gray!20]  
   (4.8,0) -- (4.8,5);
\path[draw,color=gray!20]  
   (4.933333,0) -- (4.933333,5);
\path[draw,color=gray!20]  
   (5.066667,0) -- (5.066667,5);
\path[draw,color=gray!20]  
   (5.2,0) -- (5.2,5);
\path[draw,color=gray!20]  
   (5.333333,0) -- (5.333333,5);
\path[draw,color=gray!20]  
   (5.466667,0) -- (5.466667,5);
\path[draw,color=gray!20]  
   (5.6,0) -- (5.6,5);
\path[draw,color=gray!20]  
   (5.733333,0) -- (5.733333,5);
\path[draw,color=gray!20]  
   (5.866667,0) -- (5.866667,5);
\path[draw,color=gray!20]  
   (6,0) -- (6,5);
\path[draw,color=gray!20]  
   (6.133333,0) -- (6.133333,5);
\path[draw,color=gray!20]  
   (6.266667,0) -- (6.266667,5);
\path[draw,color=gray!20]  
   (6.4,0) -- (6.4,5);
\path[draw,color=gray!20]  
   (6.533333,0) -- (6.533333,5);
\path[draw,color=gray!20]  
   (6.666667,0) -- (6.666667,5);
\path[draw,color=gray!20]  
   (6.8,0) -- (6.8,5);
\path[draw,color=gray!20]  
   (6.933333,0) -- (6.933333,5);
\path[draw,color=gray!20]  
   (7.066667,0) -- (7.066667,5);
\path[draw,color=gray!20]  
   (7.2,0) -- (7.2,5);
\path[draw,color=gray!20]  
   (7.333333,0) -- (7.333333,5);
\path[draw,color=gray!20]  
   (7.466667,0) -- (7.466667,5);
\path[draw,color=gray!20]  
   (7.6,0) -- (7.6,5);
\path[draw,color=gray!20]  
   (7.733333,0) -- (7.733333,5);
\path[draw,color=gray!20]  
   (7.866667,0) -- (7.866667,5);
\path[draw]  
   (4,0) -- (4,5);
\node[anchor=east, cm={1.5,0.0,0.0,1.5,(0.0,0.0)}] at (2.0,2.5) {$\mathbb{C}_{\operatorname{Re}\le 0}$};
\node[anchor=east, cm={1.5,0.0,0.0,1.5,(0.0,0.0)}] at (3.9,0.625) {$\scriptstyle {\operatorname{Re}= 0}$};
\path[draw,fill=black]  
   (4.05,2.5) .. controls (4.05,2.527614)  and (4.027614,2.55) .. (4,2.55)
    .. controls (3.972386,2.55)  and (3.95,2.527614) .. (3.95,2.5)
    .. controls (3.95,2.472386)  and (3.972386,2.45) .. (4,2.45)
    .. controls (4.027614,2.45)  and (4.05,2.472386) .. cycle;
\node[anchor=east, cm={1.5,0.0,0.0,1.5,(0.0,0.0)}] at (3.9,2.5) {$\scriptstyle 0$};
\path[draw,color={rgb,1:red,0.5;green,0;blue,0},dashed]  
   (4.7,2.5) .. controls (4.7,2.886599)  and (4.386599,3.2) .. (4,3.2)
    .. controls (3.613401,3.2)  and (3.3,2.886599) .. (3.3,2.5)
    .. controls (3.3,2.113401)  and (3.613401,1.8) .. (4,1.8)
    .. controls (4.386599,1.8)  and (4.7,2.113401) .. cycle;
\node[anchor=south east,color={rgb,1:red,0.5;green,0;blue,0}, cm={1.5,0.0,0.0,1.5,(0.0,0.0)}] at (3.41,1.9100000000000001) {$\scriptstyle S^1$};
\path[draw,color=brown,dashed]  
   (5.4,0) .. controls (5.16961,0.399048)  and (4.35,2.037076) .. (4.35,2.5)
    .. controls (4.35,2.962924)  and (5.16961,4.600952) .. (5.4,5);
\node[anchor=west,color=brown, cm={1.5,0.0,0.0,1.5,(0.0,0.0)}] at (4.9399999999999995,4.5) {$\scriptstyle l$};
\node[color={rgb,1:red,0.42;green,0.56;blue,0.14}, cm={1.5,0.0,0.0,1.5,(0.0,0.0)}] at (6.5,4.375) {U};
\path[draw,color={rgb,1:red,0.5;green,0;blue,0}]  
   (4.01063,1.800079) .. controls (4.382126,1.805721)  and (4.684496,2.100614)
    .. (4.699436,2.471853)
    .. controls (4.714375,2.843092)  and (4.436679,3.161329)
    .. (4.066837,3.196804);
\path[draw,color=brown]  
   (4.803849,1.192404) .. controls (4.563887,1.720378)  and (4.35,2.263518)
    .. (4.35,2.5) .. controls (4.35,2.71808)  and (4.531895,3.19695)
    .. (4.748393,3.684218);
\path[draw,fill=white,color=white]  
   (8,0) -- (5.2,1.75) -- (5.2,3.25) -- (8,5) -- cycle;
\node[color=gray!70, cm={1.5,0.0,0.0,1.5,(0.0,0.0)}] at (4.35,4.65) {V};
\path[draw]  
   (8,0) -- (5.2,1.75) -- (5.2,3.25) -- (8,5);
\node[anchor=west, cm={1.5,0.0,0.0,1.5,(0.0,0.0)}] at (6.0,2.5) {$\mathbb{C}_{\operatorname{Re}\ge 0}$};
\end{scope}
\end{tikzpicture}%
}%

\makeatletter
\newcommand*\bigcdot{\mathpalette\bigcdot@{.5}}
\newcommand*\bigcdot@[2]{\mathbin{\vcenter{\hbox{\scalebox{#2}{$\m@th#1\bullet$}}}}}
\makeatother

\makeatletter
\newcommand{\pushright}[1]{\ifmeasuring@#1\else\omit\hfill$\displaystyle#1$\fi\ignorespaces}
\newcommand{\pushleft}[1]{\ifmeasuring@#1\else\omit$\displaystyle#1$\hfill\fi\ignorespaces}
\makeatother

\author[\v{S}pela \v{S}penko and Michel Van den Bergh]{\v{S}pela \v{S}penko and Michel
  Van den Bergh} 
\address[\v{S}pela \v{S}penko]{D\'epartement de Math\'ematique, Universit\'e Libre de Bruxelles, Campus de la Plaine CP 213, Bld du Triomphe, B-1050 Bruxelles}
\email{spela.spenko@ulb.be}
\address[Michel Van den Bergh]{Vakgroep Wiskunde, Universiteit Hasselt, Universitaire Campus \\
  B-3590 Diepenbeek}
\email{michel.vandenbergh@uhasselt.be}
\address[Michel Van den Bergh]{Vakgroep Wiskunde en Data Science, Vrije Universiteit Brussel, Pleinlaan 2, 1050 Brussel} 
\email{michel.van.den.bergh@vub.be}
\thanks{The first author is supported by a MIS grant from the National Fund for Scientific Research (FNRS) and an ARC grant from the Université Libre de Bruxelles. The second author is a senior researcher at the Research
  Foundation Flanders (FWO).  While working on this project he was
  supported by the ERC grant SCHEMES and FWO grant G0D8616N: ``Hochschild cohomology and
  deformation theory of triangulated categories''. Part of this paper was written while the authors were in residence at the Simons Laufer Mathematical Sciences Institute (formerly MSRI) in Berkeley, California, during the Spring 2024 semester.}


\title[HMS symmetries of toric boundary divisors]{HMS symmetries of toric boundary divisors}
\def\LGr{\operatorname{LGr}}

\def\conv{\operatorname{conv}}
\def\Skel{\operatorname{Skel}}
\def\Grass{\operatorname{Grass}}
\def\spantwo{\operatorname{span}}

\def\tl{\operatorname{tl}}
\def\vol{\operatorname{vol}}
\def\HMS{\operatorname{HMS}}
\begin{document}  
\begin{abstract}
  Let $X$  
  be a projective crepant resolution of a Gorenstein affine toric variety and let
  $((\CC^*)^k,f)$ be the LG-model which is the Hori-Vafa mirror dual of $X$.
  Let ${D}$ be a generic fiber of $f$ equipped with the restriction of the standard Liouville
  form on $(\CC^*)^k$.  Let
  $\Kscr_A$ be the so-called ``stringy K\"ahler moduli space'' of $X$. We show that $\pi_1(\Kscr_A)$ 
  acts 
  on the wrapped Fukaya category of ${D}$. Using results by Gammage -- Shende and Zhou, this result implies that
  $\pi_1(\Kscr_A)$ acts on $D^b(\coh(\partial X))$ where $\partial X$ is the toric boundary
  divisor  of $X$. We show that the induced action of $\pi_1(\Kscr_A)$ on $K_0(\coh(\partial X))$ may be 
extended in a natural way to an action on $K_0(X)$ which corresponds to a GKZ system.
\end{abstract}

\maketitle
\setcounter{tocdepth}{1}
\tableofcontents
\section{Introduction}

Let\footnote{In order to avoid making too many technical definitions we first state our results in somewhat restricted generality. At the end of the introduction we   specify the actual hypotheses under which we prove each result.} $X$ be a crepant projective resolution of 
a Gorenstein affine toric variety, possibly by a DM-stack.\footnote{We say that a morphism
  between DM-stacks is projective if the induced morphism between the
  coarse moduli spaces is projective.}\footnote{Projectivity is needed in order to be able to construct a tropical limit on the mirror.} 
Let $A$ be the generators of the one-dimensional cones of the fan of $X$.  
The \emph{Hori-Vafa mirror} for $X$ is given by the LG model
$((\CC^*)^A,f)$ where $f$ is a generic Laurent polynomial with
monomials given by $A$. 
 The
\emph{stringy K\"ahler moduli space (SKMS)} $\Kscr_A$ of $X$ is the
DM-stack
\[
(\CC^A-V(E_A))/(\CC^\ast)^k
\]
where $V(E_A)$ is the divisor in $\CC^A$ defined by the
\emph{principal $A$-determinant} \cite[\S10]{GKZbook} (see \S\ref{sec:prindet}), after identifying the $\CC^A$ 
with the space of Laurent polynomials that are linear combinations of the monomials corresponding to the columns of $A$. The $(\CC^\ast)^k$-action is given by rescaling the
variables in the Laurent polynomials.  See \S\ref{eq:KA} for a precise definition and for example  \cite{Iritani,CCIT,DonovanSegal,Kite} for background.

Homological mirror symmetry suggests the existence of an action of $\pi_1(\Kscr_A)$ on $D^b(\coh(X))$.
 Indeed one expects an
action of $\pi_1(\Kscr_A)$ by some kind of ``parallel transport''  on the Fukaya category associated to
the LG model $((\CC^\ast)^k,f)$. This action should then be transported to an action on $D^b(\coh(X))$.
Unfortunately, while there exist a number of definitions for the Fukaya 
category associated to an LG-model, which apply in the generality that we need  \cite{AA, GammageShende, Hanlon,HanlonHicks}, these typically involve some extra choices (e.g.\ a tropical limit) which
obstruct the construction of a parallel transport action, or at least make it nonobvious.\footnote{Recently Wenyuan Li pointed out to us that perhaps a suitable definition is given in \cite{jeffs}.}

Therefore in this paper we aim for a more easily achievable goal.
Namely, we construct an action of $\pi_1(\Kscr_A)$ by parallel transport on the
wrapped Fukaya category $\Wscr\Fscr(D)$ ($\ZZ$-graded with $\ZZ$-coefficients) of a generic fiber ${D}$ of $f$, equipped with the restriction of a toric Liouville form on $(\CC^\ast)^k$. 
\begin{theorem}[\protect{Theorem \ref{th:mainth1}}] \label{th:mainth1intro} 
There exists a natural action of $\pi_1(\Kscr_A)$ on $\Wscr\Fscr(D)$.
\end{theorem}

This result may then be combined with HMS for toric boundary divisors as in \cite{GammageShende}.  
\begin{theorem}[\protect{Theorem \ref{th:mainth1}, Theorem \ref{prop:lastmile}, \cite[Theorem 1.0.1]{GammageShende}}] \label{th:mirror1000} 
We have $\Tw(\Wscr\Fscr(D)_\CC)\cong D^b(\coh(\partial X))$  where
$\partial X$ is the toric boundary divisor of $X$, i.e.\ the complement of the generic orbit.
Hence by Theorem \ref{th:mainth1intro} 
we obtain a natural action of $\pi_1(\Kscr_A)$ on  $D^b(\coh(\partial X))$.
\end{theorem}
\begin{remark}
The equivalence $\Tw(\Wscr\Fscr(D)_\CC)\cong D^b(\coh(\partial X))$ is proved in \cite{GammageShende} 
for a Liouville form on ${D}$, obtained from tropical degeneration, which is different from the restriction of the ambient toric Liouville form we are using.
However we can complete our argument by showing that the two Liouville forms are isotopic, proving the conjecture stated in \cite[Remark 5.3.3]{GammageShende}. See Theorem \ref{prop:lastmile}.

Moreover \cite{GammageShende} require the fan of $X$ to satisfy  a certain technical condition called ``per\-fectly-centered''.
We can remove this condition by using K\"ahler potentials of degree two during the proof. The usefulness of such potentials was realised by Zhou in \cite{Zhou}.
\end{remark}
Of course the ``natural action'' of $\pi_1(\Kscr_A)$ on
$D^b(\coh(\partial X))$, whose existence is asserted in Theorem
\ref{th:mirror1000}, could in principle be trivial. To show that this is
not the case we relate the action of $\pi_1(\Kscr_A)$ on
$K_0(\Tw(\Wscr\Fscr(D)_\CC))$ obtained from Theorem
\ref{th:mainth1intro} (and hence the corresponding action on
$K_0(\coh(\partial(X)))$ via Theorem \ref{th:mirror1000}) 
to the action of $\pi_1(\Kscr_A)$ on $H^{k-1}(D,\ZZ)$ obtained from
the Gauss-Manin connection (the Gauss-Manin connection is obtained
from the fact that $\Kscr_A$ describes the possible coefficients of
$f$, up to rescaling the variables).  

\begin{theorem}[Proposition \ref{prop:Lazarev}] \label{th:GKZ1} Assume that $X$ is a scheme. Then there exists an isomorphism
$K_0(\Tw(\Wscr\Fscr(D)_\CC))\cong H^{k-1}(D,\ZZ)$ which is compatible with the action of  $\pi_1(\Kscr_A)$ on $K_0(\Tw(\Wscr\Fscr(D)_\CC))$ obtained from Theorem \ref{th:mainth1intro} and
the action of $\pi_1(\Kscr_A)$
on $H^{k-1}(D,\ZZ)$ obtained from  the Gauss-Manin connection.
\end{theorem}
\begin{remark} The reason why we  need to assume $X$ to be a scheme is that, although Theorem \ref{th:GKZ1} is purely a symplectic statement, we use HMS, and more specifically
the results in \cite{HanlonHicks}, to prove it. The results in loc.\ cit.\ assume that $X$ is a scheme.
\end{remark}

It is well-known that the Gauss-Manin connection on $H^{k-1}(D,\CC)$ is closely related to the GKZ system \cite{GFK0, SchulzeWalther,Reichelt} (which yields in particular that it is non-trivial):
\begin{proposition}[{Proposition \ref{prop:GKZ1_}, \eqref{eq:sesH_}}] \label{prop:GKZ1}
The long exact relative cohomology sequence for the pair $(\CC^{\ast k},D)$ is $\pi_1(\Kscr_A)$-equivariant and yields the following subsequence
\begin{equation}\label{eq:sesH}
0\r H^{k-1}(\CC^{\ast k},\ZZ)\r H^{k-1}(D,\ZZ)\r H^{k}(\CC^{\ast k},D,\ZZ)\r H^{k}(\CC^{\ast k},\ZZ)\r 0
\end{equation}
such that  $H^{k-1}(\CC^{\ast k},\ZZ)\cong \ZZ^k$, $H^{k}(\CC^{\ast k},\ZZ)\cong \ZZ$ are trivial representations and such that the induced action of  $\pi_1(\CC^A-V(E_A))$ 
(via the canonical map $\pi_1(\CC^A-V(E_A))\r \pi_1(\Kscr_A)$) 
on $H^{k-1}(\CC^{\ast k},D,\CC)$ 
is given by the
GKZ system  with trivial parameters corresponding to the matrix A.
\end{proposition}
The action of $\pi_1(\Kscr_A)$ on $D^b(\coh(\partial X))$  in Theorem \ref{th:mirror1000}
may be regarded
as an approximation to the expected action on $D^b(\coh(X))$. On the level of Grothendieck groups
$X$ and $\partial X$ are related in the following way.
\begin{theorem}[{Theorem \ref{eq:K0boundary}}]\label{thm:K0boundary}
Let $i:\partial X\hookrightarrow X$ be the inclusion.  Then we have an exact sequence
\begin{equation}
\label{eq:K0ses:intro}
0\r X(T)\xrightarrow{\partial} K_0(\coh(\partial X))\xrightarrow{i_\ast} K_0(X)\xrightarrow{\rk} \ZZ\r 0\,.
\end{equation}

\end{theorem}
The two exact sequences  \eqref{eq:sesH}\eqref{eq:K0ses:intro}
look very similar and indeed we use
  the results of \cite{Lazarev} and \cite{HanlonHicks} to show that there is a natural isomorphism between them 
  when $X$ is a scheme (see Proposition \ref{prop:seqsiso}):
\[
\footnotesize
\begin{tikzcd}
0\ar[r]& H^{k-1}(\CC^{\ast k},\ZZ)\ar[r]\ar[d,"\cong"] &H^{k-1}(D,\ZZ)\ar[r]\ar[d,"\cong"]& H^{k}(\CC^{\ast k},D,\ZZ)\ar[r]\ar[d,"\cong"] &H^{k}(\CC^{\ast k},\ZZ)\ar[r] \ar[d, "\cong"]&0\\
0\ar[r]& X(T)\ar[r, "\partial"'] &K_0(\coh(\partial X))\ar[r,"i_\ast"']& K_0(X)\ar[r,"\rk"']& \ZZ\ar[r]& 0
\end{tikzcd}
\]
By transporting the $\pi_1(\Kscr_A)$-action on $H^{k}(\CC^{\ast k},D,\ZZ)$ to $K_0(X)$ via the third vertical arrow, we see
in particular that the $\pi_1(\Kscr_A)$-action on $K_0(\coh(\partial X))$ obtained from Theorem \ref{th:mirror1000} extends in a natural way to an action on $K_0(X)$ which corresponds to the GKZ system.

\begin{corollary}[Corollary \ref{cor:GKZ}]\label{cor:K0Xaction}
The induced action on $K_0(X)_\CC$ is obtained from the GKZ-system corresponding to the matrix $A$ with trivial parameters. 
\end{corollary}

\begin{remark}
We will also discuss in \S\ref{sec:another_action} another method to obtain an action of 
$\pi_1(\Kscr_A)$ on $K_0(X)$ when $X$ is a scheme. It is easy to see that the action of $\pi_1(\Kscr_A)$ on
$D^b(\coh(\partial X))$ descends to an action on $\Perf(\partial X)$. 
This implies that $\pi_1(\Kscr_X)$ acts on $K_{\operatorname{top}}^0(\partial X)$ which we show to be isomorphic to $K_{\operatorname{top}}^0(X)$. The latter  turns out to be  isomorphic to $K_0(X)$. 
Hence the $\pi_1(\Kscr_A)$-action on  $D^b(\coh(\partial X))$  gives an action on $K_0(X)$. 
We do not know if this action is the same as the previous one, or if it is even non-trivial.
\end{remark}

\subsection*{Comment on assumptions}\label{conditions}
The results mentioned above are actually proved under the following assumptions:
\begin{itemize}
  \item
Theorem \ref{th:mainth1intro}: any $A$ and $D=f^{-1}(0)$.
\item
Theorem \ref{th:mirror1000}: ``star-shaped'' triangulation, cf. \S\ref{sec:stacky}.
\item
Theorem \ref{th:GKZ1}, Corollary \ref{cor:K0Xaction}: ``star-shaped'' triangulation and  $X$ is a scheme ($A$ consists of all lattice points of $P = {\rm conv}(A)$).
\item
Proposition \ref{prop:GKZ1}: any $A$ with $|A|>1$ and $D=f^{-1}(0)$.
\item
Theorem \ref{thm:K0boundary}: any smooth toric DM stack.
\end{itemize}
The relation between the setting in the current introduction and the setting in the rest of the paper (e.g. the relation between the zero and the generic fiber) is explained in \S\ref{sec:adjoin}.

\section{Acknowledgement}
We thank Denis Auroux, Jeff Hicks, Cristopher Kuo, Oleg Lazarev, Wenyuan Li, Vivek Shende and Peng Zhou for patiently answering our (sometimes naive) questions and for providing insightful comments.

\section{$A$-notions}\label{sec:toricprelims}
Various interesting objects can be associated to finite subsets $A\subset \ZZ^k$. We collect some in this section. Our setting will be more general than in the introduction
since we do not assume $A\subset \ZZ^{k-1}\times \{1\}$.
\subsection{The principal $A$-determinant}
\label{sec:prindet}
Let $k \in \NN$ and $A=\{a_1,\dots,a_d\}$ be a finite subset of $\ZZ^k$. Let $P\subset \RR^k$ be the convex hull of $A$ and let $\Fscr(P)$ be
its set of faces.
Let $x=(x_1,\dots,x_k)$ be variables and for  $\alpha\in \CC^A$ put 
\[
f_\alpha(x):=\sum_{a_i\in A} \alpha_i x^{a_i}.
\]
For $F\in \Fscr(P)$ we denote by $\nabla_F$ the set of all $\alpha\in \CC^{F}$ such that
the singular locus of $V(f^F_\alpha)$ for
$f^F_\alpha:=\sum_{a_i\in F} \alpha_i x^{a_i}$ (i.e.\ the common zeros
of $\partial_i f^F_\alpha$ for $\partial_i=\partial/\partial_{x_i}$
and $f^F$) meets $(\CC^*)^k\subset \CC^k$. If $\overline{\nabla}_F$ is a
hypersurface in $\CC^{F}$, then we denote by $\Delta_F$ its defining
equation. Otherwise, we set $\Delta_F=1$. The \emph{principal $A$-determinant} is a polynomial
which satisfies
$E_A:=\prod_{F\in \Fscr(P)} \Delta_F^{m_F}$ for suitable nonzero
multiplicities $m_F\in \NN$  \cite[Theorem 10.1.2]{GKZbook}.

We set $V(A):=\bigcup_F p^{-1}_F(\nabla_F)$ where $p_F:\CC^{A}\to \CC^{F}$ is the projection.
It is clear that $V(E_A)$ and $V(A)$ can only differ in codimension $\leq 2$ and in fact they coincide:

\begin{proposition}\label{lem:VEA=VA}\cite[Proposition 1.1]{SVdBGKZdisc}
We have $V(A)=V(E_A)$. 
\end{proposition}
For use below we note the following proposition.
\begin{proposition}[{\cite[Proposition 9.1.4]{GKZbook}}]
\label{prop:affine} If $p:\ZZ^k\r \ZZ^l$ is an injective affine transformation then $V(A)=V(p(A))$.
\end{proposition}
Note that if $A\subset A'$ have the same convex hull then
\[
V(A)=V(A')\cap \CC^A.
\]
We let $(\CC^*)^k$ act on $\CC^{A}$ with weights given by the elements of $A$. Note that $V(A)$ is $(\CC^*)^k$-invariant for this action. 
We define the quotient stack:
\begin{equation}
\label{eq:KA}
\Kscr_A:=(\CC^A-V(A))/(\CC^*)^k.
\end{equation}
It is  not hard to see that $\Kscr_A$ is a Deligne-Mumford stack if $A$ spans $\RR^k$. 
Thanks to \cite[Example 5.6]{MR3144243} and the fact that $(\CC^\ast)^k$ is connected, we have a presentation for $\alpha\not\in V(A)$:
\begin{equation}
\label{eq:presentation}
\pi_1((\CC^\ast)^k,1 )\r \pi_1(\CC^{A}-V(A),\alpha) \r \pi_1(\Kscr_A,\alpha)\r 1
\end{equation}
The following lemma which follows directly from the definition of $V(A)$ is convenient.
\begin{lemma} \label{lem:convenient} $V(A)$ contains the coordinate planes $\{\alpha_i=0\}$ for those $a_i\in A$ which are vertices of $P$.
\end{lemma}
We recall the following lemma.
\begin{lemma} \label{lem:notzero}
Let $\alpha\in \CC^A$. If $|\{\alpha_i\neq 0\}|>1$ (in particular if $|A|>1$ and $\alpha\not\in V(A)$ by Lemma \ref{lem:convenient}) then $0$ is  in the image of $f_\alpha(x)$ considered as a function 
$(\CC^\ast)^k\r \CC$.
\end{lemma}
\begin{proof} There must be some variable which occurs with at least two different powers. We may assume it is $x_1$.
Write $f_\alpha(x)=\sum_{i\in I} p_i(\hat{x}) x_1^i$ where $\hat{x}=(x_2,\ldots,x_k)$ and $p_i\neq 0$ for $i\in I$. By hypothesis $|I|>1$. We may find $\hat{y}\in (\CC^\ast)^{k-1}$ such that 
$p_i(\hat{y})\neq 0$ for all $i\in I$. Now let $y_1\neq 0$ be a zero of $\sum_{i\in I} p_i(\hat{y}) x_1^i$ and put $y=(y_1,\hat{y})$. Then $f_\alpha(y)=0$.
\end{proof}

\subsection{Polyhedral decompositions of $\RR^k$}
\label{sec:polyhedral}
Let $\nu$ be a  strictly convex\footnote{\label{footnote:3} Our convention for convexity
    of a function is that $x^2$ is convex whereas $-x^2$ is not. This is opposite to \cite[Definition 6.1.4]{CoxLittleSchenck}, which we otherwise use as a reference.} 
function $A\r \ZZ$  so that the singular locus of the
associated convex piecewise linear function $P\r \RR$ 
 is a triangulation $\Delta_\nu$ of $P$.

For $a_i\in A$ we
let $\Hscr(a_i)$ be the function $\RR^k\r \RR:u\mapsto a_i\cdot u  -\nu(a_i)$. Let $\Pi_\nu\subset \RR^k$ be  the nonsmooth locus of the function 
\begin{equation}
\label{eq:legendre}
\RR^k \r \RR:u\mapsto \max_{a_i\in A}\Hscr(a_i).
\end{equation}
We call $\Pi_\nu$ the \emph{tropical amoeba} of $f_\alpha$ (although
it does not depend on $\alpha$, but only on $\nu$).  

For $\tau$ a face in $\Delta_\nu$ (which we identify with the set of $a_i$ that belong to it) we define $C_\tau$ by
\begin{equation}
\label{eq:Cequations}
\begin{aligned}
\Hscr(a_i)&=\Hscr(a_j) && \text{for all $a_i,a_j\in \tau$},\\
\Hscr(a_i)&\ge \Hscr(a_j) && \text{for all $a_i\in \tau$ and $a_j\not\in \tau$}.
\end{aligned}
\end{equation}
Note $C_\tau=\bigcap_{{a_i} \in \tau} C_{a_i}$. 
The fact that $\nu$ is strictly convex on $A$ implies that the $C_\tau$ form the faces of the polyhedral decomposition of $\RR^k$ determined by $\Pi_\nu$.
We have  $\dim C_\tau=k-|\tau|+1$ and $C_\tau$ is a face of $\Pi_\nu$ if and
only if $|\tau|\ge 2$ (in the case $|\tau|=1$, $C_\tau$ is a chamber in the complement of $\Pi_\nu$).

For later use we note that the affine span $L_\tau$ of $C_\tau$ is given by
\begin{align}\label{eq:Ltau}
\Hscr(a_i)&=\Hscr(a_j) \qquad \text{for all $a_i,a_j\in \tau$}.
\end{align}
For $0<\epsilon\ll 1$ we  put
\begin{equation}
\label{eq:tau_epsilon}
C_{\tau,\epsilon}=\{p\in \RR^k\mid \forall {a_i}\in \tau: d(p,C_{a_i})\le \epsilon\}.
\end{equation}
One verifies there exists a constant $u$, depending only on $\Pi_\nu$, such
that 
for all $p\in C_{\tau,\epsilon}$
\begin{equation}
\label{eq:compact1}
d(p,C_\tau)<u\epsilon.
\end{equation}

\subsection{Tropical amoeba from Laurent polynomials} 
If for a fixed $\alpha\in (\CC^\ast)^A$ we put 
\[
f_t=\sum_{i=1}^d \alpha_i t^{-\nu(a_i)} x^{a_i}
\]
then we have \cite[Proposition 6.1.3]{GammageShende}\cite{Mikhalkin}
\[
\lim_{t\r\infty} \Log_t f^{-1}_{t}(0)=\Pi_\nu
\]
where $\Log_t(x)=\Log(x)/\log t$ with $\Log(x):=(\log |x_i|)_{i=1}^k$. 
\begin{remark} \label{rem:dominance}
It is clear from the definition that $p=\Log_t x\in C_{a_\ell}$ if and only if for all $i$ one has 
\[
\left|\frac{t^{-\nu(a_i)} x^{a_i}}{t^{-\nu(a_\ell)} x^{a_\ell}}\right|\le 1.
\]
\end{remark}
\subsection{Tailoring}
\label{sec:tailoring}
It was observed by  Mikhalkin in \cite{Mikhalkin} that $f^{-1}_t(0)$ can
be deformed into a non-holomorphic hypersurface with an easier to
understand combinatorial structure. This procedure was streamlined in
\cite{Abouzaid} and also in \cite{Zhou}. Fix $\alpha\in
(\CC^\ast)^d$. Choose\footnote{In \cite[\S4]{Abouzaid} it is assumed
  $\epsilon_1=\epsilon/2$, $\epsilon_2=\epsilon$.  However under the
  restriction $\epsilon_1=\epsilon_2/2$ the conditions imposed on the
  $(\chi_{a_i})_i$ are not independent of the metric.}
$1\gg \epsilon_2>\epsilon_1>0$. Let $\chi_{a_i}:\RR^k\r \RR$ be smooth
functions with the following property 
\begin{equation}
\label{eq:distance}
\begin{aligned}
d(p,C_{a_i})&\le \epsilon_1 \Rightarrow \chi_{a_i}(p) =1\\
d(p,C_{a_i})&\ge \epsilon_2\Rightarrow \chi_{a_i}(p) =0
\end{aligned}
\end{equation}
where the distance function is computed with respect to a constant Riemannian metric  (see Remark \ref{rem:metric}
below for a comment on metrics).

In addition we demand that the $\chi_{a_i}$ are strictly decreasing in the radial direction from $C_{a_i}$ and that the level set  $\{\chi_{a_i}(p)=\xi\}$ for $0<\xi<1$ are convex. 
For $p\in \RR^k$ we put
\begin{align*}
A_0(p)&=\{a_i\in A\mid d(p,C_{a_i})<\epsilon_1\},\\
A_1(p)&=\{a_i\in A\mid \epsilon_1\le d(p,C_{a_i})\le \epsilon_2\},\\
A_2(p)&=\{a_i\in A\mid d(p,C_{a_i})>\epsilon_2\}.
\end{align*}
For all $p$ we have $A=A_0(p)\coprod A_1(p)\coprod A_2(p)$. 
For every $p\in \RR^k$ the set  $\tau(p):=A_0(p)\coprod A_1(p)$ is a face of $\Delta_\nu$. 
We use this notation to put an additional condition on
$\chi_{a_i}$. We demand\footnote{See \cite[\S1.2]{Zhou} how this can be achieved.}  that for $\tau$ a face of $\Delta_\nu$, $a_i \in \tau$,  the level surface  
$\{p\mid \chi_{a_i}(p)=\xi,  
\tau(p)=\tau\}$ for  $0<\xi<1$ is parallel to $C_\tau$ 
 outside a compact set.

It follows from this that the $(\chi_{a_i})_i$ are determined by their values on a compact region of~$\RR^k$ and hence we obtain that there is some constant $K_1>0$ such that 
\begin{equation}
\label{eq:parbound}
\forall j=1,\ldots,k,\forall p\in \RR^k:|(\partial_j \chi_{a_i})(p)|\le K_1/(\epsilon_2-\epsilon_1).
\end{equation}
We put for $s\in [0,1]$
$
\chi_{a_i,s}=(1-s)+s\chi_{a_i}
$
and\footnote{In \cite{Abouzaid} one uses the functions $\phi_{a_i}(x):=1-\chi_{a_i}(\Log_t x)$ and it is moreover assumed that $\forall j:\alpha_j=1$. Thus the notation
in loc.\ cit.\ is:
$f_{t,s}(x)=\sum_{a_i\in A} (1-s\phi_{a_i}(x)) x^{a_i}$.}
\begin{equation}
\label{eq:ft1}
f_{t,s}(x)=\sum_{i=1}^d \alpha_i \chi_{a_i,s}(\Log_t x) t^{-\nu(a_i)} x^{a_i}
\end{equation}
for some fixed $\alpha\in (\CC^\ast)^A$. We state some observations.
\begin{enumerate}
\item \label{it:tau}  
Inspecting the definition of $\tau(p)$, we see that $f_{t,1}$ is obtained by gluing pieces which only depend 
on the faces of $\Delta_\nu$ (see \S\ref{sec:pantscover} below). 
\item There is a very useful refinement of Remark \ref{rem:dominance}.
Let $p=\Log_t x\in C_{a_\ell}$ so that in particular $\chi_{a_\ell}(p)=1$. If $a_i\in A_1(p)\cup A_2(p)$
then by \cite[Lemma 4.5]{Abouzaid}
\begin{equation}
\label{eq:negligible}
\left|\frac{t^{-\nu(a_i)}x^{a_i}}{t^{-\nu(a_\ell)}x^{a_\ell}}\right|\le t^{-K_2\epsilon_1}
\end{equation}
for $t\gg 0$ and for a suitable constant $K_2>0$ not depending on
$x,t,\epsilon_i$.\footnote{The constant $K_2$ here could be replaced
  by $c$ in loc.cit. This constant $c$ comes from Lemma 4.1 in
  loc.cit. and it does not depend on $x,t$ and not on $\epsilon_i$ for
  $\epsilon_i$ sufficiently small which is what we assume here.} 
\item For $p=\Log_t x$ it will be advantageous to write
\[
f_{t,s}(x)=f^{\text{core}}_{t,s}(x)+f^{\text{small}}_{t,s}(x)
\]
where 
\begin{align*}
f^{\text{core}}_{t,s}(x)&=\sum_{a_i\in A_0(p)} \alpha_i t^{-\nu(a_i)}  x^{a_i}\\
f^{\text{small}}_{t,s}(x)&=\sum_{a_i\in A_1(p)\cup A_2(p)} \alpha_i t^{-\nu(a_i)} \chi_{a_i,s}(\Log_t x) x^{a_i}.
\end{align*}
Note that $f^{\text{core}}_{t,s}(x)$ is holomorphic. Using \eqref{eq:negligible} one may bound the terms in $f^{\text{small}}_{t,s}(x)$, as well as in  $(\partial f^{\text{small}}_{t,s})(x)$ and $(\bar{\partial} f^{\text{small}}_{t,s})(x)$ (using
in addition \eqref{eq:parbound}).
\item For use below we note the following. Assume $t\gg 0$. Let $x\in f^{-1}_{t,s}(0)$,   $p=\Log_tx\in C_{a_\ell}$.
Then by a combination of the formulas on the first display on \cite[p1134]{Abouzaid} and on the first display \cite[p1135]{Abouzaid}\footnote{In loc. cit. $\delta$ is $a_\ell$ in the notation here, while $\gamma$ is a vertex of $\tau(p)$ different from $a_\ell$. Note that $\tau(p)$ cannot consist of only one vertex.} together with \eqref{eq:negligible}
we obtain:
\begin{equation}
\label{eq:mainbound}
\left| \frac{\partial f^{\text{core}}_{t,s}(x)}{t^{-\nu(a_\ell)}x^{a_\ell}}\right|
\ge K_3
\end{equation}
where $K_3>0$ is a constant independent of $x, t,\epsilon_i$.
\end{enumerate}
\begin{remark} \label{rem:metric} We have been sloppy by not specifying the metrics in \eqref{eq:distance} and \eqref{eq:mainbound}. The reason is that all the metrics we will use are globally equivalent
  and they are also compatible with the metrics used in \cite{Abouzaid}. See \S\ref{sec:zhou} and in particular Proposition \ref{prop:bounds} below for a more in-depth discussion.
So  \eqref{eq:distance}\eqref{eq:negligible}\eqref{eq:mainbound} are true for all such metrics with suitably altered $(\epsilon_i)_i$ and constants $(K_i)_i$.
\end{remark}
For the comfort of the reader we state some preliminary lemmas which are implicit in \cite{Abouzaid}.
\begin{lemma} \label{rem:easysmooth} 
Assume $|A|>1$. Then $f^{-1}_{t,s}(0)$ is smooth and non-empty for $s\in [0,1]$ and $t\gg 0$.
\end{lemma}
\begin{proof}
That the fibers are smooth if non-empty follows easily from \eqref{eq:parbound}\eqref{eq:negligible}\eqref{eq:mainbound}.  Indeed for $x\in f^{-1}_{t,s}(0)$
  we have to prove $(df_{t,s})(x)=(\partial +\bar{\partial})(f_{t,s})(x)\neq 0$. If we split $f_{t,s}(x)$ as $f^{\text{core}}_{t,s}(x)+f^{\text{small}}_{t,s}(x)$ then a quick verification 
(using the fact that $\bar{\partial} f^{\text{core}}_{t,s}(x)=0$) shows that the contribution given by \eqref{eq:mainbound} dominates. 

By \cite[Lemma 4.12]{Abouzaid} it follows that all $f^{-1}_{t,s}(0)$ are symplectomorphic (without a priori assuming that they are non-empty). Hence it is sufficient to prove that $f^{-1}_{t,0}(0)$ is  non-empty. This follows from Lemma \ref{lem:notzero}
since for the Laurent polynomial $f_{t,0}$ all the coefficients $\alpha_i t^{-\nu(a_i)}$ are nonzero.
\end{proof}
\begin{lemma} \label{lem:nui} Denote the coefficients of the Laurent polynomial $f_{t,0}$ by $\alpha_\nu$. Thus $(\alpha_{\nu})_i=\alpha_i t^{-\nu(a_i)}$ for $i=1,\ldots,d$. Then for $t\gg 0$ we have $\alpha_\nu\not\in V(A)$.
\end{lemma}
\begin{proof} This follows from the definition of $V(A)$ together with Lemma \ref{rem:easysmooth} applied for all $f^F_{p_F(\alpha_\nu)}$ for $F$ a face that is not a vertex, for $F$ a vertex this follows as we fixed $\alpha\in (\CC^*)^A$.
\end{proof}
\begin{remark}
\label{rem:cutoff}
Below in \S\ref{sec:paistropical}
we will use \eqref{eq:parbound}\eqref{eq:negligible}\eqref{eq:mainbound} to obtain more precise information about $f^{-1}_{t,s}(0)$.
In particular it will follow from Theorem \ref{prop:tropical} below combined with Moser's lemma (see Theorem \ref{lem:moser2}) that $f^{-1}_{t,1}(0)\cong f^{-1}_{t,0}(0)$ as Liouville
manifolds (see \cite[Proposition 4.9]{Abouzaid} for the symplectic structure). Since $f^{-1}_{t,0}(0)$ does not depend on the cutoff functions it follows that $f^{-1}_{t,s}(0)$
is independent (up to isomorphism of Liouville manifolds) of cutoff functions for all $s\in [0,1]$.
\end{remark}
\subsection{The $A$-toric variety}
\label{sec:Atoric}
Let $X_A$ be the closure of $\{(x^{a_1}:\dots:x^{a_d})\in \PP^{d-1}\mid x\in (\CC^*)^k\}$ in $\PP^{d-1}$. Then $X_A$ is a (usually nonnormal) toric variety for the torus $(\CC^\ast)^k$, which acts
with weights $(a_i)_i$ on the homogeneous coordinates of $\PP^{d-1}$. We recall the following result from \cite{GKZbook}.
\begin{proposition}\cite[Proposition 5.1.9, Theorem 5.3.1]{GKZbook}\label{prop:XA}
The $(\CC^\ast)^k$-orbits in $X_A$ are in $1:1$ correspondence with the faces of $P$. The orbit corresponding to a face $F$ of $P$ consists of 
$\{(z_1:\dots:z_d)\in X_A\mid z_i\neq 0\iff a_i\in F\}=\{(z_1:\dots:z_d)\in
  \PP^{d-1}\mid \exists u\in (\CC^*)^k\text{such that} \;z_i=u^{a_i}\; \text{if
    $a_i\in F$}\text{ and } z_i=0\;\text{if $a_i\not\in F$}\}$. 
\end{proposition}
For a proof of the last equality see \cite[Proposition 2.1]{SVdBGKZdisc}.

\subsection{Stacky fans and star-shaped triangulations}
\label{sec:stacky}
Recall that a stacky fan \cite[\S3]{borisov2005orbifold} $\Sigma$ in
$N_\RR$ for a lattice $N$ is a fan in which the one-dimensional cones~$\tau$ come with chosen generators (``stacky primitives'') $v_\tau\in \tau\cap N$ (for an
ordinary fan, $v_\tau$ is taken to be the generator of $\tau\cap N$ as
a semi-group). In contrast to loc.\ cit.\ we will not assume that a
stacky fan is automatically simplicial as it is not necessary for the basic definitions (as explained in \cite[Remark
3.4]{borisov2005orbifold} it is necessary for smoothness).  We write $T_N=N\otimes_{\ZZ}\CC^\ast$ and
we let $X_\Sigma$ be the toric stack corresponding to $\Sigma$ (see
\cite[\S3]{borisov2005orbifold}). $X_\Sigma$ contains $T_N$ as an open
substack and the complementary closed substack, denoted by
$\partial X_\Sigma$, is called the \emph{toric boundary}.

\medskip

Assume $0\in A$. Following \cite[Definition 3.3.1]{GammageShende} we say that
an integral triangulation of $P={\rm conv}(A)$ with vertices $A$ is \emph{star-shaped} if every
simplex in $\Tscr$ which is not contained in $\partial P$ has $0$ as a
vertex. Such a triangulation (which need not exist) gives rise to a pure full-dimensional simplicial stacky fan $\Sigma$ whose cones
are of the form $\RR_{\ge 0}\tau$ for $\tau$ a simplex in $\Tscr$ not
containing $0$.\footnote{An example of a fan that does not arise from a star-shaped triangulation is given by the fan determined by two $2$-dimensional cones spanned by $(-1,2)$, $(0,1)$ and $(0,1)$, $(1,2)$.}
In particular the 1-dimensional cones are the half lines $\RR_{\ge 0}v$ where $v$ is a vertex of $\Tscr$ different from $0$. The stacky primitive on this cone is taken to be $v$. We obtain
\[
A=\{\text{stacky primitives}\} \coprod \{0\}.
\]
\begin{definition}
Let $\Tscr$ be a star-shaped triangulation of $P$ with vertices $A$ 
with corresponding stacky fan $\Sigma$.
We call $\Kscr_A$ (see \eqref{eq:KA}) the \emph{stringy K\"ahler moduli space} of $X_\Sigma$.
\end{definition}
\begin{remark} Note that this definition does not precisely cover the situation considered in the introduction where definitely $0\not\in A$. 
See \S\ref{sec:SKMS} below for the connection.
\end{remark}
\subsection{Stacky fans and the  FLTZ-skeleton}
\label{sec:fltz}
Let $\Sigma\subset N_\RR$ be a  stacky fan as in \S\ref{sec:stacky}. Let $S_{N^\ast}$ be the real torus $N^\ast_{\RR}/N^\ast$. We have
\[
 T^\ast (S_{N^\ast})=S_{N^\ast}\times N_\RR, \quad  T_{N^\ast}=T(S_{N^\ast})=S_{N^\ast}\times N^\ast_\RR,
\]
where $T_{N^\ast} =N^\ast\otimes_\ZZ \CC^{\ast}$ is the dual torus of $T_N$.
We define a (singular) conical Lagrangian of $T^\ast (S_{N^\ast})$ via:
\[
\Lambda_\Sigma=\bigcup_{\sigma\in \Sigma} \sigma^\perp\times \sigma
\]
where for a cone $\sigma$, $\sigma^\perp\subset S_{N^\ast}$ is a (possibly non-connected) subtorus  of  $S_{N^\ast}$ defined via: 
\[
\sigma^\perp=\{\eta\in S_{N^\ast}\mid \forall \tau\in \Sigma,\tau\subset\sigma, \dim\tau=1:\langle \eta,v_\tau\rangle\cong 0\mod \ZZ\}.
\]
Put $S^\infty N_\RR=(N_\RR-0)/\RR_{>0}$ and $T^{\ast,\infty}(S_{N^*})=S_{N^\ast}\times S^\infty N_\RR$, $T^{\infty}(S_{N^*})=S_{N^\ast}\times S^\infty {N^*_\RR}$.
The FLTZ-skeleton $\Lambda_\Sigma^\infty$ \cite{FLTZ} (see also \cite{RSTZ}) is the image 
of $\Lambda_\Sigma \cap  (S_{N^\ast} \times (N_\RR-0))$ in $S_{N^\ast} \times S^\infty N_\RR$. This is a (singular) Legendrian in the contact manifold  $T^{\ast,\infty}(S_{N^\ast})$.

\section{Some symplectic geometry}
\subsection{Liouville domains and manifolds}
\label{sec:liouville}
 An \emph{exact symplectic manifold} is a pair
  $(N,\theta)$ where $N$ is a manifold (with boundary if specified) and $\theta$ is a 1-form
  such that $\omega_\theta:=d\theta$ is nondegenerate.
 If $N$ is given then we call $\theta$ \emph{a Liouville form}.  The \emph{Liouville vector field} $Z_\theta$
  associated to $\theta$ is defined by
  \[
    i_{Z_\theta}\omega_\theta=\theta.
  \]
    A morphism of exact symplectic manifolds $(N,\theta)\r (N',\theta')$ is a pair\footnote{We make the function $f$ part of the data of an exact
      morphism since it makes it easier to consider families.} $(\phi,f)$ where
          $\phi:N\r N'$ is a smooth morphism, $f:N\r\RR$ is a smooth function with compact support 
          and $\phi^\ast(\theta')+df=\theta$.

If $N$ has no boundary then  we say that a Liouville form $\theta$ is \emph{complete} if $Z_\theta$
      is a complete vector field. In that case we call the corresponding flow $(\phi_{Z_\theta,t})_t$ the \emph{Liouville flow}
      and we call $(N,\theta)$ a \emph{complete exact symplectic manifold}.       If $(N,\theta)$ is complete then we say that a subset $M\subset N$ \emph{generates} $N$ if $\bigcup_{t\ge 0} \phi_{Z_\theta,t} M=N$.

A \emph{Liouville domain} is a  compact exact symplectic manifold with boundary $(M,\theta)$ such that $Z_\theta{\mid} \partial M$ is outward pointing.
A Liouville domain $M$ has a \emph{Liouville completion} $(M,\theta)\subset (\hat{M},\hat{\theta})$ constructed e.g.\ as
\begin{equation}
  \label{eq:completionformula}
\hat{M}=  \dirlim_{s\r -\infty} (M,e^{-s}\theta)
\end{equation}
where the transition maps in the direct limit are given by $\phi_{Z_\theta,s}$ for $s<0$. Here
$(\hat{M},\hat{\theta})$ is a complete exact symplectic manifold and $M$ generates $\hat{M}$ as defined above.

A \emph{Liouville manifold}\footnote{What we call a Liouville manifold
  is sometimes called Liouville manifold of finite type, e.g. in
  \cite{CiEl}.} is a complete exact symplectic manifold $(N,\theta)$
containing a generating Liouville domain. If $(M,\theta)$ is such a
generating Liouville domain then $\hat{M}\cong N$. 
If $N$ is a Liouville manifold then the \emph{skeleton} $\Skel(N)$ consists of the set of points that do not escape to infinity under the Liouville flow. 
\begin{remark}
If $N$ is a Liouville manifold then  every generating Liouville domain is a deformation retract by the inverse Liouville flow.
\end{remark}

\subsection{Taming the Liouville vector field}
\label{sec:taming}
Liouville manifolds may become easier to handle when an auxiliary function is present which behaves well with respect to the Liouville vector field. 
\begin{proposition} \label{prop:prestein} Let $(N,\theta)$ be a complete exact symplectic manifold. Assume that there exists a  proper bounded below smooth function $\psi:N\r \RR$ such  
that $L=\{m\in N\mid (Z_\theta\psi)(m)\le 0\}$ is compact.
Let $C>\max(\psi|L)$.
Then $M=\psi^{-1}(]-\infty,C])$ is a generating Liouville domain for $N$. In particular $N$ is a Liouville manifold.
\end{proposition}
\begin{proof} $M$ is a manifold with boundary $\partial M=\psi^{-1}(C)$ since $d\psi$ has no zeros on~$\partial M$ (a zero of $d\psi$ is in particular a zero of $Z_\theta\psi=(d\psi)(Z_\theta)$). Moreover $M$ is compact since $\psi$ is bounded below and proper.
Finally $M$ is a Liouville domain since $Z_\theta\psi>0$ on $\partial M$.

  Assume $M$ does not generate $N$ and let $m\in N\setminus \hat{M}$  be such that $\psi(m)$ is minimal (such $m$ exists
  since $\psi$ is bounded below and proper, and $N\setminus \hat{M}$ is closed, cf. Lemma \ref{lem:open} below). Since $N$ and $\hat{M}$ are $\phi_{Z_\theta,s}$-invariant, the same is true of $N\setminus \hat{M}$. Hence
  $\bigcup_{s\in \RR} \phi_{Z_\theta,s}(m)\cap \hat{M}=\emptyset$. Since $(Z_{\theta}\psi)(m)>0$ we can find a small $\epsilon<0$ such that $\psi(\phi_{Z_\theta,\epsilon}(m))<\psi(m)$, contradicting the choice of $m$.
\end{proof}
We have used the following result.
 \begin{lemma}
  \label{lem:open}
 Let $(N,\theta)$ be a complete exact symplectic manifold without boundary and let $M\subset N$ be a compact submanifold (with boundary) of codimension zero such that
 $(M,\theta{|}M)$ is a Liouville domain. There is a unique induced map $\gamma:\hat{M}\r N$ which extends $M\subset N$. This map is moreover an injective immersion
 (in particular its image is open). If $M$ is generating then $\gamma$ is an isomorphism.
\end{lemma}
\begin{proof} If $m\in \hat{M}$ then we should have
  $\gamma(m)=\phi^N_{Z_\theta,s}(\phi^{\hat{M}}_{Z_\theta,-s}(m))$ for
  $s$ such that $\phi^{\hat{M}}_{Z_\theta,-s}(m)\in M^\circ$. It is an
  easy verification that this is independent of $s$ and hence yields a
  well defined map $\hat{M}\r N$. The claim that the resulting
  $\gamma$ is an injective immersion follows from the fact that the
  inclusion map $M^\circ\subset N$ is tautologically an injective
  immersion.  The last claim is clear.
\end{proof}

We consider a particular case in which the hypotheses on $(Z_\theta,\psi)$ in Proposition \ref{prop:prestein} are satisfied.
\begin{lemma} \label{lem:particular}
Assume that $(N,g)$ is a Riemannian manifold and put $Z=\grad_g \psi$ for
a smooth function $\psi$ on $N$.
Then TFAE for $m\in N$:
\begin{enumerate}
\item $Z_m=0$,
\item $(Z\psi)(m)\le 0$,
\item $d\psi_m=0$.
\end{enumerate}
\end{lemma}
\begin{proof} This follows from the definition of gradient:
$
g(\grad_g \psi,-)=d\psi
$.
Thus $d\psi=g(Z,-)$ and $
Z\psi=g(Z,Z)
$ which yields the asserted equivalences.
\end{proof}
In order to prove that the Liouville vector field is complete one may use the following 
result.
\begin{theorem}[{\cite[Theorem 2.2]{MR2839561}}] 
\label{thm:complete}
Let $N$ be a manifold and let $Z$ be a smooth vector field on $N$. Assume that $\psi$ is a smooth proper function on $N$ such that there exist $A,B\ge 0$ such that 
\begin{equation}
\label{eq:completeness}
|Z\psi|\le A|\psi|+B.
\end{equation}
Then $Z$ is complete.
\end{theorem}
\begin{remark} \label{rem:complete} It clear that if we have \eqref{eq:completeness} outside a compact set then it holds everywhere possibly after replacing $B$ by a larger constant.
\end{remark}

\subsection{K\"ahler geometry}
\label{sec:Kahlerreview}
\subsubsection{General conventions}
\label{sec:general}
If $(N,J,\psi)$ is K\"ahler manifold (with global K\"ahler potential $\psi$) then one can make $N$ into an exact symplectic manifold by
\begin{equation}
\label{eq:kahlercase}
\begin{aligned}
\theta&=-d\psi \circ J\\
\omega &=d\theta\\
Z&=\grad_g\psi
\end{aligned}
\end{equation}
(we have dropped the $(-)_\theta$ decoration which we used in \S\ref{sec:liouville},\S\ref{sec:taming}) where $g(X,Y)$ is the Riemannian metric given by $\omega(X,JY)$.
We have $d=\partial+\bar{\partial}$ where $J$ acts on $\partial$, $\bar{\partial}$ with eigenvalues $i$, $-i$ respectively.
The formulas for $\theta,\omega$ become:
\begin{align*}
\theta&=-i (\partial-\bar{\partial})\psi,\\
\omega&=2i \partial \bar{\partial}\psi.\\
\end{align*}
\subsubsection{Local coordinates}
In local holomorphic  coordinates $(z_k)_k=(x_k+iy_k)_k$ we have $\partial=\sum_i \partial_{z_i}dz_i$, $\bar{\partial}=\sum_i \partial_{\bar{z}_i}d\bar{z}_i$ 
with $\partial_{z_k}=(1/2)(\partial_{x_k}-i\partial_{y_k})$, $\bar{\partial}_{z_k}=(1/2)(\partial_{x_k}+i\partial_{y_k})$ and $J(\partial_{x_k})=\partial_{y_k}$.
Usually the coordinates are clear
from the context and then we write $\partial_i=\partial_{z_i}$, $\bar{\partial}_i=\partial_{\bar{z}_i}$.
\subsection{Stein manifolds}
\label{sec:stein}
A Stein manifold is a K\"ahler manifold equipped with a global K\"ahler potential which is proper and bounded below.
Let $(N,J,\psi)$ be a Stein manifold.
We can make $N$ into an exact symplectic manifold as described in \S\ref{sec:general}.
However, without further conditions,
the Liouville vector field may not be complete.\footnote{A counterexample is given in \cite[Example 2.10]{CiEl}. It is the function $\psi(z):=\sqrt{1+|z|^2}$ on~$\CC$.} 
The following result describes a standard method for constructing Liouville manifolds.
\begin{proposition}
\label{prop:stein}
Let $(N,J,\psi)$ be a Stein manifold as above. Assume that the Liouville vector field is complete and that the zeros of $d\psi$ form a compact set. Then $(N,\theta)$ with $\theta$ as in \eqref{eq:kahlercase} is a Liouville manifold.
\end{proposition}
\begin{proof}
This follows from Proposition \ref{prop:prestein} and Lemma \ref{lem:particular}.
\end{proof}
\subsection{Fibers of Stein manifolds}
Let $(N,J,\psi)$ be a Stein manifold as in \S\ref{sec:stein} which is also a Liouville manifold (see e.g.\ Proposition \ref{prop:stein}) with $\theta$ as in \eqref{eq:kahlercase}.
Let $f:N\r \CC$ be a smooth function whose image contains $0$ and let $N_0=f^{-1}(0)$. Let $\psi_0, \theta_0$ be the restrictions of $\psi$, $\theta$ to $N_0$. Then we have
the following result which is a refinement of a result stated in \cite[p670]{Donaldson}.
\begin{proposition}
\label{prop:fiber} Assume that $(Z_\theta,\psi)$ satisfy \eqref{eq:completeness}.
Assume $\partial f\neq 0$ on $N_0$ and assume furthermore that there exists $0\le r<1$ such that $|\bar{\partial} f|\le r |\partial f|$ on $N_0$.  
Then $(N_0,\theta_0)$ is a complete exact symplectic manifold. Let $\tau\in [0,\pi/2]$ be the Hermitian angle  (see \S\ref{sec:hermitian}) between $Z_\theta$ and $\partial f$ where
we view $\tau$ as a function
on $N_0$.
If  $\{\sin \tau \le r\}$ is a compact subset of $N_0$ then $(N_0,\theta_0)$ is a Liouville manifold.
\end{proposition}
\begin{proof} We have $|df|\ge |\partial f|-|\bar{\partial}f|\ge (1-r)|\partial f|>0$ on $N_0$. So $N_0$ is smooth.
The fact that $N_0$ is moreover symplectic is stated in \cite[p670]{Donaldson}. It is a direct corollary of Lemmas \ref{lem:don1},\ref{lem:donaldson1} below (also due to Donaldson) with 
$V=T_mN$, $a=df_m$ for
$m\in N_0$. Let $Z_0$ be the Liouville vector field of $N_0$. Then $Z_0$ is the symplectic projection of $Z_\theta$ on $TN_0$.
By Proposition \ref{prop:projectionbound} below there exists a $\gamma$ such that $|Z_0|\le \gamma |Z_\theta|$. Since by assumption  $(Z_\theta,\psi)$ satisfy \eqref{eq:completeness},
the same is true for  $(Z_0,\psi)$. Hence $Z_0$ is complete by Theorem \ref{thm:complete}.

It remains to prove that $N_0$ is a Liouville manifold. By Proposition \ref{prop:prestein} we need to prove that $\{Z_0\psi\le 0\}$ is compact. By \eqref{eq:kahlercase} we have
$Z_0\psi=g(Z_\theta,Z_0)$. Then it follows from Proposition \ref{prop:anotherbound} below that $\{Z_0\psi\le 0\}\subset \{Z_\theta=0\}\cup \{\sin \tau \le r\}$. The compactness of 
$\{\sin \tau \le r\}$ follows from the hypotheses and the compactness of $\{Z_\theta=0\}$  follows from the assumption that $(N,\theta)$ is a Liouville manifold as then $\{Z_\theta=0\}$ is a closed subset contained in a generating Liouville domain of $N$ which is compact.
\end{proof}
\subsection{Families of Liouville manifolds}
Let $B$ be a manifold. A family of manifolds (with boundary) is a surjective submersion $M\r B$.
We will sometimes  informally write such a family as $(M_t)_{t\in B}$.

Let $N\r B$ be a family  of Liouville manifolds. We say that a proper submanifold (with boundary) $M\subset N$ is a \emph{family of generating Liouville domain} 
if for all $t\in B$,  $M_t$ is a generating Liouville domain for $N_t$. We say that $N\r B$ locally has families of generating Liouville domains if such families exist locally over $B$.

\medskip
We now state a version of Moser's lemma for families of Liouville manifolds.

\begin{theorem}[\S\ref{subsubsec:proof_Moser}] \label{lem:moser2}
  Let $(N_t,\theta_t)_{t\in B}$ be a family Liouville manifolds which locally has families of generating Liouville domains. Assume that $B$ is smoothly contractible and let $b\in B$. Then there exists a diffeomorphism
  $\psi: N_b\times B\r N$ over $B$, inducing the identity on $N_b$ as well as a smooth function $f$ on $N_b\times B$ with $f_b=0$, supported on some $L\subset N_b\times B$ with $L\r B$ is proper,
  such that for $t\in B$ we have
  $\psi_t^\ast(\theta_t)+df_t=\theta_b$.
\end{theorem}

\begin{remark}\label{rem:moser_generality}
This version of Moser's lemma is more general than \cite[Theorem 6.8]{CiEl} in the sense that we allow the underlying
manifold in the family to vary and we also allow the base manifold to be more general (in loc.\ cit.\ $B=[0,1]$). Needless to say that this added generality is in fact illusory. 
\end{remark}

\subsection{Families of fibers of Stein manifolds}
Let $(N,J,\psi)$ be a Stein manifold as in \S\ref{sec:stein} which is also a Liouville manifold (see e.g.\ Proposition \ref{prop:stein}) with $\theta$ as in \eqref{eq:kahlercase}.
Let $B$ be a compact base manifold (with boundary) and let $f:B\times N\r \CC$ be a smooth function such that for all $t\in B$ the image of $f_t$  contains $0$ and $f_t$ is a submersion.
Put $N_0=f^{-1}(0)$ and let $\pi:N_0\r B$ be the family which is the restriction of the projection $B\times N\r N$.
Let $\psi_0, \theta_0$ be the restrictions to $N_0$ of pullbacks of $\psi$, $\theta$ via the projection $B\times N\r N$. 
Then the following result is a family version of Proposition \ref{prop:fiber}.
\begin{proposition}
\label{prop:fiber_families} Assume that $(Z_\theta,\psi)$ satisfy \eqref{eq:completeness}.
Assume that $\partial f\neq 0$ and there exists $0\le r<1$ such that $|\bar{\partial} f|\le r |\partial f|$ on the fibers of $\pi:N_0\r B$.
Then the fibers of $\pi$ are complete exact symplectic manifolds. Let $\tau\in [0,\pi/2]$ be the Hermitian angle  (see \S\ref{sec:hermitian}) between $Z_\theta$ and $\partial f$
(the latter computed fiberwise).
If the restriction of $\pi$ to  $\{\sin \tau \le r\}$ is proper then $\pi$ is a family of Liouville manifolds with a family of generating Liouville domains. 
\end{proposition}
\begin{proof} This is proved in exactly the same way as Proposition \ref{prop:fiber}.
\end{proof}

\subsection{Graded symplectic manifolds}  
\subsubsection{Gradings via the symplectic Lagrangian} For a symplectic vector space $V$ of real dimension $2r$ we let $\LGr(V)$ be the
  Grassmannian of Lagrangian subspaces of $V$. The Maslov index provides an isomorphism 
\begin{equation}
\label{eq:maslov}
\pi_1(\LGr(V))\cong \ZZ.
\end{equation}
Let $\widetilde{\LGr}(V)$ be the universal cover of $ \LGr(V)$. Since the universal cover is a $\ZZ$-torsor it defines
an element $u_V\in H^1(\LGr(V),\ZZ)$.
\begin{remark} \label{rem:pi1iso}
The isomorphism \eqref{eq:maslov} and the corresponding universal cover can be understood very concretely. For a vector space $W$ let us write $W_0$ for the non-zero elements.
We  equip $V$ with a compatible complex structure (unique up to homotopy by Proposition \ref{prop:compat} below). Then we may define the ``phase'' map
\[
\LGr(V)\xrightarrow{P\mapsto \wedge^r_{\RR} P} (\wedge_{\CC}^r V)_0/(
\RR-{0})\cong (\wedge_{\CC}^r V)_0^{\otimes 2}/\RR_{>0}\cong S^1
\]
which  induces an isomorphism on $\pi_1$. In particular $\widetilde{\LGr}(V)$ is the pullback of the universal cover $\RR\r S^1$.
\end{remark}
Let $E$ be a symplectic vector bundle on a manifold $M$
  and let $\pi:\LGr(E)\r M$ be the bundle of Lagrangian subspaces of
  $E$.
A \emph{grading} on $E$ is a $\ZZ$-torsor 
\[
\widetilde{\LGr}(E)\r \LGr(E)
\]
whose fiber in $m\in M$ is $\widetilde{\LGr}(E_m)\r \LGr(E_m)$. (Hence $\LGr(E)$ is a fiberwise universal cover of $\LGr(E)$.) In other words a grading
is given by an element $u_E\in H^1(\LGr(E),\ZZ)$ whose restriction in $m\in M$ is given by $u_{E_m}$.
\begin{lemma}\cite[Lemma 2.2]{MR1765826}
\label{lem:ob}
The obstruction against the existence of a grading on a symplectic vector bundle on $M$ lives in $H^2(M,\ZZ)$. Two different
gradings differ by an element of $H^1(M,\ZZ)$. Finally the automorphisms of a grading
are given by $H^0(M,\ZZ)$.
\end{lemma}
The last claim is not in loc.\ cit.\ but it is clear from the fact that a grading is a $\ZZ$-torsor.
\begin{definition} A \emph{graded symplectic manifold} is a symplectic manifold, equipped with a grading on $TM$. We write $u_M$ for the corresponding $u_{TM}$.
\end{definition}

A graded Lagrangian on $M$ is a commutative diagram
\begin{equation}
\label{eq:grlang}
\begin{tikzcd}
&\widetilde{\LGr}(TM)\ar[d]\\
& \LGr(TM)\ar[d,"\pi"]\\
L\ar[ru,"\tilde{\imath}"]\ar[ruu]\ar[r,hook,"i"] & M
\end{tikzcd}
\end{equation}
\begin{lemma}\cite[Lemma 2.3]{MR1765826} 
The obstruction against the existence of a grading on a Lagrangian $L$ is in $H^1(L,\ZZ)$. 
The difference between two gradings can be measured by an element of $H^0(L,\ZZ)$.
\end{lemma}
\begin{proof} The obstruction is given by $\tilde{\imath}(u_M)$. The fact that two gradings differ by an element of $H^0(L,\ZZ)$ is clear.
\end{proof}
\subsubsection{The category of graded symplectic manifolds }
If $\widetilde{\LGr}(TM)\r \LGr(TM)$, $\widetilde{\LGr}(TN)\r \LGr(TN)$ define graded symplectic manifolds then a graded symplectomorphism is a 
commutative diagram
\[
\begin{tikzcd}
\widetilde{\LGr}(TM) \ar[d]\ar[r,"\widetilde{\LGr}(df)"] & \widetilde{\LGr}(TN)\ar[d]\\
{\LGr}(TM) \ar[d,"\pi"']\ar[r, "\LGr(df)"] & {\LGr}(TN)\ar[d,"\pi"]\\
M\ar[r,"f"'] &N
\end{tikzcd}
\]
where $f$ is a symplectomorphism and $\LGr(df)$ is the map which sends a Lagrangian subspace $L\subset T_m M$ to the Lagrangian
subspace $(df)(L)\subset T_{f(m)}N$.
\begin{lemma}
The obstruction to lifting a symplectomorphism $f:M\r N$ to a graded symplectomorphism lives in $H^1(M,\ZZ)$. The difference between two lifts determines an element of  $H^0(M,\ZZ)$.
\end{lemma}
\begin{proof} The grading on $M$ and the pullback of the grading on $N$ via $f$ define two different gradings on them. Their difference
defines an element of $H^1(M,\ZZ)$ via Lemma \ref{lem:ob}. Hence this gives the obstruction. If the obstruction vanishes then two
different lifts give an automorphism of the grading on $M$, hence an element of $H^0(M,\ZZ)$, again using Lemma \ref{lem:ob}.
\end{proof}
\subsubsection{Construction of graded symplectic isotopies}
We show that an  isotopy of symplectic manifolds, e.g.\ constructed using Theorem \ref{lem:moser2}, can be uniquely lifted to a graded isotopy. More precisely: 
\begin{theorem} \label{lem:simpiso} Let $(N_t,\omega_t)_{t\in B}$ be a family of graded symplectic manifolds with $B$ contractible.
Let $b\in B$ and assume that we have a family of symplectomorphisms $\phi_t:N_b\r N_t$ such that $\phi_b=\Id$. Then this can be lifted in a unique way to a family of graded symplectomorphisms,
such that the graded lift of $\phi_b$ is the obvious graded lift of the identity.
\end{theorem}
\begin{proof} Note first that Lemma \ref{lem:ob} has a family version  for a family of symplectic vector bundles $E$. In that case $\LGr(E)$ stands for the bundles of Lagrangians in $E$ living in 
the fibers of the family. The proof of Lemma \ref{lem:ob} remains the same.
So the obstruction to lifting $\phi:N_b\times B\r N$ to a family of graded symplectic manifolds is in $H^1(N_b\times B,\ZZ)=H^1(N_b,\ZZ)$ (the equality holds because $B$ is contractible). Since the obstruction is zero when specialized to $t=b$
we obtain that it is zero.  Again by a family version of Lemma \ref{lem:ob} two graded lifts differ by an element $H^0(N_b\times B,\ZZ)=H^0(N_b,\ZZ)$. Hence two graded lifts coincide
when they coincide for $t=b$. Since we fix the graded lift for $t=b$, this proves uniqueness.
\end{proof}
\subsection{Construction of gradings}
\subsubsection{Gradings via $S^1$-bundles}
\label{sec:grp1}
We recall the following. 
\begin{proposition}[{\cite[Proposition 2.6.4]{MR3674984}}] \label{prop:compat}
If $E$ is a vector bundle on a manifold $M$ of rank $2r$ then sending a symplectic structure on $E$  to a compatible complex structure (after fixing an inner product on $E$) 
defines a homotopy equivalence between the symplectic structures and complex structures on $E$.
\end{proposition}
This proposition makes it possible to assign Chern classes $c_i(E)$ in $H^{2i}(M,\ZZ)$ to a symplectic vector bundle $E$ on $M$.

If $E$ is a symplectic vector bundle on a manifold of rank $2r$ then after equipping with a compatible complex structure
$E$ becomes a complex vector bundle.
For a line bundle $L$ let $L_0$ be the complement of the zero section. We obtain a morphism
\[
\LGr(E)\xrightarrow{P\in \LGr(E_m)\mapsto \wedge^r_{\RR} P} (\wedge_{\CC}^r E)_0/(\underline{
\RR}-{0})\cong (\wedge_{\CC}^r E)_0^{\otimes 2}/\underline{\RR}_{>0}.
\]
Here $\nu(E):=(\wedge_{\CC}^r E)_0^{\otimes 2}/\underline{\RR}_{>0}$ is an $S^1$-bundle. From Remark \ref{rem:pi1iso} we obtain that a grading on $E$ is, up to homotopy, the same as giving a section
of $\nu(E)$ or an everywhere non-zero section of  $(\wedge_{\CC}^r E)^{\otimes 2}$ (since $(\underline{\RR}_{>0},\times)\cong (\underline{\RR},+)$ has
vanishing cohomology). This leads to the following observation:
\begin{lemma} \label{lem:S1grading}
If $E\r M$ is a complex (or equivalently symplectic) vector bundle of rank $2r$ on a manifold $M$ then a grading on $E$ is the same as giving an everywhere non-zero section of $(\wedge^r_{\CC}E)^{\otimes 2}$ up to homotopy.
\end{lemma}
The obstruction against a complex line bundle $L$ to have a non-zero section is $c_1(L)\in H^2(M,\ZZ)$. Hence:
\begin{lemma}\cite[Lemma 2.2]{MR1765826} The obstruction against the existence of a grading on $E$ is $2c_1(E)$.
\end{lemma}
\subsubsection{Inducing gradings on submanifolds}
\label{sec:inducing}
Let $(F,\omega)$ be a symplectic submanifold of $(M,\omega)$.  Then as symplectic vector bundles we have
\[
TM{\mid} F=TF\oplus N_{F/M}
\] 
so that after choosing complex structures we have
\[
(\wedge^{d_M}_\CC TM{\mid}F)=\wedge^{d_F}_\CC TF\otimes_F \wedge^{d_M-d_F}_\CC N_{F/M}.
\]
So to restrict a grading on $M$ to  $F$ we have to give a grading on $N_{F/M}$. 
\begin{remark}
\label{rem:cheapgrading}
In the special case that $\rk N_{F/M}=2$ then below we will often specify a grading on $N_{F/M}$ by
giving an everywhere non-zero section of $N_{F/M}$. The actual grading is then given by the line bundle generated by this section.
\end{remark}
\subsubsection{Gradings on fibers}
\label{sec:fibergrading}
Let $M$ be a symplectic manifold and let $f:M\r \CC$ be a function which is a surjective submersion on a neighbourhood of $F:=f^{-1}(0)$ such that $F$ is  in addition symplectic. 
Then $N_{F/M}=f^\ast T_0(\CC)$ (as real vector bundles). Hence we can construct a non-vanishing section $\xi$ of $N_{F/M}$ 
by pulling back  $1\in \CC\cong T_0(\CC)$ (pulling back another non-zero element gives a  homotopic choice).
\begin{definition} \label{def:lemfibers}
If $M$ is graded then 
the restriction of the grading on $M$ to $f^{-1}(0)$ is the one obtained  from $\xi$ via Remark \ref{rem:cheapgrading}.
\end{definition}
Assume $X$ is a nowhere vanishing section of $N_{f^{-1}(0)/M}$. Then it defines another grading on $f^{-1}(0)$. So  according to Lemma \ref{lem:ob} the difference with the restricted grading should be an element of
$H^1(f^{-1}(0),\ZZ)$. 
One may verify the following lemma.
\begin{lemma} \label{lem:arglemma}
Define $\phi:f^{-1}(0)\to S^1$ by setting  to $X$ and the restricted grading is given by
$2[\phi]\in \Hom_{\operatorname{Grp}}(\pi_1(f^{-1}(0)), \pi_1(S^1)) =H^1(f^{-1}(0),\ZZ)$.
\end{lemma}
\begin{remark} Note that the difference between the two gradings does not depend on the grading on $M$.
\end{remark}
\subsubsection{Restriction of gradings in Liouville pairs}
\label{sec:restriction}
Recall that a \emph{Liouville pair} is a pair $(M,F)$  where $M,F$ are Liouville domains  
 and $F$ is embedded in $\partial M$ as a hypersurface
such that the Liouville form $\theta$ on $M$ restricts to the Liouville form  on~$F$. As usual we put $\omega=d\theta$.

Let $(M,F)$ be a Liouville pair and let $Z$ be the Liouville vector field on $\hat{M}$.
We have
\begin{equation}
\label{eq:border}
T(\hat{M}){|}\partial M=T(\partial M) \oplus \underline{\RR} (Z|\partial M).
\end{equation}
\begin{definition} \label{def:liouvillepair}
If $\hat{M}$ (or equivalently $M$) is equipped with a grading then
the restricted grading on $F$ is the one obtained  from the non-zero section $Z{|}F$ of $N_{F/\hat{M}}$ via Remark \ref{rem:cheapgrading}.
\end{definition}
For use in \S\ref{sec:inducing2} below we will elaborate a bit on the geometry of $(M,F)$.
Let $R$ be the Reeb vector field
 on $\partial 
M$.  It is characterized by $\theta(R)=1$ and 
$i_R(\omega|\partial M)=0$.  We denote the restrictions of $R,Z$ to $F$ by $R_F$ and $Z_F$ respectively.
We have
\begin{equation}
\label{eq:red}
T(\partial M){|}F=TF\oplus \underline{\RR} R_F
\end{equation}
 as an element in $TF\cap \RR R_F$ would be in $\ker (\omega |F)=0$.
It follows
\begin{equation}
\label{eq:proportional}
T(\hat{M}){|}F\overset{\eqref{eq:border}}{=}T(\partial M){|}F \oplus \underline{\RR} Z_F\overset{\eqref{eq:red}}{=}TF\oplus \underline{\RR} R_F \oplus \underline{\RR} Z_F.
\end{equation}

Note that \eqref{eq:proportional} is however not a symplectic orthogonal decomposition. Let
$Z'_F$ be the projection on the symplectic orthogonal to $TF$. So
\[
Z_F=Z'_F+Y
\]
where $Y$ is a section of $TF$ and $\omega(Z'_F,Y)=0$.
Then we obtain a symplectic orthogonal decomposition
\begin{equation}
\label{eq:proportional2}
T(\hat{M}){|}F=TF\oplus \underline{\RR} R_F \oplus \underline{\RR} Z_F'=TF\oplus N_{F/\hat{M}}
\end{equation}
and one verifies $\omega(Z'_F,R_F)=1$.
 
\subsubsection{Gradings from polarizations}
\label{sec:polar}
\begin{definition}\label{def:pol} Let $M$ be a manifold and let $E$ be a symplectic vector bundle on $M$. A \emph{polarization} of $E$ is a Lagrangian subbundle $L$ of $E$.
If $M$ is a symplectic manifold then a polarization of $M$ is a polarization of $TM$.
\end{definition}
\begin{remark} \label{rem:pol}
Giving a polarization on $E$ is the same up to contractible choices
 as giving a real vector bundle $L$ such that $L\otimes_\RR \CC\cong E$ where $E$ is equipped with its unique (up to a contractible choice) compatible complex structure.
Below we make no distinction between these two concepts of polarization. Note that by this observation it makes sense to talk about a polarization of an almost complex manifold which is
not necessarily equipped with a symplectic form.
\end{remark}
If $L$ is a polarization on a complex vector bundle $E$ or rank $2r$ then $\wedge^r L$ defines a rank one real sub-bundle of $\wedge^r_\CC E$.
whose square defines a grading on $E$ as in Lemma \ref{lem:S1grading}.
\begin{example}
\label{ex:canonical}
If $N$ is a manifold then the Lagrangian foliation of $T^\ast N$ given by the fibers of $T^\ast N\r N$ defines a polarization of $T^\ast N$. We call this the corresponding grading the \emph{canonical grading}
on $T^\ast N$.
\end{example}
\subsubsection{Inducing polarizations in Liouville pairs}
\label{sec:inducing2}
We place ourselves in the setting of \S\ref{sec:restriction}. So
$(M,F)$ is a Liouville pair. Assume $L\subset T\hat{M}$ is
polarization of $\hat{M}$ which contains $Z$. It follows that
$\theta(L)=0$ (since for $l\in L$ we have
$\theta(l)=\omega(Z,l)=0$, as $L$ is Lagrangian).  Put
$L_0=(L|\partial M)\cap T(\partial M)$. We have $\theta(L_0)=0$ which is the
only property of $L_0$ we will use.

We claim that the projection 
\begin{equation}
\label{eq:projection}
T(\partial M){|}F\overset{\eqref{eq:red}}{=}TF\oplus \underline{\RR} R_F\r TF
\end{equation}
is injective on $L_0|F$. Indeed if $l\in L_0\cap \RR R_F$ then applying $\theta$ we find $l=0$.  We furthermore claim that \eqref{eq:projection} sends
 $L_0|F$ to a Lagrangian subbundle on $TF$. Indeed if $l_i'=l_i+\lambda_i R_F$ for $i=1,2$ are in $L_0|F$ with $(l_i)_i\in TF$ and $(\lambda_i)_i\in \RR$ then
$\omega(l_1,l_2)=\omega(l'_1,l'_2)=0$.
\begin{definition} \label{def:restrictpol} Let $L'_0$ be the image of $L_0$ under \eqref{eq:projection}. We say that $L'_0$ is the polarization on $F$ induced from the polarization $L$ on $\hat{M}$.
\end{definition}
\begin{proposition} \label{prop:twogradingssame} 
The grading of $F$ corresponding to $L'_0$ (see \S\ref{sec:polar}) is the restriction (see Definition \ref{def:liouvillepair}) of the grading on $\hat{M}$ corresponding to $L$.
\end{proposition}
\begin{proof} Let $2r=\dim M$.  Note that \eqref{eq:border} restricts to a decomposition
\begin{equation}
\label{eq:restricteddecomp}
 L|F= L_0{|}F\oplus \underline{\RR} Z_F.
\end{equation}
First consider the inclusion
\[
  i:L{|} F\overset{\eqref{eq:restricteddecomp}}{=} L_0{|}F\oplus \underline{\RR} Z_F \hookrightarrow T(\partial M)\oplus \underline{\RR} Z_F
\overset{\eqref{eq:red}}{=} TF\oplus \underline{\RR} R_F \oplus \underline{\RR} Z_F' \overset{\eqref{eq:proportional2}}{=}T\hat{M}{|}F\,.
\]
We choose compatible complex structures on the symplectic vector bundles $TF$ and $N_{F/T^*N}= \underline{\RR} R_F \oplus \underline{\RR} Z_F'$.
For $N_{F/T^\ast M}$ we simply put $JZ'_F=R_F$. Let  $i_0$ be the composition $L_0|F\r TF\oplus \underline{\RR} R_F\r TF$ and as in Definition \ref{def:restrictpol}  put $L'_0:=i_0(L_0)\subset TF$. We have
\begin{equation}
\label{eq:wedgepol}
\wedge_\CC^{r-1} TF=\wedge^{r-1}_\RR L'_0\oplus J(\wedge^{r-1}_\RR L'_0).
\end{equation}
Choose (locally) a basis for $L_0$: $e_1,\ldots,e_{r-1}$. Then  by \eqref{eq:red} we have $i(e_j)=e'_j+\alpha_j R_F$ for $e'_j$ sections of $TF$ and $\alpha_j$ functions on $F$.
We have $i_0(e_j)=e'_j$.
We obtain
\begin{align*}
(\wedge^r &i)(e_1\wedge_{\RR} \cdots \wedge_{\RR} e_{r-1}\wedge Z_F)\\
&=(e'_1+\alpha_1 R_F)\wedge_{\CC} \cdots \wedge_{\CC} (e'_1+\alpha_1 R_F) \wedge_{\CC} (Z'_F+Y)\\
&=e'_1\wedge_\CC \cdots \wedge_\CC e'_{r-1} \wedge_\CC Z'_F+\sum_{j=1}^{r-1} \pm \alpha_j e'_1\wedge_\CC \cdots \wedge_\CC \widehat{e'}_j \wedge_\CC \cdots e'_{r-1} \wedge_\CC Y \wedge_\CC R_F \\
&=(A+JB)\wedge Z'_F
\end{align*}
for sections $A$, $B$ of $\wedge^{r-1} L'_0$. Note that in the third line we have used $R_F\wedge_\CC Z'_F=0$ since $R_F$ and $Z'_F$ are proportional over $\CC$.

The restricted grading on $F$ corresponds to the section $A+JB$ of  $(TF)_0/\underline{\RR}_0$ whereas the grading obtained from $L_0$ corresponds to $A$.
But $t\mapsto A+tJB$ for $t\in [0,1]$ defines a homotopy between these sections. Note that $A+tJB\neq 0$ because $A\neq 0$ ($L'_0$ is a polarization) and \eqref{eq:wedgepol}.
This finishes the proof.
\end{proof}

\subsection{The background class}
\subsubsection{Orientation data}
In order for the Fukaya category to be $\ZZ$-graded we need to consider Lagrangians which are  equipped with \emph{orientation data} (see \cite[\S5.3]{GPS3}). 
Similarly to the fact that graded Lagrangians only make sense for graded symplectic manifolds, orientation data depends on a so-called \emph{background class} on the symplectic manifold.
More precisely, let $M$ be a symplectic manifold and let $\Lscr\r \LGr(TM)$ be the universal (vector)-bundle of Lagrangian subspaces of the tangent spaces and let $w_2(\Lscr)$ be its second 
Stieffel-Whitney class (recall that $w_2(\Lscr)$ is the obstruction against $\Lscr$ being spin). A \emph{background class} $\eta$ is an element of $H^2(M,\ZZ/2\ZZ)$. Then $\pi^\ast(\eta)+w_2(\Lscr)\in H^2(\LGr(TM),\ZZ/2\ZZ)$ defines a $\RR\PP^\infty$ bundle $\LGr(M)^\#\r \LGr(M)$
(as $\RR \PP^\infty=B(\ZZ/2\ZZ)$). Orientation data on a Lagrangian $i:L\hookrightarrow M$
consists of a commutative diagram similar to \eqref{eq:grlang}:
\[
\begin{tikzcd}
&\LGr(TM)^\#\ar[d]\\
& \LGr(TM)\ar[d,"\pi"]\\
L\ar[ru,"\tilde{\imath}"]\ar[ruu]\ar[r,hook,"i"] & M
\end{tikzcd}
\]
If the background class is trivial then  orientation data on $L$ consists of a spin structure on $TL$.
In general orientation data is a twisted spin structure with obstruction class given by $-i^\ast(\eta)$.
We have an analogue of Theorem \ref{lem:simpiso}.
\begin{theorem} \label{lem:bgisotopies}
Let $(N_t,\omega_t,\eta_t)_{t\in B}$ be a family of symplectic manifolds with background class, with $B$ connected.
Let $b\in B$ and assume that we have a family of symplectomorphisms $\phi_t:N_b\r N_t$ such that $\phi_b=\Id$. Then $\phi_t^\ast(\eta_t)=\eta_0$.
\end{theorem}
\begin{proof}
This follows from the fact that since $B$ is connected the maps $N_b\xrightarrow{\phi_b=\Id} N_b\hookrightarrow N$ and $N_b\xrightarrow{\phi_t} N_t\hookrightarrow N$ are homotopic.
\end{proof}
\subsubsection{The background class associated to a polarization}
Let $M$ be a manifold equipped with a symplectic vector bundle $E$. Let $L\subset E$
 be a polarization as in Definition \eqref{def:pol}. Then the background class associated to $L$ is simply $w_2(L)\in H^2(M,\ZZ/2\ZZ)$.

If $M$ is symplectic and if $E=TM$ is equipped with a  polarization then this yields
a background class on $M$ as explained in \cite[\S5.3]{GPS3}.
\subsubsection{The background class of a cotangent bundle}
\label{sec:backgroundcotangent}
If $N$ is a manifold then the Lagrangian foliation of $T^\ast N$  as in Example \ref{ex:canonical} defines a polarization on $T^\ast N$. Writing $\pi:T^\ast N\r N$ for the
projection, this polarization is given by $\pi^\ast(T^\ast N)$. So the background class is equal to $\pi^\ast(w_2(T^\ast N))$. 
 In particular if $N$ has free tangent bundle then
the background class is trivial.
\subsection{The wrapped Fukaya category of a Liouville manifold}
\label{sec:wrapped}
Let $N=(N,\theta)$ 
be a Liouville manifold.
The wrapped Fukaya category $\Wscr\Fscr(N)$ is an $A_\infty$-category  $\Wscr\Fscr(N)$ with homological units which is appropriately functorial 
 (see \cite{GPS1}). We quote some results below from \cite{GPS1, TannakaOh}. For convenience we state them in the most open setting but as explained in \cite{TannakaOh} they remain
valid in the presence of additional data such as a grading and/our background class. If both the grading and the background class are present then $\Wscr\Fscr(N)$ is a $\ZZ$-graded,
$\ZZ$-linear $A_\infty$-category (see \cite[\S5.3]{GPS3}).

\begin{proposition}\cite[\S3.7, Lemma 3.41]{GPS1} \label{cor:functorial}
  Let $(\phi,f):(N_1,\theta_1)\r (N_2,\theta_2)$ be an isomorphism (see \S\ref{sec:liouville}) of Liouville manifolds. Then there is a strict $A_\infty$-quasi-equivalence
  $\Wscr\Fscr(\phi,f):\Wscr\Fscr(N_1,\theta_1)\r \Wscr\Fscr(N_2,\theta_2)$.
 Moreover $\Wscr\Fscr(\phi,f)$ is  compatible with compositions and if  $N_1=N_2$ then $\Wscr\Fscr(\Id,0)$ is the identity.
\end{proposition}
Write $\Wscr\Fscr(\phi,f)$ for the morphism of $A_\infty$-categories $\Wscr\Fscr(N_1,\theta_1)\r \Wscr\Fscr(N_2,\theta_2)$ described in  Proposition \ref{cor:functorial}.
We recall the following.
\begin{proposition}\cite[Theorem 1.0.3]{TannakaOh} \label{cor:isotopy}
  Let $(\phi_t,f_t):N\r N$, $t\in [0,1]$, be an arbitrary isotopy of Liouville manifolds. Then
  $\Wscr\Fscr(\phi_0,f_0)$ and $\Wscr\Fscr(\phi_1,f_1)$ are $A_\infty$-homotopic.
\end{proposition}

\begin{corollary}
  \label{cor:independence1}
  Let $(N_t,\theta_t)$, $t\in [0,1]$, be a family of Liouville manifolds.
  Let $(\phi_{i,t},f_{i,t}):(N_0,\theta_0)\r (N_t,\theta_t)$ for
  $i=0,1$ be two isotopies which are the identity for $t=0$. Then $\Wscr\Fscr(\phi_{0,1},f_{0,1})$ and $\Wscr\Fscr(\phi_{1,1}, f_{1,1})$ (which are $A_\infty$-functors from $\Wscr\Fscr(N_0,\theta_0)$ to
  $\Wscr\Fscr(N_1,\theta_1)$)
  are $A_\infty$-homotopic.
\end{corollary}
\begin{proof}
  This follows from Proposition \ref{cor:isotopy} by considering the isotopy
  $(\phi_{1,t},f_{1,t})^{-1}\circ (\phi_{0,t},f_{0,t}):(N_0,\theta_0)\r (N_0,\theta_0)$
  for $t\in [0,1]$.
\end{proof}

Moreover, two homotopic isotopies which share  starting and ending points induce $A_\infty$-homotopic equivalences. 
\begin{corollary}
  \label{cor:independence2}
  Let $(N_{u,t},\theta_{u,t})$, $u,t\in [0,1]$, be a family of Liouville manifolds
  such that $(N_{u,0},\theta_{u,0})=(N_{0,0},\theta_{0,0}):=(N_0,\theta_0)$,
  $(N_{u,1},\theta_{u,1})=(N_{0,1},\theta_{0,1}):=(N_1,\theta_1)$ for all $u$.

  For $u=0,1$ assume that $(\phi_{u,t},f_{u,t}):(N_0,\theta_0)\r (N_{u,t},\theta_{u,t})$ are isotopies which are the identity for $t=0$. 
Then $\Wscr\Fscr(\phi_{0,1},f_{0,1})$ and $\Wscr\Fscr(\phi_{1,1}, f_{1,1})$ (which are $A_\infty$-functors, $\Wscr\Fscr(N_0,\theta_0)
  \r \Wscr\Fscr(N_1,\theta_1)$)
  are $A_\infty$-homotopic.

\end{corollary}
\begin{proof}
The difference with Corollary \ref{cor:independence1} is that the two isotopies $(\phi_{0,t},f_{0,t})$ and $(\phi_{1,t},f_{1,t})$ may pass through different families of Liouville manifolds 
(but they still have the same begin and end point). So we have to connect these two families. This is achieved by introducing the $u$-coordinate.

  We let $B=[0,1]_u\times [0,1]_t$, $b=(0,0)$ and we construct a trivializing family $(\phi'_{u,t},f'_{u,t}):(N_0,\theta_0)\r (N_{u,t},\theta_{u,t})$ as in
  Theorem \ref{lem:moser2}. Replacing $(\phi'_{u,t},f'_{u,t})$ by $(\phi'_{u,t},f'_{u,t})\circ (\phi'_{u,0},f'_{u,0})^{-1}$ we may, and we will, assume that
  $(\phi'_{u,0},f'_{u,0})=(\Id_{N_0},0)$ for $u\in [0,1]$.

  Writing $\sim$ for $A_\infty$-homotopy equivalence,
 Corollary \ref{cor:independence1} yields $\Wscr\Fscr(\phi'_{u,1},f'_{u,1})\sim \Wscr\Fscr(\phi_{u,1},f_{u,1})$ for $u=0,1$.
 Hence it is sufficient to prove $\Wscr\Fscr(\phi'_{0,1},f'_{0,1})\sim \Wscr\Fscr(\phi'_{1,1},f'_{1,1})$. To this end we consider the isotopy $(\phi_{u,1},f_{u,1})^{-1}\circ (\phi_{0,1},f_{0,1}):(N_0,\theta_0)\r (N_0,\theta_0)$ for $u\in [0,1]$ and we invoke Proposition 
  \ref{cor:isotopy}.
\end{proof}

\section{Symplectic geometry in the toric case}
\label{sec:zhou} 
\subsection{Kähler structure}
\label{subsec:some_formulas}
We refer to \S\ref{sec:Kahlerreview} for a quick review of K\"ahler geometry. Here we consider K\"ahler structures on tori.
For $(z_j)_j\in (\CC^\ast)^k$ it will be convenient to use logarithmic polar coordinates $w_j=(\rho_j+i\eta_j)_j$ for $\rho_j\in \RR$, $\eta_j\in \RR/2\pi\ZZ$ such that $z_j=e^{w_j}$. We have $\partial=\sum_i \partial_{w_i}dw_i$, $\bar{\partial}=\sum_i \partial_{\bar{w}_i}d\bar{w}_i$ 
with $\partial_{w_k}=(1/2)(\partial_{\rho_k}-i\partial_{\eta_k})$, $\bar{\partial}_{w_k}=(1/2)(\partial_{\rho_k}+i\partial_{\eta_k})$ and $J(\partial_{\rho_k})=\partial_{\eta_k}$.
For a K\"ahler potential $\psi$ the associated symplectic and Liouville forms are given by (see \S\ref{sec:Kahlerreview}):
\begin{align*}
\omega&=2i \partial \bar{\partial}\psi=2i\sum_{jk} (\partial_{w_j}\bar{\partial}_{w_k} \psi) dw_jd\bar{w}_k,\\
\theta&=-i (\partial-\bar{\partial})\psi=-i\sum_{j} ((\partial_{w_j} \psi) dw_j-(\bar{\partial}_{w_j} \psi) d\bar{w}_j).
\end{align*}
We consider the special case that $\psi$ only depends on $(\rho_i)_i$. In that case using the formulas for $\partial_{w_j}$ we get:
\begin{equation}
\label{eq:omega}
\begin{aligned}
\omega
      &=\sum_{jk} (\partial_{\rho_j}\partial_{\rho_k} \psi) d\rho_j d\eta_k, \\
\theta
&=\sum_j (\partial_{\rho_j} \psi) d\eta_j.
\end{aligned}
\end{equation}
Thus the metric associated to $\omega$ is given by:
\begin{equation}
\label{eq:zhoumetric}
g=\sum_{jk} (\partial_{\rho_j}\partial_{\rho_k} \psi) (d\rho_j d\rho_k+d\eta_j d\eta_k).
\end{equation}
\subsection{The Legendre transform}
\label{sec:legendre0}
Let us now work coordinate free and put $N=\ZZ^k$. From this we obtain
identifications
$T(S_{N^\ast})=S_{N^\ast}\times N^\ast_\RR=N^\ast_{(\RR/\ZZ)\times
  \RR}=N^\ast_{\CC^\ast}=(\CC^\ast)^k$ where $S_{N^\ast}$ is as in
\S\ref{sec:fltz}. Recall from \S\ref{sec:fltz} that we likewise have
$T^\ast(S_{N^\ast})=S_{N^\ast}\times N_\RR$.  We view $\psi$ as a function on $N^\ast_\RR$. As in \cite[\S2.1]{Zhou}
we define the \emph{Legendre transform} as
$\Phi_\psi:N_\RR^\ast\r N_\RR:(\rho_i)_i\mapsto (\partial_{\rho_i}
\psi)_i$ (this is obviously invariant under linear coordinate
changes).  As explained in \cite{legendre}  $\Phi_\psi$ is an
injective immersion.\footnote{It does not have to be a
  diffeomorphism. E.g.\ in dimension $1$ we may consider a smooth function $\psi:\RR\r \RR$ such that $\psi''(\rho)>0$
  which is nearly linear far from~$0$. In that case
  $\rho \r\psi'(\rho)$ will have bounded codomain. A concrete example is given by $\int_0^x\int_0^y\exp(-w^2)dw dy$.}

Put $p_j=\partial_{\rho_j} \psi$.  Then from \cite[\S2.1,2.2]{Zhou} we get a corresponding injective immersion  
\begin{equation}
\label{eq:legendre}
\id\times\Phi_\psi:T(S_{N^\ast})\r T^\ast(S_{N^\ast}):(\eta_i,\rho_i)_i\mapsto (\eta_i,p_i)_i.
\end{equation}
If we put 
\begin{equation*}
\theta_{st}=\sum_j p_j d\eta_j\qquad 
\omega_{st}=\sum_j dp_j d\eta_j\qquad
Z_{st}=\sum_j p_j \partial_{p_j}
\end{equation*}
then we get
\begin{equation}
\label{eq:legendre10}
\theta=(\Id\times\Phi_\psi)^\ast(\theta_{st})\qquad
\omega=(\Id\times \Phi_\psi)^\ast(\omega_{st})\qquad
Z=(\Id\times\Phi_\psi)^\ast(Z_{st}).
\end{equation}
\subsection{Strictly convex functions}
For use below we discuss some properties of K\"ahler potentials which only depend on $\rho$.
Assume $\psi:\RR^k\r \RR$  is proper, bounded below such that the matrix $\partial_{\rho_i}\partial_{\rho_j}\psi$ is positive definite in every point of $\RR^k$.
\begin{lemma} \label{lem:extremal}
$\psi$ has a unique extremal point $\rho_L$ on every affine subspace $L$
of $\RR^k$. This extremal point is a non-degenerate minimum. Moreover the function $L\mapsto \rho_L$ is continuous on the set of affine spaces of dimension $l$ when we view the latter
set as an open subset of $\Grass(l+1,\RR^{k+1})$ (via $L\mapsto \spantwo(0,L\times\{1\})$) 
and equip it with the induced topology.
\end{lemma}
\begin{proof}
Let $\psi_L=\psi\mid L$. The fact that $\partial_{\rho_i}\partial_{\rho_j}\psi_L$ is positive definite implies that all extrema (on $L$) are non-degenerate local minima. Connecting two local minima by a line
we quickly obtain a contradiction so there is at most one local minimum on~$L$. However the fact that $\psi:L\r \RR$  is proper, bounded below shows that at least one global
minimum exists. Hence there is a unique global minimum on $L$. The claim about continuity in $L$ is an application of the implicit function theorem.
\end{proof}
A related lemma is the following.
\begin{lemma} \label{lem:maximum}
Assume that $Q\subset \RR^k$ is a convex polygon (not necessarily full dimensional). Then every maximum of $\psi{\mid} Q$ is achieved in a vertex of
$Q$.
\end{lemma}
\begin{proof} A maximum will be in the relative interior of some face $F$. Hence in particular it is an extremal point on the affine span $L$ of $F$. If $F$ is not a point then we obtain a contradiction with Lemma \ref{lem:extremal} which asserts that
any extremal point on~$L$ must be a unique global  minimum, and hence cannot be a local maximum.
\end{proof}
\subsection{Potentials homogeneous of degree two}
\label{sec:weight2}
Following \cite{Zhou} we now discuss \emph{homogeneous potentials of degree two}. These are potentials $\psi:(\CC^\ast)^k-\{\rho=0\}\r \RR$, depending
only on $\rho=(\rho_i)_i$ satisfying $\psi(\lambda \rho)=\lambda^2\psi(\rho)$ for $\lambda>0$. A particular case of
such a potential is a \emph{toric potential} in which case $\psi(\rho)$ is a positive definite quadratic form.
However in \cite{Zhou} the utility of more general potentials is demonstrated.
\begin{remark} We extend $\psi$ to a continuous function $\psi:(\CC^\ast)^k\r \RR$ by putting $\psi(\rho=0)=0$. However $\psi$ is not smooth for $\rho=0$ unless it is toric.
\end{remark}
\begin{lemma} Let $\psi$ be a potential depending only on $\rho$. Then $\grad_g \psi=\sum_k u_k\partial_{\rho_k}$ with $u_k$ being the solution to
\begin{equation}
\label{eq:gradg100}
\sum_k u_k \partial_{\rho_k}\partial_{\rho_j}\psi=\partial_{\rho_j}\psi.
\end{equation}
\end{lemma}
\begin{proof} 
We should have for any vector $X$
\begin{equation}
\label{eq:gradequation}
g(\sum_k u_k\partial_{\rho_k},X)=X\psi.
\end{equation}
It is enough to consider $X=\partial_{\rho_j}$ for all $j$. Applying \eqref{eq:gradequation} with $X=\partial_{\rho_j}$ and using \eqref{eq:zhoumetric} we get \eqref{eq:gradg100}.
\end{proof}
\begin{lemma} 
Assume $\psi$ is homogeneous of degree two. In that case:
\begin{equation}
\label{eq:gradg1}
\grad_g \psi=\sum_j \rho_j\partial_{\rho_j}
\end{equation}
and
\begin{equation}
\label{eq:normgradg1}
g(\grad_g \psi,\grad_g \psi)=2\psi.
\end{equation}
\end{lemma}
\begin{proof}
In order for \eqref{eq:gradg1} to be true, by \eqref{eq:gradg100}  we should have
\[
\sum_k \rho_k \partial_{\rho_k}\partial_{\rho_j} \psi=\partial_{\rho_j} \psi.
\]
This is an application of the Euler identity since $\partial_{\rho_j}\psi$ is homogeneous
of degree one.

 To obtain \eqref{eq:normgradg1} we observe that $g(\grad_g \psi,\grad_g \psi)=(\grad_g \psi)(\psi)$
which is equal to the right hand side of \eqref{eq:normgradg1} by \eqref{eq:gradg1} and the Euler identity.
\end{proof}
\begin{remark} The formula \eqref{eq:gradg1} is a simplification of \cite[Proposition 2.7]{Zhou}.
\end{remark}
\begin{remark} In the original coordinates the formula \eqref{eq:gradg1} can be written as
\begin{equation}
\label{eq:psigradient}
\grad_g\psi=\sum_j\log |z_j|(z_j\partial_{z_j}+\bar{z}_j \partial_{\bar{z}_j}).
\end{equation}
\end{remark}

We can interpolate between homogeneous potentials of degree $2$. For further reference we record this in a lemma.
\begin{lemma}\label{lem:interpolation}
Let $\psi_1,\psi_2$ be two homogeneous potentials of degree two, then $t\psi_1+(1-t)\psi_2$, $t\in [0,1]$ is a homogeneous potential of degree two.
\end{lemma}

For the use below we will need special homogeneous potentials of degree $2$. 
\begin{definition}\cite[Definiton 2.8]{Zhou}\label{def:adapted}
Let $Q\subset \RR^k$ be  a convex polytope (possibly unbounded) containing $0$ as the interior point. 
A  homogeneous degree two convex function $\psi : \RR^k \to \RR$ is {\em adapted to $Q$}  if for each
face $F$ of $Q$ the restriction of $\psi$ to~$F$ has a unique minimum point in the relative interior
of $F$. 
\end{definition}

\begin{proposition}\cite[Proposition 2.10]{Zhou}\label{prop:adapted}
For any convex polytope $Q\subset \RR^k$ containing $0$ as an interior point, the set of Kähler potentials adapted to $Q$ is non-empty and contractible.

\end{proposition}
For use below we note the following fact.
\begin{lemma} \label{lem:liouville}
If $\iota:\CC^{\ast k} \r \CC^{\ast k}: x\mapsto (x^{a_1},\ldots, x^{a_k})$ is an \'etale map of tori (i.e.\ $\det(a_1,\ldots,a_n)\neq 0$) then
the pullback under $\iota$ yields a 1-1 correspondence between potentials depending only on $\rho$, and the same holds for degree two potentials.
\end{lemma}
\begin{proof}
If we write $\CC^{\ast k}=(S^1)^k\times \RR^k$ then $\iota=(\iota_0,\iota_1)$ where  $\iota_1:\RR^k\r \RR^k$ is a
linear isomorphism. From this the conclusion easily follows.
\end{proof}

\subsection{Equivalent metrics}
In this section we show that all metrics associated to potentials homogeneous of degree two are globally equivalent (for $\rho\neq 0$).
We recall the following basic lemma.
\begin{lemma} \label{lem:norm} Let $q(x)=\sum_{i,j=1}^k g_{ij} x_i x_j$ be a positive definite quadratic form. Then
\begin{equation*}
a:=\min_x \frac{q(x)}{|x|^2} \qquad
b:=\max_x \frac{q(x)}{|x|^2}
\end{equation*}
are continuous functions of $(g_{ij})_{ij}$. Moreover $a>0$.
\end{lemma}
\begin{proof}
We may make an orthogonal change of coordinates such that
$
q(x)=\lambda_1 x_1^2+\cdots + \lambda_k x_k^2
$
with $\lambda_1\ge \cdots\ge \lambda_k>0$. The $(\lambda_i)_i$ are the eigenvalues 
of $(g_{ij})_{ij}$ and hence they are continuous functions of $(g_{ij})_{ij}$. We then find
$a=\lambda_k$, $b=\lambda_1$, finishing the proof.
\end{proof}
\begin{proposition}\label{prop:bounds}
Assume that $\psi$ is a potential homogeneous of degree two. Then there exist $0<a\le b$ in $\RR$ such that for all $\rho\in \RR^k-\{0\}$, $\eta\in (S^1)^k$ and $X\in T_{\rho,\eta}(\CC^{\ast k})$ 
we have
\[
a \le \frac{g(\rho)(X,X)}{|X|^2} \le b
\]
where the norm $|X|$ is computed with respect to the metric $\sum_i(d\rho_i^2+d\eta_i^2)$.
\end{proposition}
\begin{proof}
The coefficients of
\[
g=\sum_{ij} (\partial_{\rho_i}\partial_{\rho_j}\psi) (d\rho_i d\rho_j+d\eta_i d\eta_j)
\]
are functions of $\rho$ of degree zero on $\{\rho\neq 0\}$. In other word they are determined by their value on $\{|\rho|=1\}$ which is a
compact set.
By Lemma \ref{lem:norm} we have 
\[
0<a(\rho) \le \frac{g(\rho)(X,X)}{|X|^2} \le b(\rho)
\]
where $a(\rho),b(\rho)$ are continuous functions on $\{|\rho|=1\}$. By putting $a=\min_{|\rho|=1} a(\rho)$ and $b=\max_{|\rho|=1} b(\rho)$
we get what we want.
\end{proof}
\begin{remark} \label{rem:propforforms}
We will need a result similar to Proposition \ref{prop:bounds} for 1-forms: this follows formally from the following observation. Let $g_1$, $g_2$ be two
linear Riemannian metrics on a vector space $V$ such that $ag_1(X,X)\le g_2(X,X)\le bg_1(X,X)$ for $b\ge a>0$ and let  $X_1, X_2\in V$ be such that $g_1(X_1,-)=g_2(X_2,-)$. Then we claim 
$g_2(X_2,X_2)\le g_1(X_1,X_1) \le bg_2(X_2,X_2)$. To see this we may, without loss of generality, assume that $X_1\neq 0$, and equivalently $X_2\neq 0$. Then
\begin{align*} 
g_1(X_1,X_1)^2&=g_2(X_2,X_1)^2\\
&\le g_2(X_2,X_2) g_2(X_1, X_1)\\
&\le bg_2(X_2,X_2) g_1(X_1,X_1).
\end{align*}
This yields $g_1(X_1,X_1)\le bg_2(X_2,X_2)$. The other inequality is similar.
\end{remark} 
\begin{corollary} \label{eq:stein} Assume that $\psi$ is a potential homogeneous of degree two. 
We have $\psi(\rho)>0$ for $\rho\neq 0$. Moreover $\psi$ is proper. 
\end{corollary}
\begin{proof} Assume $\psi(\rho)\le 0$ for some $\rho\neq 0$. By \eqref{eq:normgradg1} we have $(\grad_g \psi)_\rho=0$ and by \eqref{eq:gradg1} this is not possible if  $\rho\neq 0$. 
To prove properness of $\psi$,   we have to prove, using \eqref{eq:normgradg1}\eqref{eq:gradg1}, that $\{g(\rho)(\sum_{i} \rho_i\partial_{\rho_i},\sum_{i} \rho_i\partial_{\rho_i})\le C\}$ defines a compact set for $C\in \RR$.
By Proposition \ref{prop:bounds} it suffices to do this for the standard metric for which it is obvious.
\end{proof}
\begin{lemma}
Let $N=(\CC^\ast)^n$ and let $\psi:(\CC^\ast)^k\r \RR$ be a K\"ahler potential which outside a compact set is homogeneous of degree two. Then $N$ is Stein. 
Moreover $(N,\theta)$ with
$\theta$ as in \eqref{eq:kahlercase} is a Liouville manifold.
\end{lemma}
\begin{proof}
By \eqref{eq:gradg1}, outside a compact set the Liouville vector field is given by $\sum_j \rho_j \partial_{\rho_j}$ which is obviously complete. Moreover $\{|\rho| \le C\}$ for $C\gg 0$
defines a generating Liouville domain. 
The fact that $N$ is Stein follows from Corollary \ref{eq:stein}.
\end{proof}

\begin{lemma}
\label{lem:bounds} Assume that $\psi$, $\psi'$ are two potentials, homogeneous of degree two. Then there exist $0<c<d$ such that on $\{\rho\neq 0\}$ we have
\[
c\psi\le \psi'\le d\psi.
\]
\end{lemma}
\begin{proof}
Both $\psi$ and $\psi'$ are determined by their values on the compact set $\{|\rho|=1\}$ on which they take strictly positive values by Corollary \ref{eq:stein}. From this the lemma easily follows.
\end{proof}

\subsection{Gradings}
\label{sec:canonical}
Let  $z_1,\ldots,z_k$ be toric coordinates of an algebraic
torus $T=(\CC^\ast)^k$. According to \S\ref{sec:grp1} a grading on $T$ can be given by an everywhere non-zero
section of $\omega_\CC^{\otimes -2}$. 
We may take
\begin{equation}
\label{eq:firstgrading}
\left(z_1\partial_1 \wedge \cdots\wedge z_k\partial z_k\right)^{\otimes 2}.
\end{equation}
This does not depend on the choice of toric coordinates. 
Writing $T=TM$ where $M=(S^1)^k$ one checks that the fibers of $TM\r M$ define a polarization of $TM$ in the sense of Remark \ref{rem:pol}.
Hence we see that $T$ has a second ``canonical grading'', as explained in \S\ref{sec:polar}.
\begin{lemma} The two gradings on $T$ we have defined are the same. 
\end{lemma}
\begin{proof}
We choose logarithmic polar coordinates $w= \rho+i\eta$ on $\CC^\ast$ so that $z=e^w$. Then $z\partial_z=\partial_w=(1/2)(\partial_\rho-i\partial_\eta)$ with $\eta\in \RR/2\pi \ZZ$.
The real line subbundles of  $\omega^{-1}_\CC$ associated to the two gradings respectively have bases:
\[
\left(\partial_{\rho_1}-i\partial_{\eta_1}\right)\wedge\cdots\wedge
\left(\partial_{\rho_k}-i\partial_{\eta_k}\right)
\]
and
\[
\partial_{\rho_1}\wedge\cdots \wedge\partial_{\rho_k}.
\]
The line bundle generated by 
\[
\left(\partial_{\rho_1}-it\partial_{\eta_1}\right)\wedge\cdots\wedge
\left(\partial_{\rho_k}-it\partial_{\eta_k}\right)
\]
for $t\in [0,1]$ interpolates between them.
\end{proof}

\subsection{The Legendre transform revisited}
\label{sec:legrevisited}
In this section we assume that $\psi$ is a homogeneous potential of degree two. Otherwise we use the notations from \S\ref{sec:weight2}.
\begin{lemma} \begin{enumerate}
\item $\Phi_\psi$ is a diffeomorphism $N^\ast_\RR-\{0\}\r N_\RR-\{0\}$ which is linear on the half-rays starting in the origin.
\item Let $\psi'$ be a K\"ahler potential (i.e.\ $\psi'$ is smooth) on $N^\ast_\RR$ which agrees with $\psi$ outside a compact set. Then $\Phi_{\psi'}:N^\ast_\RR\r N_\RR$ is a diffeomorphism.
\end{enumerate}
\end{lemma}
\begin{proof}
\begin{enumerate}
\item \label{it:zhou} This is explained in \cite{Zhou} after Definition 2.6.
\item We already know that $\Phi_{\psi'}$ is an injective immersion 
  (see \S\ref{sec:legendre0}). It follows from the inverse function theorem that  $\Phi_{\psi'}$ is open (alternatively one may invoke \ the ``Invariance of domain'' theorem \cite{MR1511658}).
Let $S$ be the complement of the 
  image of $\Phi_{\psi'}$. Then $S$ is closed.  Assume that outside the sphere $K\subset N^\ast_\RR$ centred at $0$ the potentials $\psi$ and $\psi'$ agree. Let $K'$ be the compact set in $N_\RR$ which consists of the points inside $\Phi_{\psi}(\partial K)$.
Then $\psi'$ defines a diffeomorphism between $N^\ast_\RR-K$ and $N_\RR-K'$.  Hence we have $K'=S\coprod \Phi_{\psi'}(K)$ with $ \Phi_{\psi'}(K)$ being compact and hence closed. But this is impossible since
$K'$ is contractible and hence connected. 
\qedhere
\end{enumerate}
\end{proof}
It follows that we obtain a diffeomorphism
\begin{equation}
\label{eq:legendrediffeo1}
\Id\times\Phi_{\psi'}:T(S_{N^\ast})\r T^{\ast}(S_{N^\ast})
\end{equation}
extending the partially defined one in \eqref{eq:legendre}.

 If we let $T^{\ast,\infty}(S_{N^\ast})$, $T^{\infty}(S_{N^\ast})$ be as in \S\ref{sec:fltz} then \eqref{eq:legendrediffeo1} induces a diffeomorphism
\[
\Id\times\Phi^\infty_\psi:T^{\infty}(S_{N^\ast})\r T^{\ast,\infty}(S_{N^\ast}).
\]
For use below we record the following.
\begin{lemma} \label{lem:canonicalgrading}
Under $\Id\times \Phi_{\psi'}$ the canonical grading on $T(S_{N^\ast})$ (see \S\ref{sec:canonical}) corresponds to the canonical grading on $T^\ast(S_{N^\ast})$ (see Example \ref{ex:canonical}).
\end{lemma}
\begin{proof} The fibers of  $T^\ast(S_{N^\ast})\r S_{N^\ast}$ form a Lagrangian foliation. Likewise
the fibers of $T(S_{N^\ast})\r S_{N^\ast}$ form a Lagrangian foliation by the formulas \eqref{eq:omega}. The two Lagrangian foliations correspond to each other via  $\Id\times \Phi_{\psi'}$.
Hence the corresponding gradings correspond. These are the gradings defined in \S\ref{sec:canonical} and Example \ref{ex:canonical}.
\end{proof}
\section{Some families of hypersurfaces in tori}
\subsection{Families of Laurent polynomials}\label{subsec:families}
We use the notations as in \S\ref{sec:prindet} without further comment. 
We equip $(\CC^\ast)^k$ with a toric K\"ahler potential $\psi$ as in \S\ref{sec:weight2}.
For $\alpha\in \CC^A-V(A)$ we consider $f^{-1}_\alpha(0)\subset (\CC^\ast)^k$. To avoid trivialities we assume $|A|>1$ so that $f^{-1}_\alpha(0)$ is non-empty by Lemma \ref{lem:notzero}. 
We equip $f^{-1}_\alpha(0)$ with the K\"ahler structure induced from $(\CC^\ast)^k$.  In this section we prove the following result.
\begin{theorem} \label{prop:Dlocally} The family of exact symplectic manifolds $(f^{-1}_\alpha(0))_{\alpha\in \CC^A-V(A)}$ is a family of Liouville manifolds which locally has
families of generating Liouville domains.
\end{theorem}
\begin{proof} We apply Proposition \ref{prop:fiber_families} with $N=(\CC^\ast)^k$ and
$r=0$ together with Proposition \ref{lem:compactangle0} below with $\delta=0$. 
\end{proof}
As the above proof shows, the following proposition is much more general than what we need. However it seems that the case $\delta=0$, which assert that the
critical points of $\grad_g\psi \mid f^{-1}_\alpha(0)$ for $\alpha\in K$ form a compact set, is not easier to prove. Moreover some results obtained in the proof will be reused in
the proof of
 Theorem \ref{prop:tropical} below.
\begin{proposition} \label{lem:compactangle0}
Let $K\subset \CC^A-V(A)$ be a compact set. For $\alpha\in K$ let $\tau$ be the Hermitian angle (see \S\ref{sec:hermitian} below) between $\grad_g\psi$ and $\partial f_\alpha$ on $f^{-1}_\alpha(0)$ (so
$\tau$ is a function on $N_K:=\bigcup_{\alpha\in K}  f^{-1}_\alpha(0)\subset (\CC^\ast)^k\times K$). Then for $0\le \delta<\pi/2$, the set $\{\tau \le \delta\}$ is a compact subset of $N_K$.
\end{proposition}
\begin{proof} Assume that $\{\tau \le \delta\}$ is not compact. So choose a sequence $(x^{(\ell)},\alpha^{(\ell)})_\ell$ in $N_K\subset (\CC^\ast)^k\times K$ such that $\tau^{(\ell)}:=\tau(x^{(\ell)},\alpha^{(\ell)})\le \delta$, which does not
have a convergent subsequence. We will obtain a contradiction.

Since $(x^{(\ell)})_{\ell}$ must be unbounded (otherwise there would be a convergent subsequence) we may, as $\psi$ is proper by Corollary \ref{eq:stein}, replace  $(x^{(\ell)})_{\ell}$ by a subsequence such that 
\begin{equation}
\label{eq:unbounded}
\lim_{\ell\r\infty} \psi(x^{(\ell)})=\infty.
\end{equation}
After further replacing  $(x^{(\ell)},\alpha^{(\ell)})_\ell$ by a subsequence we may assume that  $((x^{(\ell)})^{a_1}:\cdots:(x^{(\ell)})^{a_d})_\ell$ is convergent in $X_{A}$ 
(see \S\ref{sec:Atoric}) and $(\alpha^{(\ell)})_\ell$ is convergent
with a limit $\bar{\alpha}$ in $K\subset \CC^{d}-V(A)$ (since $(\alpha^{(\ell)})_\ell\subset K$ and $K$ is compact).

By Proposition \ref{prop:XA} there exists a face $F$ of $P$ such that the limit of the sequence  $((x^{(\ell)})^{a_1}:\cdots:(x^{(\ell)})^{a_d})_\ell$ 
lies in the $(\CC^\ast)^k$-orbit corresponding to $F$. Choosing an arbitrary ${a}_l\in F$ this implies:
\begin{equation}\label{eq:facelimit1}
\lim_{\ell\to \infty} ({x}^{(\ell)})^{{a}_j-{a}_l}=
\begin{cases}
\text{exists $\neq 0$} & \text{if ${a}_j\in F$},\\
0&\text{if ${a}_j\not\in F$}.
\end{cases}
\end{equation}
Put $p=\dim F$. After pullback for a suitable \'etale covering $e:(\CC^\ast)^k\r (\CC^\ast)^k$, and replacing $A$ by $A\circ e$
we may assume that 
  $x^{a_i-a_j}$, $a_i,a_j\in F$ are monomials in $(x_l)_{l=1,\ldots, p}$.
We choose a lift of  the sequence $(x^{(\ell)})_\ell\subset  (\CC^\ast)^k$ for the \'etale cover $e$ and denote it also by  $(x^{(\ell)})_\ell$.
Note that by Proposition \ref{prop:affine}  $V(A)$ is invariant under \'etale coverings.

With these simplifications we observe that $(x^{a_i-a_l})_{a_i\in F}$ determines
$(x_i)_{i=1}^p$ up to a finite number of possibilities. By  \eqref{eq:facelimit1} we may replace  $(x^{(\ell)},\alpha^{(\ell)})_\ell$ by a further subsequence
such that 
\begin{equation}
\label{eq:ulimit1}
\bar{x}_i= \lim_{\ell\to\infty} x^{(\ell)}_i\qquad  \text{exists and is nonzero for $i=1,\ldots,p$}.
\end{equation}
By \eqref{eq:psigradient} we get
\begin{equation}
\label{eq:gradu10}
\grad_g\psi=\sum_i\log|x_i| (x_i\partial_{x_i}+\bar{x}_i\bar{\partial}_{x_i})
\end{equation}
while by definition
\[
\partial f_\alpha=\sum_i (x_i\partial_{x_i} f_\alpha) \frac{dx_i}{x_i}.
\]
Put $g_\alpha=x^{-a_l}f_\alpha$. Then on $f^{-1}_\alpha(0)$ we have
\begin{equation}
\label{eq:umult}
\partial g_\alpha=x^{-a_l}\partial f_\alpha.
\end{equation}
By \eqref{eq:facelimit1} we have
\begin{equation}
\label{eq:limitexists}
\lim_{\ell\r \infty} (x_i\partial_{x_i})( g_{\alpha^{(\ell)}})(x^{(\ell)})\text{ exists for $i=1,\ldots,k$}.
\end{equation}
We can however be more precise. If $i>p$ then, since $x^{{a}_j-{a}_l}$ for $a_j\in F$ does not depend on $x_i$, we get from \eqref{eq:facelimit1}
\begin{equation}
\label{eq:splim2}
\lim_{\ell\r \infty} (x_i\partial_{x_i})( g_{\alpha^{(\ell)}})(x^{(\ell)})=0.
\end{equation}
If $1\le i\le p$ then (see \S\ref{sec:prindet} for the notation $p_F$)
\begin{equation}
\label{eq:splim1}
\lim_{\ell\r \infty} (x_i\partial_{x_i})(g_{\alpha^{(\ell)}})(x^{(\ell)})=(x_i\partial_{x_i})( g^{F-a_l}_{p_F(\bar{\alpha})})(\bar{x}).
\end{equation}
Since $\bar\alpha\not\in V(A)=V(A-a_l)$ (cfr.\ Proposition \ref{prop:affine}), we have $(\partial g^{F-a_l}_{p_F(\bar{\alpha})})(\bar{x})\neq 0$ and hence
there is some $1\le i\le p$ such that 
\begin{equation}\label{eq:supp}
\lim_{\ell\r \infty} (x_i\partial_{x_i})( g_{\alpha^{(\ell)}})(x^{(\ell)})\neq 0
\end{equation}
which yields
\begin{equation}
\label{eq:splim3}
\lim_{\ell\r \infty} |(\partial g_{\alpha^{(\ell)}})(x^{(\ell)})|\neq 0.
\end{equation}
To see this note that it can be checked for any metric obtained from a toric potential by Proposition \ref{prop:bounds}\footnote{Proposition \ref{prop:bounds} is stated for vector fields
but there is an analogue for (complex valued holomorphic) forms by Remark \ref{rem:propforforms} combined with \eqref{eq:normproportionality}.}. Hence we can do it for the potential $\sum_i (\log |x_i|)^2$ for which it follows from \eqref{eq:zhoumetric} and \eqref{eq:supp}.

Let $\operatorname{Ngrad}_g\psi$ be the normalization of $\grad_g\psi$. Then the components of $\operatorname{Ngrad}_g\psi$ are bounded.
Note furthermore that by \eqref{eq:normgradg1}\eqref{eq:psigradient}\eqref{eq:unbounded}\eqref{eq:ulimit1} and \eqref{eq:gradu10} the components of 
$(x_i\partial_{x_i}+ \bar{x}_i\bar{\partial}_{x_i})$ in $\operatorname{Ngrad}_g\psi(x^{(\ell)})$ for $i\in \{1,\ldots,p\}$ tend to zero as $\ell\r \infty$. 
Hence by \eqref{eq:facelimit1}\eqref{eq:splim2} 
\begin{equation}
\label{eq:keypoint}
\lim_{\ell\r\infty} \partial g_{\alpha^{(\ell)}}(x^{(\ell)})(\operatorname{Ngrad}_g\psi(x^{(\ell)}))=0.
\end{equation}
Let us write $\angle_h$ for Hermitian angle (see \S\ref{sec:hermitian}). Since Hermitian angles are not affected by complex multiples, we have by \eqref{eq:umult}
\begin{align*}
\tau^{(\ell)}=\angle_h(\grad_g\psi(x^{(l)}),\partial f_{\alpha^{(\ell)}}(x^{(\ell)}))&=\angle_h(\operatorname{Ngrad}_g\psi(x^{(\ell)}),\partial g_{\alpha^{(\ell)}}(x^{(\ell)})).
\end{align*}
Then we find using \eqref{eq:splim3}\eqref{eq:keypoint}
\begin{align}\label{eq:costaul}
\lim_{\ell\r\infty}\cos\tau^{(\ell)}
&=\lim_{\ell\r\infty}\frac{|\partial g_{\alpha^{(\ell)}}(x^{(\ell)})(\operatorname{Ngrad}_g\psi(x^{(\ell)}))|}{ |\partial g_{\alpha^{(\ell)}}(x^{(\ell)})|}\\
&=0.
\end{align}
Hence $\lim_{\ell\r\infty}\tau^{(\ell)}=\pi/2$ which contradicts $0\le \tau^{(\ell)}\le \delta<\pi/2$.
\end{proof}
\subsection{Tropical localization}
\label{sec:paistropical}
We use the notations as in \S\ref{sec:prindet}-\S\ref{sec:tailoring} without further comment. These include $A$, $\nu$, $\chi_{a_i}$, $C_{a_i}$, $K_i$, $\epsilon_i$. We assume that $0\not\in \Pi_\nu$
so that by \cite[Lemma 4.6]{Abouzaid} 
\begin{equation}
\label{eq:notzero}
\{\rho=0\}\cap \Log_t(f^{-1}_{t,s}(0))=\emptyset
\end{equation}
 for $\epsilon_1,\epsilon_2$ chosen small enough and $t\gg 0$. So the homogeneous potentials of degree two
are defined on $f^{-1}_{t,s}(0)$.
If the condition $0\not\in \Pi_v$ is not satisfied then in the discussion below one should restrict to toric potentials which are defined everywhere.

\medskip

By \cite[Proposition 4.9]{Abouzaid} all $f^{-1}_{t,s}(0)$ for $s\in [0,1]$ are symplectomorphic for $t\gg 0$ (they are non-empty e.g.\ by Lemma \ref{rem:easysmooth}). We prove a slight strengthening of this result in Theorem \ref{prop:tropical} below.

For $s\in [0,1]$ we choose a smooth family $(\psi_s)_s$ of potentials, homogeneous of degree two, on $(\CC^\ast)^k-\{\rho=0\}$ (they will be further restricted later).
We decorate the concepts associated with $\psi_s$ by a subscript $(-)_s$; e.g. the associated Liouville form is denoted by $\theta_s$ and we use the same notation for
its restriction to $f^{-1}_{t,s}(0)$. We use $|-|_g$ for the norm with respect to the Hermitian metric corresponding to a Riemannian metric $g$. To avoid confusion we advise the reader
to consult Remark \ref{rem:normproportionality}.
\begin{theorem}\label{prop:tropical}
For $t\gg 0$, the family $(f^{-1}_{t,s}(0),\theta_s)_{s\in [0,1]}$ is a family of Liouville manifolds which locally has
families of generating Liouville domains.
\end{theorem}
\begin{proof} We apply Proposition \ref{prop:fiber_families}
together with Lemmas \ref{lem:Cbound}, \ref{lem:compactangle}  below. 
\end{proof}

\begin{lemma} \label{lem:Cbound}
\begin{equation}
\label{eq:Cbound}
\forall r>0:\exists T>0:\forall t\ge T:\forall s\in [0,1]:\forall x\in f^{-1}_{t,s}(0):|\bar{\partial} f_{t,s}(x)|_{g_s} < r|\partial f_{t,s}(x)|_{g_s}.
\end{equation}
Hence in particular
\[
\exists T>0:\forall t\ge T:\forall s\in [0,1]:\forall x\in f^{-1}_{t,s}(0):\partial f_{t,s}(x)\neq 0.
\]
\end{lemma}
\begin{proof}
For $x\in (\CC^\ast)^k$  it follows from \eqref{eq:parbound}\eqref{eq:negligible}\eqref{eq:mainbound} that
 $|\bar{\partial} f_{t,s}(x)|=|\bar{\partial} f^{\text{small}}_{t,s}(x)|$ can be made arbitrarily small relative to $|\partial f_{t,s}(x)|\ge |\partial f^{\text{core}}_{t,s}(x)|-|\partial f^{\text{small}}_{t,s}(x)|$ for $t\gg 0$. The lemma follows from this combined with Proposition \ref{prop:bounds} (and Remark \ref{rem:propforforms}).
\end{proof}
\begin{lemma} \label{lem:compactangle} Choose $0\le r<1$.
Let $\tau$ be the Hermitian angle (see \S\ref{sec:hermitian}) between $\grad_{g_s}\psi_s$ and $\partial f_{t,s}$ on $f^{-1}_{t,s}(0)$ (so 
$\tau$ is a function on $N_t:=\bigcup_{s\in [0,1]}  f^{-1}_{t,s}(0)$).
Then for $t\gg 0$ the set $\{\sin \tau \le r\}$ is a compact subset of $N_t$.
\end{lemma}
\begin{proof}
We choose $0<a<b$, $0<c<d$ such that for all $s\in [0,1]$ we have
\begin{equation}
\label{eq:equivalent}
\begin{gathered}
a|X|^2 \le |X|^2_{g_s}\le b|X|^2\\
c\psi \le \psi_s\le d\psi
\end{gathered}
\end{equation}
as in Proposition \ref{prop:bounds}, Lemma \ref{lem:bounds} where $\psi=(\Log)^2$ denotes the standard toric potential on $(\CC^\ast)^k$, and $|-|$ is the associated metric.
 Note that we can choose $a,b,c,d$ uniformly in $s$ since $[0,1]$ is a compact set.

 Put $\delta=\sin^{-1} r\in [0,\pi/2.[$. 
 Assume that $\{\tau \le \delta\}$ is not compact. So choose a sequence $(x^{(\ell)},s^{(\ell)})_\ell$ in $N_t$ such that $\tau^{(\ell)}:=\tau(x^{(\ell)},s^{(\ell)})\le \delta$, which does not
have a convergent subsequence. We will obtain a contradiction. By further replacing $(x^{(\ell)},s^{(\ell)})_\ell$ by a subsequence we may assume that all $p^{(\ell)}=\Log_t x^{(\ell)}$ 
are contained in the same $C_{a_h}$ and furthermore $A_{\ast}(p^{(\ell)})$ for $\ast \in \{0,1,2\}$ do not depend on $\ell$.

Since $(x^{(\ell)})_{\ell}$ must be unbounded (otherwise there would be a convergent subsequence) we may replace  $(x^{(\ell)})_{\ell}$ by a subsequence such that  
\[
\lim_{\ell\r\infty} \psi(x^{(\ell)})=\infty.
\]
By \eqref{eq:equivalent} we also obtain
\[
\lim_{\ell\r\infty} \psi_{s^{(\ell)}}(x^{(\ell)})=\infty.
\]
Put
\[
c^{(\ell)}=
\frac{ |\partial f^{\text{small}}_{t,s^{(\ell)}}(x^{(\ell)})|_{g_s}}{ |\partial f^{\text{core}}_{t,s^{(\ell)}}(x^{(\ell)})|_{g_s}}
\]
where $f^{\text{core}}_{t,s}$, $f_{t,s}^{\text{small}}$ are as defined
in \S\ref{sec:tailoring}.  The fraction $c^{(\ell)}$ is well defined
by \eqref{eq:mainbound} and moreover by
\eqref{eq:parbound}\eqref{eq:negligible}\eqref{eq:mainbound}\eqref{eq:equivalent}
we obtain
\begin{equation}
\label{eq:cell}
c^{(\ell)}\le C_1t^{-C_2}
\end{equation}
for some constants $C_1,C_2>0$ depending only on $A$, $\nu$, $\alpha$, $(\epsilon_i)_i$. We have:
\begin{align*}
\cos\tau^{(\ell)}&=\frac{|\partial f_{t,s^{(\ell)}}(x^{(\ell)})(\operatorname{Ngrad}_{g_{s^{(\ell)}}}\psi_{s^{(\ell)}}(x^{(\ell)}))|_{g_{s^{(\ell)}}}}{|\partial f_{t,s^{(\ell)}}(x^{(\ell)})|_{g_{s^{(\ell)}}}}
\end{align*}
(the notation $\operatorname{Ngrad}$ was introduced in the proof of Proposition \ref{lem:compactangle0}).
For the numerator we have (using Cauchy-Schwarz):
\begin{align*}
|\partial f_{t,s^{(\ell)}}(&x^{(\ell)})(\operatorname{Ngrad}_{g_{s^{(\ell)}}}\psi_{s^{(\ell)}}(x^{(\ell)}))|_{g_{s^{(\ell)}}}\\
&\le |\partial f^{\text{core}}_{t,s^{(\ell)}}(x^{(\ell)})(\operatorname{Ngrad}_{g_{s^{(\ell)}}}\psi_{s^{(\ell)}}(x^{(\ell)}))|_{g_{s^{(\ell)}}}\\
&\qquad\qquad\qquad\qquad+|\partial f^{\text{small}}_{t,s^{(\ell)}}(x^{(\ell)})(\operatorname{Ngrad}_{g_{s^{(\ell)}}}\psi_{s^{(\ell)}}(x^{(\ell)}))|_{g_{s^{(\ell)}}}\\
&\le |\partial f^{\text{core}}_{t,s^{(\ell)}}(x^{(\ell)})(\operatorname{Ngrad}_{g_{s^{(\ell)}}}\psi_{s^{(\ell)}}(x^{(\ell)}))|_{g_{s^{(\ell)}}}+|\partial f^{\text{small}}_{t,s^{(\ell)}}(x^{(\ell)})|_{g_{s^{(\ell)}}}\\
&=|\partial f^{\text{core}}_{t,s^{(\ell)}}(x^{(\ell)})(\operatorname{Ngrad}_{g_{s^{(\ell)}}}\psi_{s^{(\ell)}}(x^{(\ell)}))|_{g_{s^{(\ell)}}}+c^{(\ell)}|\partial f^{\text{core}}_{t,s^{(\ell)}}(x^{(\ell)})|_{g_{s^{(\ell)}}}
\end{align*}
and for the denominator we have
\[
|\partial f_{t,s^{(\ell)}}(x^{(\ell)})|_{g_{s^{(\ell)}}}\ge (1-c^{(\ell)})|\partial f^{\text{core}}_{t,s^{(\ell)}}(x^{(\ell)})|_{g_{s^{(\ell)}}}.
\]
Put 
\[
\cos \tau^{(\ell),\text{core}}=\frac{|\partial f^{\text{core}}_{t,s^{(\ell)}}(x^{(\ell)})(\operatorname{Ngrad}_{g_{s^{(\ell)}}}\psi_{s^{(\ell)}}(x^{(\ell)}))|_{g_{s^{(\ell)}}}}{|\partial f^{\text{core}}_{t,s^{(\ell)}}(x^{(\ell)})|_{g_{s^{(\ell)}}}}.
\]
Then we get
\begin{equation}
\label{eq:cosformula}
\cos\tau^{(\ell)}\le 
\frac{\cos \tau^{(\ell),\text{core}}+c^{(\ell)}}{1-c^{(\ell)}}.
\end{equation}
Since $f^{\text{core}}_{t,s}(x)$ is a Laurent polynomial (independent of $\ell$ by the choice of $A_\ast(x^{(\ell)})$ above) we obtain as in the proof of Proposition \ref{lem:compactangle0} (cf. \eqref{eq:costaul}) that
\begin{equation}
\label{eq:coslimit}
\lim_{\ell \r \infty} \cos \tau^{(\ell),\text{core}}=0.
\end{equation}
The key point is that the formula for $\grad_g \psi$ used in the
proof of Proposition \ref{lem:compactangle0} (see \eqref{eq:gradu10}) is valid for arbitrary potentials homogeneous of degree two by \eqref{eq:gradg1}\eqref{eq:psigradient}.

Hence using \eqref{eq:cell}\eqref{eq:cosformula}\eqref{eq:coslimit} we can choose $t\gg 0$ in such a way such that for $\ell\gg 0$ we have
\[
\cos\tau^{(\ell)}<\sqrt{1-r^2}
\]
and hence $\sin \tau^{(\ell)}>r$, or $\tau^{(\ell)}>\delta$, contradicting the initial choice of $(\tau^{(\ell)})_\ell$.
\end{proof}
\section{Star-shaped triangulations and HMS in the tropical limit}
\label{sec:setting}
In this section we freely use the notations and concepts introduced in \S\ref{sec:toricprelims}. 
We assume $0\in A\subset \ZZ^k$ and we assume that $P=\conv(A)$ is equipped with a star-shaped triangulation $\Tscr$ as in \S\ref{sec:stacky}.
We denote the associated stacky fan by~$\Sigma$. We assume $|A|>1$ so that Lemma \ref{rem:easysmooth} applies. 
We assume that a function $\nu:A\r \RR$ as in \S\ref{sec:polyhedral} exists whose associated triangulation  $\Delta_\nu$
coincides with $\Tscr$. 
\begin{convention}
\label{conv:starshaped}
Whenever we are using an explicit star-shaped triangulation we silently make the following additional conventions.  
\begin{enumerate}
\item We put $a_1=0$. 
\item\label{item:conv2}  We assume that $\nu(a_1)=0$ and that $0$ is the unique minimal value of $\nu(a_i)$. 
This is equivalent to $0$ being in the relative interior of $C_0:=C_{a_1}$ (see \S\ref{sec:polyhedral}). See also \cite[Example 6.1.2]{GammageShende}.
\item For $i=1,\ldots,d$ we put $\alpha_i=1$ for $i\neq 1$ and $\alpha_1=-1$.  
\end{enumerate}
Sometimes we let our numbering start from $0$ instead of $1$. In that case the above conventions apply with $1$ replaced by $0$.
\end{convention}

We choose a potential $\psi$ on $\RR^k$ which is homogeneous of degree two outside a  compact neighbourhood $K$ of $0$
and which is adapted to $C_0$ (see Definition \ref{def:adapted}). This
is important for Theorem \ref{thm:Zhou} below which is valid only
under this assumption.\footnote{In \cite{GammageShende} the authors
  use the toric potential $\rho\mapsto|\rho|^2$ and an additional
  technical condition on the fan, i.e. they assume that the fan is
  ``perfectly centred''. Their condition is equivalent to the toric
  potential being adapted to $C_0$ for some choice of
  $\nu$.\label{foot:PC}}

We choose cut-off functions~$\chi_{a_i}$ as
in \S\ref{sec:tailoring} and we use $(\alpha_i)_i$ and 
$(\chi_{a_i})_i$ to define $f_{t,1}:(\CC^\ast)^k\r \CC$ as in
\S\ref{sec:tailoring}. We also assume that $t$ is large enough so that $f_{t,1}^{-1}(0)$ is a Liouville manifold.  
This is a special cases for $s=1$ of Theorem \ref{prop:tropical}. In general we choose $t$ large enough such that $K/\log t$ is disjoint from all regions of interest. In particular 
$(K/\log(t))\cap \partial C_0=\emptyset$ (see \eqref{item:conv2}) and  $K\cap \Log (f^{-1}_{t,1}(0))=\emptyset$ (see \eqref{eq:notzero}).

Assuming furthermore that the cut-off functions $\chi_{a_i}$ are chosen sufficiently
carefully (besides the properties we already imposed in
\S\ref{sec:tailoring}) one may prove the following
result. 
\begin{theorem}[{\cite[Theorem 3]{Zhou}}]\label{thm:Zhou} If $t\gg 0$ then
the diffeomorphism (see \S\ref{sec:legrevisited})  
\[
\Id\times \Phi_\psi^\infty:T^\infty(S_{N^\ast})\r T^{\ast,\infty}(S_{N^\ast})
\]
induces a homeomorphism between $\Skel(f^{-1}_{t,1}(0))$ (see \S\ref{sec:liouville}) and $\Lambda^\infty_\Sigma$ (see \S\ref{sec:fltz}). The map 
$\Skel(f^{-1}_{t,1}(0))\r T^\infty(S_{N^\ast})$ is the composition $\Skel(f^{-1}_{t,1}(0))\r (N^\ast-\{0\})\times S_{N^\ast}\r T^\infty(S_{N^\ast})$.
\end{theorem} 
We have the following result.
\begin{proposition}[{\cite[Corollary 6.2.6]{GammageShende}}] 
\label{prop:GSpair} If $t\gg 0$ then
there are generating Liouville domains $M\subset (\CC^\ast)^k$, $F\subset f^{-1}_{t,1}(0)$ such that $(M,F)$ is a Liouville pair (see \ref{sec:restriction}) and such that $\partial M\cap K=\emptyset$.
\end{proposition}

The proposition was proved in loc. cit. for the toric potential. 
One starts with choosing a Liouville domain $M'$ which is large enough so that $\partial M'\cap K=\emptyset$. 
Then we proceed as in loc.\  cit.\ invoking the following lemma which in loc.\ cit.\ is stated for toric potentials but which is true more generally.
\begin{lemma}[{\cite[Lemma 6.2.5]{GammageShende},\cite{Zhou}}] \label{prop:GSpushout}
If $t\gg 0$ then in a neighbourhood of the skeleton of $f^{-1}_{t,1}(0)$ the Liouville vector field on $(\CC^\ast)^k$ is nowhere tangent 
to $f^{-1}_{t,1}(0)$.
\end{lemma}
\begin{proof} This statement is actually implicit in the statement of
  Theorem \ref{thm:Zhou}.  For the benefit of the reader we recall the
  proof (based on some results in \cite{Zhou}) since it avoids the
  induction over pants used in \cite[Lemma 6.2.5]{GammageShende}.  Let
  $\tilde{C}_0$ be the closure of the connected component of the
  complement of $\Log_t(f^{-1}_{t,1}(0))$ that contains $0$.  By
  \cite[\S5.2,5.4]{Zhou} we have
  $\Skel(f^{-1}_{t,1}(0))\subset \Log^{-1}_t(\partial \tilde{C}_0)$.
  Hence it is sufficient to prove for
  $x\in \Log^{-1}_t(\partial \tilde{C}_0)\cap f^{-1}_{t,1}(0)$ that
  $Z_x$ is not contained in $T_x(f^{-1}_{t,1}(0))$ since this will
  then also be true for a small neighbourhood of the skeleton (since
  the skeleton is compact, cf. \eqref{eq:capformula}).  Assume that
  $Z_x\in T_x(f^{-1}_{t,1}(0))$ holds.  Let $\rho_t:=\Log_t(x)$. Then
  $Z_\rho:=(d\Log_t)(Z_x)\in T_{\rho_t}(\partial\tilde{C}_0)$ since
  $(d\Log_t)(T_x(f^{-1}_{t,1}(0)))\subset T_{\rho_t}(\partial
  \tilde{C}_0)$.

  We choose logarithmic polar coordinates $(\rho_i,\eta_i)_i$ on
  $(\CC^\ast)^k$ as in \S\ref{sec:zhou}. Then by \eqref{eq:gradg1}
  $Z_x=\sum_i \rho_i\partial_{\rho_i}$ and hence
  $Z_\rho=(1/\log t) \sum_i \rho_i\partial_{\rho_i}$. So it follows
  that a line through the origin is tangent to
  $\partial \tilde{C}_0$.  Something like this cannot hold if we
  replace $\tilde{C}_0$ by $C_0$ since $C_0$ is a convex  (we
  define the tangent bundle to $C_0$ as the kernel of the outward
  conormal bundle).

It follows that the statement is also false for $\tilde{C}_0$ since by
\cite[Propositions 1.12, 1.13]{Zhou} the Hausdorff distance between
$\partial\tilde{C}_0$ and $\partial{C}_0$ goes to zero as
$(\epsilon_i)_i\r 0$, and the same is true for the unit outward
conormal bundles (see also \cite[Lemma 4.6]{Abouzaid}).
\end{proof}

Below  $\Dscr(\Wscr\Fscr(-))$ stands for the Karoubian closure of the pretriangulated closure of $\Wscr\Fscr(-)$.

\begin{theorem}[\protect{\cite[Theorem 1.0.1]{GammageShende}}] \label{th:mirror_} 
Assume $A$, $\nu$, $\Tscr$, $\psi$ and the cutoff functions are as stated in the beginning of this section. 
Let $\Sigma$ be the stacky fan determined by $\Tscr$ and let $X_\Sigma$ (see \S\ref{sec:stacky})
be the associated toric stack.

Let $(M,F)$ be a Liouville pair in $(T(S_{N^\ast}),f_{t,1}^{-1}(0))$ as in Proposition \ref{prop:GSpair}. 
We equip $T(S_{N^\ast})=\CC^{\ast k}$ with the canonical grading on a torus introduced in \S\ref{sec:canonical} and we
equip~$F$ (and hence $f^{-1}_{t,1}(0)$) with the restricted grading as defined in Definition \ref{def:liouvillepair}. Finally we equip
$f^{-1}_{t,1}(0)$ with the trivial background class.

Then for $t\gg 0$ we have
$\Tw(\Wscr\Fscr(f^{-1}_{t,1}(0))_\CC)\cong D^b(\coh(\partial X_\Sigma))$  where
$\partial X_\Sigma$ is the toric boundary divisor of $X_\Sigma$, i.e.\ the complement of the generic orbit.
\end{theorem}
\begin{proof} For the benefit of the reader we sketch the proof given in loc. cit., making sure that the grading and background class
are the expected ones.
We consider the symplectomorphism (see \eqref{eq:legendrediffeo1}, \eqref{eq:legendre10})
\begin{equation}
\label{eq:legendrediffeo}
\Id\times\Phi_\psi:\CC^{\ast k}=T(S_{N^\ast})\r T^{\ast}(S_{N^\ast}).
\end{equation}
Put
\[
\tilde{f}_{t,0}=f_{t,0}\circ (\Id\times\Phi_\psi)^{-1}
\]
and $\tilde{M}=(\Id\times\Phi_\psi)(M)$, $\tilde{F}=(\Id\times\Phi_\psi)(F)$. Hence $(\tilde{M},\tilde{F})$ is a Liouville pair  which completes to 
$(T^\ast(S_{N^\ast}),\tilde{f}^{-1}_{t,1}(0))$. 
By Theorem \ref{thm:Zhou}, $\Skel(\tilde{f}_{t,0}^{-1}(0))\cong \Lambda^\infty_\Sigma$.
By Lemma \ref{lem:canonicalgrading} the canonical gradings on  $T(S_{N^\ast})$ and $T^\ast(S_{N^\ast})$ correspond. Hence so do the restricted gradings on
$f^{-1}_{t,1}(0)$ and $\tilde{f}^{-1}_{t,1}(0)$ obtained via the Liouville pairs
$(M,F)$ and $(\tilde{M},\tilde{F})$ (see Definition \ref{def:liouvillepair}). Since $S_{N^\ast}$ has free tangent bundle, the background class on $T^\ast(S_{N^\ast})$
coming from the Langrangion foliation is trivial (see \S\ref{sec:backgroundcotangent}). Hence all constructions are compatible with trivial background classes.

For these gradings and (trivial) background classes we trivially have 
\begin{equation}
\label{eq:totilde}
\Wscr \Fscr(f^{-1}_{t,1}(0))=\Wscr\Fscr(\tilde{f}^{-1}_{t,1}(0))).
\end{equation}

Furthermore {by Theorem \ref{th:weinstein} below $f^{-1}_{t,1}(0)$ is (Morse-Bott) Weinstein} 
 and hence the same is true for $\tilde{f}^{-1}_{t.1}(0)$.  
Hence by \cite[Theorem 1.4]{GPS3} we have 
\begin{equation}
\label{eq:tomicro}
\Tw (\Wscr\Fscr(\tilde{f}^{-1}_{t,1}(0))^\circ \cong \mu \operatorname{sh}_{\Lambda^\infty_\Sigma}(\Lambda^\infty_\Sigma)^c 
\end{equation}
where
$\operatorname{sh}_{\Lambda^\infty_\Sigma}(\Lambda^\infty_\Sigma)$ is
computed using the embedding
$\Lambda^\infty_\Sigma\hookrightarrow T^{\ast,\infty}(S_{N^\ast})$.
By \cite[Thm 7.4.1]{GammageShende}
$D^b(\coh(\partial X_\Sigma))\cong (\mu
\operatorname{sh}_{\Lambda^\infty_\Sigma}(\Lambda^\infty_\Sigma)_\CC)^c$. 
By Grothendieck
duality we have
$D^b(\coh(\partial X_\Sigma))^\circ \cong D^b(\coh(\partial
X_\Sigma))$. Combining this with \eqref{eq:totilde} and
\eqref{eq:tomicro} (with coefficients extended from $\ZZ$ to $\CC$) finishes the proof.
\end{proof}

\section{Grading and background class}
\label{sec:gradingbackground}
\subsection{Introduction}
\label{sec:gradingbackgroundintro}
In this section we assume that we are in the setting of
\S\ref{sec:setting}. Note in loc.\ cit.\ we put a number of
restrictions on $A,\nu,\Tscr,\psi,\dots$ which we will continue to assume (in particular Convention \ref{conv:starshaped}).
Throughout we assume that $t\gg 0$ so that the conclusions of
\S\ref{sec:setting} hold, as well as those of \S\ref{sec:paistropical}.

Our strategy will be to combine Theorems \ref{prop:Dlocally},
\ref{prop:tropical} and \ref{th:mirror_} to obtain an equivalence
$\Tw \Wscr\Fscr(f^{-1}_\alpha(0))\cong D^b(\coh(\partial X_\Sigma))$ on which
we can then define an appropriate monodromy action via Theorem
\ref{lem:moser2} and \S\ref{sec:wrapped}.  However to this end we have
to take care of gradings and background classes. For the background classes there is nothing to do.
We equip all Liouville manifolds with the trivial background class which is clearly compatible
with families so that we can combine Theorem \ref{lem:moser2} with Theorem \ref{lem:bgisotopies}.

The gradings represent a more delicate problem. We have to put gradings on
the appropriate families of Liouville manifolds so that we can apply
Theorem \ref{lem:simpiso}.
Unfortunately the grading used in
Theorem \ref{th:mirror_} is specific for $f^{-1}_{t,1}(0)$. So it is
not obvious that it can for example be extended to the family
$(f^{-1}_{t,s}(0))_s$. Therefore in this section we will show:
\begin{proposition} \label{prop:maingradingprop}
The grading on $f^{-1}_{t,1}(0)$ used in Theorem \ref{th:mirror_} is obtained from the restriction of the canonical grading on $\CC^{\ast k}$  (see \S\ref{sec:canonical})
to the zero fiber of  $f_{t,1}$ via the construction in Definition \ref{def:lemfibers}.
\end{proposition}
Obviously the canonical grading on $\CC^{\ast k}$ can be restricted to any
hypersurface $f^{-1}(0)$ where $f$ satisfies the assumptions as in \S\ref{sec:fibergrading}. So it can be defined in families of hypersurfaces. We will call this the \emph{natural grading}. 
The grading on $f^{-1}_{t,1}(0)$ used in Theorem \ref{th:mirror_} will be called the \emph{special} grading.

\subsection{Pants cover}
\label{sec:pantscover}
\subsubsection{The general case}
For use below we summarize some results from \cite{Abouzaid}, \cite{Mikhalkin}, \cite[\S 6.1.3]{GammageShende}.
We place ourselves in the setting of \S\ref{sec:toricprelims} and we will freely use notations from there. However
we will make our notations more explicit by writing $f^{-1}_{A,\nu,\alpha,t,1}(0)$ for $f^{-1}_{t,1}(0)$.
We assume throughout $t\gg 0$. We will sometimes need to choose appropriate cutoff functions but this is harmless as explained in Remark \ref{rem:cutoff}.

A \emph{translated P-pants} is a Laurent hypersurface in $\CC^{\ast k}$ of the form
\begin{equation}
\label{eq:Ppants}
\beta_0 t^{-\nu_0}x^{b_0}+\cdots+\beta_k t^{-\nu_k}x^{b_k}=0
\end{equation}
such that $\beta_i\neq 0$ for all $i$ and such that the affine span of $B=\{b_0,b_1,\ldots,b_k\}$ is a translated cofinite lattice in $\ZZ^k$.  A translated $\Pscr$-pants can be \emph{tailored} as in \S\ref{sec:tailoring}. 

If $B$ is a maximal face of $\Delta_\nu$ then we may consider  
\[
O_B=  \{x\in f^{-1}_{A,\nu,\alpha,t,1}(0)\mid \tau(\Log_t x)\subset B\}
\]
We may define suitable cut-off functions $\chi_{b_i}$ such that 
\[
O_B=\{x\in f^{-1}_{B,\nu_B,\beta,t,1}(0)\mid \tau(\Log_t x)\in B\}
\]
where the $(\nu_B,\beta)$ are the restrictions of $(\nu,\alpha)$. 
We clearly have
\begin{equation}
\label{eq:pantscover}
 f^{-1}_{t,1}(0)=\bigcup_{B\in \Delta_\nu^{\text{facets}}} O_B
\end{equation}
Moreover
$O_B$ is a deformation retract of $f^{-1}_{B,\nu_B,\beta,t,1}(0)$.
\begin{definition} \label{def:pantscover}
We call \eqref{eq:pantscover} the \emph{(truncated) $\Pscr$-pants cover} of $f^{-1}_{A,\nu,\alpha,t,1}(0)$.
\end{definition}
A special case of translated $\Pscr$-pants are ordinary translated pants
which are Laurent hypersurfaces in $\CC^{\ast k}$ of the form
\[
-1+t^{-\nu_1} x_1+\cdots+t^{-\nu_k} x_k=0
\]
In other words they are equal to $f^{-1}_{E_k,\nu,\epsilon,t,1}(0)$ with $E_k:=\{0\}\cup \{e_i|1\leq i\leq k\}$, where the $e_i$ are the standard basis vectors in $\ZZ^k$ and where 
$\epsilon_0=-1$, $\epsilon_i=1$ for $i>0$.

For future use we note that if in  \eqref{eq:Ppants} we have $\beta_0=-1$, $\beta_i=1$ for $i>0$ and $b_0=0$ then 
 with a suitable choice of cut-off functions $(\chi_{b_i})_i$ there is
a commutative diagram
\begin{equation}
\label{diag:etalemappants}
\begin{tikzcd}
\CC^{\ast k} \ar[rr, "{\iota}"]\ar[dr, "{f_{B,\nu,\beta,t,1}}"']&&\CC^{\ast k}\ar[dl, "{f_{E_k,\nu,\epsilon,t,1}}"]\\
& \CC &
\end{tikzcd}
\end{equation}
where the top row is the \'etale map
\begin{equation}
\label{eq:etalemappants}
\iota:\CC^{\ast k}\r  \CC^{\ast k}:x\mapsto (x^{b_1},\ldots, x^{b_{k}})
\end{equation}
\subsubsection{The star-shaped case}
Now assume furthermore that the triangulation $\Delta_\nu$ is
star-shaped
and that $\CC^{\ast k}$ is equipped with a smoothed homogeneous 
potential~$\psi$ of degree two adapted to~$C_0$ (see Definition 
\ref{def:adapted}). Since $\alpha$ is now constant by Convention \ref{conv:starshaped}, we drop it from the notation.
The $\Pscr$-pants cover corresponds to the facets of $\Delta_\nu$ and these in turn  correspond to
the vertices of $C_0$. If we write $C_{0,B}$ for $C_0$ of the tropical amoeba of $f^{-1}_{B,\nu_B,t,1}(0)$ then $C_{0,B}$ is the cone spanned by~$C_0$ at the vertex corresponding to $B$.
It follows from Lemma \ref{lem:extremal} that $\psi$ is adapted to all $C_{0,B}$.

In the star shaped case $\Skel(f^{-1}_{t,1}(0))$ was computed
in \cite{GammageShende} and \cite{Zhou} and it is compatible with the
$\Pscr$-pants cover.  In other words if we write
\[
\Skel(f^{-1}_{t,1}(0))=\bigcup_{B\in \Delta_\nu^{\text{facets}}} O_B\cap \Skel(f^{-1}_{t,1}(0))
\]
then
\[
 O_B\cap \Skel(f^{-1}_{t,1}(0))=O_B\cap \Skel(f^{-1}_{B,\nu_B,t,1}(0))
\]
and furthermore $O_B\cap \Skel(f^{-1}_{B,\nu_B,t,1}(0))$ is a deformation retract of $O_B$.

\subsection{The special grading and the natural grading coincide}
Let $\gamma$ be the difference between the two gradings. According to Lemma \ref{lem:ob} $\gamma\in H^1(f^{-1}_{t,1}(0),\ZZ)$.
We have to prove $\gamma=0$. This will follow from the following three claims.
\begin{enumerate}
  \item\label{item:pop1} The $\Pscr$-pants cover \eqref{eq:pantscover} yields an inclusion
\[
H^1(f^{-1}_{t,1}(0),\ZZ)\hookrightarrow \bigoplus_{B\in \Delta^{\text{facets}}_\nu} H^1(O_B,\ZZ)\cong  \bigoplus_{B\in \Delta^{\text{facets}}_\nu} H^1(f^{-1}_{B,\nu_B,t,1}(0),\ZZ)
\]
 which sends $\gamma$ to $(\gamma_B)_B$ where $\gamma_B$ represents the difference between the
special grading and the natural grading on $f^{-1}_{B,\eta_B,t,1}(0)$ (see Remark \ref{rem:better_restriction} below for a better result).
  \item\label{item:Ppop} The class $\gamma$ on a translated tailored  $\Pscr$-pants is induced by pullback of the corresponding translated tailored pants via the map $\iota$ in \eqref{eq:etalemappants}.
  \item\label{item:pop_triv} The class $\gamma$ is trivial on translated tailored pants. 
\end{enumerate}
We now prove these claims.
\begin{enumerate}
\item[\eqref{item:pop1}]
This follows by  induction on the number of $\Pscr$-pants and the Mayer-Vietoris sequence for the gluing of a new $\Pscr$-pants with the union of the earlier ones.
\item[\eqref{item:Ppop}] We have to be bit careful here since the pullback of $\iota$ does not preserve the ambient canonical grading on $\CC^{\ast k}$. However as we are only concerned
with the difference of two gradings, this does not matter.
By Lemma \ref{lem:liouville} we may put a potential, which is homogeneous of degree two outside a compact set, on the target of $\iota$ such that
$\iota$ is an \'etale map of Liouville manifolds. Hence the pullback of the ambient Liouville vector field is the ambient Liouville vector field. Similarly the vector field
$\xi$ in Definition \ref{def:lemfibers} is compatible with pullback via  \eqref{diag:etalemappants}.
Hence \eqref{item:Ppop} follows.
\item[\eqref{item:pop_triv}] Let $P_{k}=f_{t,1}^{-1}(0)$ be a translated tailored pants  (of real dimension $2k-2$)
with equation (writing $\chi_i(x)$ for $\chi_{b_i}(\Log_t x)$):
\[
-\chi_0(x)+t^{-\nu_1}\chi_{1}(x) x_1+\cdots+ t^{-\nu_k} \chi_{k}(x)x_k=0
\]
(recall that by convention here $\nu_i>0$ for $i=1,\ldots,k$)
and let $S$ be its skeleton. 
By Lemma
  \ref{lem:arglemma} it suffices to show that 
  $S \to S^1: m \mapsto \operatorname{arg}(d
  f_{t,1}(Z_m))$ induces the trivial map after applying $H_1(-,\ZZ)$ (recall that this makes sense by Lemma \ref{prop:GSpushout}).

To start we have to understand $H_1(S,\ZZ)=H_1(P_{k},\ZZ)$. We note that by the description of the skeleton in \cite[Proposition 5.5]{Zhou} 
$S$ contains loops~$\gamma_i$ corresponding to the minima of $\psi$ on the $k$ coordinate half-lines in~$\partial C_0$.  
Concretely near such a minimum the equation of $O_k$ is given by 
\begin{equation*}
t^{-\nu_1}x_1+\cdots + t^{-\nu_{i-1}} x_{i-1}
+  t^{-\nu_{i+1}} x_{i+1}+ \cdots + t^{-\nu_k} x_k=1
\end{equation*}
and $\gamma_i$ is defined by:
\[
\gamma_i:S^1\r O_k:\theta\mapsto (y_1,\ldots,y_{i-1},y_ ie^{i\theta} ,y_{i+1},\ldots,y_k)
\]
for some fixed $(y_1,\ldots,y_k)\in \RR^k_{>0}$.

The $(\gamma_i)_i$ generate $H_1(P_{k},\ZZ)$ but for our purposes it is sufficient that they generate up to torsion. This clear for $k>2$ since as by Lemma \ref{lem:claim} below the map $H_1(f^{-1}_{t,1}(0),\ZZ)\r H_1(\CC^{\ast k},\ZZ)$
is an isomorphism and the $\gamma$'s clearly generate $H_1(\CC^{\ast k},\ZZ)$. For $k=2$ we find that the $\gamma$'s are at least independent so it is sufficient to check
that $\rk H_1(O_k,\ZZ)=2$. This follows from the fact that a two-dimensional  pair of pants is a thrice punctured sphere.

\medskip

Using \eqref{eq:psigradient} we get on the image of $\gamma_i$:
\begin{align*}
df_{t,1}(Z_{\gamma_i(\theta)})&=Z_{\gamma_i(\theta)}(f_{t,1})\\
&=\sum_j\log |x_j|(x_j\partial_{x_j}+\bar{x}_j \partial_{\bar{x}_j})(f_{t,1})\biggr|_{x=\gamma_i(\theta)}\\
&=t^{-\nu_1} \log |y_1| y_1+\cdots + t^{-\nu_{i-1}} \log |y_{i-1}|  y_{i-1}\\
&+  t^{-\nu_{i+1}}  \log |y_{i+1}| y_{i+1}+ \cdots + t^{-\nu_k}  \log |y_k| y_k
\end{align*}
The fact that this is constant finishes the proof.
\end{enumerate}

We used the following lemma.
\begin{lemma}\label{lem:claim}
The restriction maps $H_1(f^{-1}_{t,1}(0),\ZZ)\r H_1(\CC^{\ast k},\ZZ)$, $H^1(\CC^{\ast k},\ZZ)\r H^1(f^{-1}_{t,1}(0),\ZZ)$
are isomorphisms.
\end{lemma}

\begin{proof}
By the universal coefficient theorem, it suffices to prove the lemma for homology. 
Let $F=\bigcup_{s\in [0,1]} f_{t,s}^{-1}(0)\subset \CC^{\ast k}\times [0,1]$. This is a trivial family by \cite[Proposition 4.9]{Abouzaid}.
For $s\in [0,1]$ we have the following diagram of pushforward maps
\[
\xymatrix{
H_1(F,\ZZ)\ar[r]\ar[d]&H_1(\CC^{\ast k}\times [0,1],\ZZ)\ar[d]\\
H_1(f^{-1}_{t,s}(0),\ZZ)\ar[r]& H_1(\CC^{\ast k},\ZZ).\\
}
\]
The vertical maps are isomorphisms and by \cite[Corollary (1.1.1)]{Oka} the bottom horizontal map is an isomorphism for $s=0$ (since $k>2$ and $f^{-1}_{t,0}(0)$ is defined by a Laurent polynomial). 
Hence  the upper horizontal map (which does not involve $s$) is an isomorphism. It follows that the lower horizontal map is an isomorphism for all $s$, in particular for $s=1$.
Thus the claim is proved.
\end{proof}
\begin{remark}
\label{rem:better_restriction}
Although we do not use it, it is interesting to observe that for $k>2$ a grading on $f^{-1}_{t,1}(0)$ is in fact determined by the restriction to a single arbitrary truncated tailored translated $\Pscr$-pants $O_B$.

If $B$ is a maximal simplex in $\Delta_\nu$ then we have, using Lemma \ref{lem:claim}, a commutative diagram.
\[
\begin{tikzcd}
& H^1(f^{-1}_{t,1}(0),\ZZ)\ar[rd]&\\
H^1(\CC^{\ast k},\ZZ) \ar[rr,"\cong"']\ar[ur, "\cong"] && H^1(O_B,\ZZ)
\end{tikzcd}
\]
It follows that the right most diagonal map is an isomorphism as well.
\end{remark}

\section{HMS action of the fundamental group of the SKMS}
\label{sec:mainresult}
For a commutative ring let $\Ho_R(A_\infty)$ be the homotopy category of $R$-linear
$\ZZ$-graded $A_\infty$-cat\-egories, i.e.\ the category
or $R$-linear $\ZZ$-graded $A_\infty$-categories with $A_\infty$-equivalences formally
inverted \cite{CanonacoOrnaghiStellari}.

If $M$ is a Liouville manifold equipped 
  with a grading and background class then $\WF(M)$ denotes the
  corresponding wrapped Fukaya category as introduced in
  \cite{GPS1}. In particular $\Wscr\Fscr(M)$ is an object in $\Ho_\ZZ(A_\infty)$.

\begin{theorem}
\label{th:mainth1}
Let ${\alpha}\in \CC^{ A}- V({A})$. 
 Let $f^{-1}_{{\alpha}}(0)$ be equipped with the restriction of the toric K\"ahler structure on $(\CC^\ast)^k$  as in \S\ref{subsec:families} and hence in particular with the
corresponding symplectic/Liouville forms. We equip $f^{-1}_\alpha(0)$ with the trivial background class and the natural grading as a hypersurface in a torus (see \S\ref{sec:gradingbackgroundintro}).
There exists a natural action of $\pi_1(\Kscr_{X_\Sigma},\alpha)$ on $\Wscr\Fscr(f^{-1}_{{\alpha}}(0))$ considered as an object in $\Ho_\ZZ(A_\infty)$.
\end{theorem}
\begin{proof}
The proof consist of several steps where we use the results from previous sections. Throughout we silently use Theorems \ref{lem:simpiso} and \ref{lem:bgisotopies}
to ensure that our constructions are compatible with the gradings and the background classes.

\begin{enumerate}
	\item\label{critical1} 
From Theorem \ref{prop:Dlocally} we obtain that $(f_{{\alpha}}^{-1}(0))_{\alpha\in \CC^{{ A}}-V({A})}$ is a family of Liouville manifolds which locally has families of generating Liouville domains.

	\item\label{isomorphism}
          Let ${\gamma}:[0,1]\r \CC^{{A}}- V({A})$ be a  path. 
          By \eqref{critical1} we may apply the Moser lemma in this setting, Theorem \ref{lem:moser2}, 
        to obtain an isotopy  $\rho_{\gamma,t}:{f}_{{\gamma}(0)}^{-1}(0)\to {f}_{{\gamma}(t)}^{-1}(0)$, $t\in [0,1]$, of Liouville manifolds. 
	\item\label{action} In particular if $\gamma(0)={\gamma}(1)$ then
        we obtain an automorphism $\rho_{\gamma,1}$ of the Liouville manifold ${f}_{{\gamma}(0)}^{-1}(0)$. 
Using Proposition \ref{cor:functorial} we obtain a corresponding $A_\infty$-auto-equivalence of $\Wscr\Fscr(f_{{\gamma}(0)}^{-1}(0))$.
	\item\label{independent} The autoequivalence in \eqref{action} depends on the actual path $\gamma$ (and not only its homotopy class) and the chosen isotopy $\rho_{\gamma,t}$. However
 it follows from Corollary \ref{cor:independence2} that different choices give rise to $A_\infty$-homotopic auto-equivalences. 
We obtain an action of $\pi_1(\CC^{A}-V(A),\alpha)$ on the wrapped Fukaya category $\Wscr\Fscr(f_{{\alpha}}^{-1}(0))$ considered as
an object in $\Ho_{\ZZ}(A_\infty)$.
\item
Thanks to the presentation \eqref{eq:presentation} of $\pi_1(\Kscr_A,\alpha)$
we have to prove that the image of $\pi_1((\CC^\ast)^k,1)$ acts trivially on   $\Wscr\Fscr(f_{{\alpha}}^{-1}(0))$. We have $\pi_1((\CC^\ast)^k,1)=\pi_1((S^1)^k,1)$ and the $(S^1)^k$ action on $(\CC^\ast)^k$ preserves
the K\"ahler form (see \eqref{eq:omega}).  Hence the action of the elements of a path $\gamma:[0,1]\r (S^1)^k $ with $\gamma(0)=\gamma(1)=1$ (i.e.\ $\gamma \in \pi_1((S^1)^k,1)$) yields an isotopy of Liouville manifolds $\rho_{\gamma\cdot {\alpha},t}:
f_{{\alpha}}^{-1}(0)\r f_{\gamma(t)\cdot {\alpha}}^{-1}(0)$. But this isotopy is such that $\rho_{\gamma\cdot {\alpha},1}$ is the identity (as it is the action by the identity element of $(S^1)^k$). It follows that the image of $\pi_1((\CC^\ast)^k,1)$ under the first
arrow in \eqref{eq:presentation} does indeed act trivially on $\Wscr\Fscr(f^{-1}_{\alpha}(0))$. This yields the desired result. \qedhere
          \end{enumerate}
\end{proof}
\begin{theorem} \label{prop:lastmile} 
We assume $0\not \in \Pi_\nu$. 
Let 
  ${\alpha}\in \CC^{A}-V({A})$. Let
  $f_{t,0}^{-1}(0)=f_{{\alpha}}^{-1}(0)$, resp. $f_{t,1}^{-1}(0)$,   be equipped with the restriction of the  K\"ahler structure coming from a toric, resp.\ any smoothed homogeneous degree $2$, potential on $(\CC^\ast)^k$  as in \S\ref{sec:paistropical}\footnote{For this we need to impose the assumption $0\not \in \Pi_\nu$.} and hence in particular with the
corresponding symplectic/Liouville forms. Equip  $f^{-1}_\alpha(0)$ with the   $\ZZ$-grading and (trivial) background class as in Theorem \ref{th:mainth1}. 
Similarly equip
  $f^{-1}_{t,1}(0)$ 
with the grading and (trivial) background class as in Theorem \ref{th:mirror_}.
Then there is an isomorphism
  $\Wscr\Fscr(f_\alpha^{-1}(0))\cong
  \Wscr\Fscr(f_{t,1}^{-1}(0))$ in $\Ho_\ZZ(A_\infty)$.
\end{theorem}
\begin{proof}
  We use a similar but somewhat more involved argument as in the proof of Theorem \ref{th:mainth1}. As before, throughout we silently use Theorems \ref{lem:simpiso} and \ref{lem:bgisotopies}.
          The steps are as follows. 
\begin{enumerate}
    \item\label{critical2}

Let $\psi_0$, $\psi_1$ be the toric, resp. the smoothed homogeneous degree $2$ potential from the statement of the theorem. Set $\psi_s=s\psi_1+(1-s)\psi_0$, which is still a smoothed homogeneous degree $2$ potential by Lemma \ref{lem:interpolation}, and let $\theta_s$ be the corresponding Liouville form. From Theorem \ref{prop:tropical} we obtain that 
for $t\gg 0$, the family $(f^{-1}_{t,s}(0),\theta_s)_{s\in [0,1]}$ is a family of Liouville manifolds which locally has
families of generating Liouville domains.

  \item\label{FWFSigma}
Applying Moser's lemma, Theorem \ref{lem:moser2}, 
 we obtain an isomorphism of Liouville manifolds  $(f^{-1}_{t,0}(0),\theta_0)\cong ({f_{t,1}}^{-1}(0),\theta_1)$. 
 Since $t\gg 0$ we have ${\alpha}_\nu:=(\alpha_i t^{-\nu(a_i)})_i\not\in V({A})$ by Lemma \ref{lem:nui}. 
Choosing an arbitrary path  in $(\CC^\ast)^A- V(A)$ connecting ${\alpha}$ and 
${\alpha}_\nu$ we obtain from \eqref{isomorphism} in the proof of Theorem \ref{th:mainth1} an isomorphism of Liouville manifolds (with the Liouville structure coming from the toric potential)
$f^{-1}_{t,0}(0)=f^{-1}_{{\alpha}_\nu}(0)\cong f^{-1}_{{\alpha}}(0)$.  Combining the two isomorphisms yields by Proposition \ref{cor:functorial}  an $A_\infty$-equivalence 
$\WF({f}_{{\alpha}}^{-1}(0),\theta_0)\cong\WF({f}_{t,1}^{-1}(0),\theta_1)$. \qedhere
\end{enumerate}	
\end{proof}

Hence combining Lemma \ref{rem:easysmooth}, Theorem \ref{th:mainth1}, Theorem \ref{prop:lastmile} and Theorem \ref{th:mirror_} we obtain:
\begin{corollary}
\label{cor:maincor}
Assume $0\in A\subset \ZZ^k$, $|A|>1$, and assume that $P=\conv(A)$ is equipped with a star-shaped triangulation $\Tscr$ with vertices $A$ as in \S\ref{sec:stacky}.
Denote the associated stacky fan by~$\Sigma$.
Assume that a strictly convex function $\nu:A\r \RR$ exists whose associated triangulation coincides with $\Tscr$.
Under these hypotheses there is a natural action of $\pi_1(\Kscr_{A})$ on  $D^b(\coh(\partial X_\Sigma))$ considered as an object in $\Ho_\CC(A_\infty)$.
\end{corollary}
\begin{remark}\label{rem:quasi-projective} The existence of $\nu$ is equivalent with $X_\Sigma$ being quasi-projective. This follows from \cite[Propositions 14.4.1, 14.4.9, 15.2.9 (and 7.2.9)]{CoxLittleSchenck}. 
\end{remark}

\section{The Grothendieck group of a toric stack and its boundary}
For notations regarding stacky fans see \S\ref{sec:stacky}. We will prove the following result. 
\begin{theorem}
\label{eq:K0boundary}
Assume that $\Sigma$ is a pure full dimensional simplicial stacky fan and put $X=X_\Sigma$. Furthermore let $i:\partial X\hookrightarrow X$ be the inclusion.  Then we have an exact sequence
\begin{equation}
\label{eq:K0ses}
0\r X(T)\xrightarrow{\partial} K_0(\coh(\partial X))\xrightarrow{i_\ast} K_0(X)\xrightarrow{\rk} \ZZ\r 0
\end{equation}
where $\partial(\chi)$ for $\chi\in X(T)$ is defined as follows. The character $\chi:T\r \CC^\ast$ defines a unit in $\Oscr(T)$. The map
$\chi:\Oscr_T\r \Oscr_T:f\mapsto \chi f$ can be extended to a map $\tilde{\chi}:I\r \Oscr_X$ where $I\subset \Oscr_X$ is such that the support of $\Oscr_X/I$ is contained in $\partial X$. Then
$\partial(\chi)=[\coker \tilde{\chi}]-[\Oscr_X/I]\in K_{0}(\coh\nolimits_{\partial X}(X))\cong K_0(\coh(\partial X))$.
\end{theorem}
\begin{lemma} The map $\partial$ in \eqref{eq:K0ses} is injective.
\end{lemma}
\begin{proof}
  Let $\sigma$ be a cone in $\Sigma$. By considering the restriction for $X_\sigma \r X$ we see that we may assume that $\Sigma$ is a maximal simplicial cone. Hence we may assume that 
$X=\AA^n/G$ where $G$ is a finite abelian group acting diagonally on $\AA^n$ and $T=\TT/G$ where $\TT=(\CC^\ast)^n$ acts on $\AA^n$ via  the canonical coordinates $(x_1,\ldots,x_n)$. 
We write $\bar{Y}$ for the coarse moduli space of a DM-stack $Y$.
We put $X_i=\{x_i=0\}/G$ and we
let $\eta_i$ be the generic point of $\bar{X}_i$. We obtain additive functions $l_i:K_0(\coh(\partial X))\r \ZZ$ which send $[\Fscr]$ to the length\footnote{We can talk about length
since coherent sheaves on $\eta_i\times_{\bar{X}} X$ have finite length.} of $ \Fscr_{\eta_i}$. We claim that the composition
$X(T)\r K_0(\coh(\partial X))\xrightarrow{(l_i)_i} \ZZ^n$ is injective. To see this let $\chi_i$ be a non-zero $T$-weight of $\Oscr(\bigcap_{j\neq i}\bar{X}_j)$ with corresponding weight vector~$m_i\in \Oscr(\bar{X})$.
Then $l_i(\partial(\chi_i))=\length_{\eta_i}(\Oscr_X/m_i\Oscr_X)>0$ (to see the first inequality note that in this case one may choose $I=\Oscr_X$) whereas $l_j(\partial(\chi_i))=0$ for $j\neq i$. This proves what we want.
\end{proof}
Recall that if $R$ is a commutative ring then there is a canonical map $R^\ast\r K_1(R)$ which is an isomorphism if $R$ is a field.
\begin{lemma} \label{lem:K1torus} Let $T$ be an algebraic torus. Then the map $\Oscr(T)^\ast\r K_1(T)$ is an isomorphism. 
\end{lemma}
\begin{proof} 
  We use induction on $\dim T$. If $T$ is zero dimensional then the claim follows from the fact that $\Oscr(T)=\CC$.
Otherwise write $T=T'\times \CC^\ast$. Then we have $\Oscr(T)=\Oscr(T')[z,z^{-1}]$. Since $\Oscr(T')$ is regular, by \cite[Theorem 3.6]{MR3076731} we have 
an isomorphism 
\begin{equation}
\label{eq:weibeliso}
K_1(\Oscr(T'))\oplus K_0(\Oscr(T'))\r K_1(\Oscr(T')[z,z^{-1}]).
\end{equation} The first component is obtained from the inclusion $\Oscr(T')\r \Oscr(T')[z,z^{-1}]$ and the second component
sends $P$ to $Pz$ (where $z\in \Oscr(T')[z,z^{-1}]^\ast$ is viewed as a class in $K_1(\Oscr(T')[z,z^{-1}])$). By induction we have $K_1(\Oscr(T'))=\Oscr(T')^\ast$ and we also have $K_0(\Oscr(T'))=\ZZ$. Under the second component of \eqref{eq:weibeliso}
$\ZZ\cong K_0(\Oscr(T'))\r K_1(\Oscr(T')[z,z^{-1}])$, $1$ is sent to $z$ and hence $n=\underbrace{1+\cdots+1}_n$ is sent to $\underbrace{z\cdots z}_n=z^n$. Thus if we identify 
$\Oscr(T')^\ast \oplus \ZZ$ with $\Oscr(T)^\ast$ via $(u,n)\mapsto uz^n$ then \eqref{eq:weibeliso} corresponds to the canonical map $\Oscr(T)^\ast\r K_1(T)$. This finishes the induction step and the proof.
\end{proof}
\begin{proof}[Proof of Theorem {\ref{eq:K0boundary}}] According to \cite[Theorem 10.2]{MR179229} there is an exact sequence\footnote{Note that \cite{MR179229} does not use Quillen K-theory.
However $K_1(T)$ as defined in loc.\ cit.\ coincides with Quillen K-theory because of \cite[Propositions 8.6, 8.9]{MR179229}.}
\[
K_1(T)\xrightarrow{\partial} K_0(\coh\nolimits_{\partial X}(X))\r K_0(X) \r K_0(T)\r 0.
\]
This translates immediately into the sequence \eqref{eq:K0ses}, except for the left most map.
By Lemma \ref{lem:K1torus} we have $K_1(T)=\Oscr(T)^\ast=\CC^\ast\oplus X(T)$. Using the formula in \cite[Proposition 7.5]{MR179229} one checks that $\partial$ restricted to $X(T)$ coincides
with the description we have given in the statement of the theorem. A similar verification shows that $\partial$ restricted to $\CC^\ast$ is zero. This finishes the proof.
\end{proof}
To complete the discussion we give some further information on $K_0(X)$ in the setting of Theorem \ref{eq:K0boundary}. If $\sigma$ is a maximal cone then we put $c(\sigma):=|N/N_\sigma|$ where
$N_\sigma$ is the lattice spanned by the $v_\tau$ for the one-dimensional cones $\tau$ contained in $\sigma$. If $\Sigma$ is smooth then $c(\sigma)=1$ for all $\sigma$.
\begin{theorem}[{\cite{BorisovHorja,Hua}}] \label{thm:BHH} Assume that $\Sigma$ is a pure full-dimensional simplicial stacky fan and  put $X=X_\Sigma$. 
In each of the following cases:
\begin{enumerate}
\item $\Sigma$ is a complete fan;
\item $\Sigma$ is a triangulation of a cone;
\end{enumerate}
$K_0(X)$ is a finitely generated free abelian group whose rank is $\sum_\sigma c(\sigma)$ where the sum runs over the maximal cones in $\Sigma$.
\end{theorem}
\begin{proof}
The conditions given in the theorem imply that the Stanley-Reisner ring \cite[Definition 5.1.2]{BrunsHerzog} of $\Sigma$ is Cohen-Macaulay by \cite[Theorem 5.1.13, Corollary 5.3.9]{BrunsHerzog}. It then follows
as in \cite[Theorem 2.2]{Hua} that $K_0(X)$ is a free $\ZZ$-module.

From \cite[Theorem 5.5, Remark 3.11]{BorisovHorja} we obtain that the rank of $K_0(X)$ is the evaluation at $1$ of the fractional polynomial 
\begin{equation}
\label{eq:polynomial}
(1-t)^{\rk N} \sum_{n \in |\Sigma|\cap N} t^{\deg n}
\end{equation}
where $\deg n$ is defined as follows. Let $\sigma$ be a cone in $\Sigma$ which contains $n$. Then $n=\sum_{\tau} n_\tau v_\tau$ where $n_\tau\in \QQ$ and $\tau$ runs over the one-dimensional cones in $\sigma$. We put $\deg n=\sum_\tau n_\tau\in \QQ$ (this does not depend on $\sigma$).

It is not hard to see that the lower dimensional cones do not contribute to \eqref{eq:polynomial}, after setting $t=1$.
Hence it is sufficient to understand the case that $\Sigma$ consists of a single full-dimensional simplicial cone $\sigma$. In that case \eqref{eq:polynomial} becomes a sum over cosets of $N_\sigma$ in $N$, each
coset giving the same contribution after setting $t=1$. So finally we have to understand the sum
\[
(1-t)^{\rk N} \sum_{n \in |\Sigma|\cap N_\sigma} t^{\deg n}
\]
which is clearly equal to one.
\end{proof}
\begin{remark} Some restrictions on $\Sigma$ are required for the
  conclusions of Theorem \ref{thm:BHH} to hold, even if $X$ is a
  scheme. For example $K_0(X)$ is not always torsion free (see the
  example by Lev Borisov given in \cite[Example 4.1]{Hua}). Moreover
  the rank of $K_0(X)$ is not always given by the maximal
  cones. For example, consider the case $N=\ZZ^2$ and $\Sigma$ consists of two
  cones, spanned respectively by $\{(1,0),(0,1)\}$ and
  $\{(-1,0),(0,-1)\}$. Then $X$ is the complement of two points in
  $\PP^1\times \PP^1$. Since these points have the same class in
  $K_0(\PP^1\times \PP^1)\cong \ZZ^4$ we conclude that $\rk K_0(X)=3$
  whereas Theorem \ref{thm:BHH} would have predicted $\rk K_0(X)=2$.
\end{remark}

\section{Induced HMS: singular cohomology vs Grothendieck group}
If $\Ascr$ is a $A_\infty$-category then $\Two(\Ascr)$ denotes is pretriangulated closure of $\Ascr$ 
(this might not be idempotent complete). 
In our notations we sometimes write $\Ascr$ when strictly speaking we mean $H^0(\Ascr)$. As before $\Dscr(\Ascr)$ is the Karoubian closure of $\Two(\Ascr)$. 
\subsection{Partially wrapped Fukaya categories of Liouville pairs}
Let $(M,F)$ be a Liouville pair (see \S\ref{sec:restriction}).
For $Y\subset \partial M$ let $Y_\infty \subset \partial_\infty \hat{M}$\footnote{For a Liouville manifold $(N,\theta)$ we define $\partial^\infty N$ as the set of unbounded $Z_\theta$-orbits. As a topological space $\partial^\infty N=(N-\Skel(N))/(\text{Liouville flow})$.} be the
image of $Y$ under the Liouville flow. 
Then we write\footnote{Sometimes the Liouville pair is defined to be $(\hat{M},F_\infty)$ instead of $(M,F)$.} $\Wscr\Fscr(\hat{M},F_\infty)$ for the partially wrapped Fukaya category of $\hat{M}$ with a stop given
by $F_\infty$ \cite[\S1.1]{GPS2}.  We summarize some facts about partially wrapped Fukaya categories we use below.
\begin{enumerate}
\item[(F1)]\label{it:F1} By \cite[(1.3)]{GPS2} the natural\footnote{By a natural functor we mean a functor that is the identity on objects.} functor
\[
  \Wscr\Fscr(\hat{M},F_\infty)\r \Wscr\Fscr(\hat{M},\Skel(F)_\infty)
\]
is a  quasi-equivalence.
\item[(F2)]
The Reeb flow on $\partial M$ is transversal to $F$. If $B=[\epsilon,\epsilon]\times F\subset
\partial M$ is a corresponding thickening of $F$ then (a smoothing of)
\[
\tilde{M}:=\hat{M}-\operatorname{int}(B) \times \RR_{>0}
\]
(where the embedding $B \times \RR_{>0}\subset \hat{M}$ is obtained by applying the Liouville flow to $B$)
defines a \emph{Liouville sector} with boundary $\partial \tilde{M}\cong\hat{F}\times \RR$ (see \cite[Lemma 2.13]{GPS1}).
The natural functor 
\begin{equation}
\label{eq:qeev}
\Wscr\Fscr(\tilde{M})\r \Wscr\Fscr(\hat{M}, F_\infty)
\end{equation}
is also a quasi-equivalence (see \cite[(1.2), Corollary 3.9]{GPS2}).
\item[(F3)]
There is a standard functor
\begin{equation}
\label{eq:cup1}
\cup : \Wscr\Fscr(\hat{F})\r \Wscr\Fscr(\tilde{M})
\end{equation}
(see \cite[\S7.3, (7.12)]{GPS3}) whose construction is as follows. 
Let $V$ be the sub-Liouville sector of 
$\CC_{\Re \ge 0}$ of the form
\[
V=\{z\in \CC_{\Re \ge 0}\mid \Re z\le \epsilon\}\cup \{ z\in \CC_{\Re \ge 0}\mid |\arg \theta|\ge \theta_0\}
\]
for $\epsilon>0$, $\theta_0\in ]0,\pi/2[$ (see \eqref{Uchoice} below).  
We assume that the Liouville structure on $V$ is deformed in such a way that $V$ is exact (see \cite[Proposition 2.25]{GPS1})
so that $V\times \hat{F}$ is (after smoothing) the product of the Liouville sector $V$ and the Liouville sector $\hat{F}$.
It follows from \cite[Proposition 2.25]{GPS1} that there is an embedding $h:V\times \hat{F}\hookrightarrow \tilde{M}$ for suitable $\epsilon, \theta_0$. Choose a conical exact Lagrangian~$l$ in $V$ as
in \eqref{Uchoice} below. 
Then
 $\cup(L)$ for a conical exact Lagrangian in $\hat{F}$ is given by the image of $l\tilde{\times} L$ under $h$ where $\tilde{\times}$ is the conicalized product defined in \cite[\S7]{GPS2}.

\item[(F4)]\label{it:F4}
We may interpret \eqref{eq:cup1} as a functor
\[
\cup : \Wscr\Fscr(\hat{F})\r \Wscr\Fscr(\hat{M},\Skel(F)_\infty)
\]
via the quasi-equivalence \eqref{eq:qeev}. In  case $\hat{F}$ is Weinstein, $\cup$ sends  the closed cocores\footnote{By cocores we mean cocores that are Lagrangian, i.e. whose dimension is 
half the dimension of the symplectic manifold. In loc.cit. these are called ``cocores of critical handles''.} in $\hat{F}$ to the corresponding linking disks in 
$\Wscr\Fscr(\hat{M},\Skel(F)_\infty)$ 
(see \cite[1st paragraph of \S1.4]{GPS2}). 
\item[(F5)]\label{it:F5}
If $\hat{F}$ is Weinstein and all  cocores are properly embedded then $\Two(\Wscr\Fscr(\hat{F}))$ is generated by the  cocores
\cite[Theorem 1.13]{GPS2}.
\end{enumerate}
\subsection{The Lazarev map}
\def\LC{\operatorname{L}_\CC}
\label{sec:lazarev}
We recall the following\footnote{In loc. cit. this result is actually stated for Weinstein domains. However one checks that for a Liouville domain $N$, 
the functor  $\Wscr\Fscr(N)\r \Wscr\Fscr(\hat{N})$ which sends a Lagrangian with Legendrian boundary to its conical completion  is quasi-fully faithful and essentially surjective. To see this last fact one may deform the Weinstein structure a bit such that  $\Wscr\Fscr(\hat{N})$ is generated by cocores (see (F5))}. 
\begin{theorem}[{\cite[Theorem 1.4]{Lazarev}}]\label{th:lazarev}
Let $M$ be a either a $2d$-dimensional  Liouville manifold which is isomorphic (cfr \S\ref{sec:liouville}) to a Weinstein manifold or else a Weinstein sector.\footnote{A Weinstein sector 
is a sector constructed from a Liouville pair $(M,F)$ as in (F2) such that $\hat{M}$ and $\hat{F}$ are both Weinstein.}
For an exact conical Lagrangian $L$ we write $c(L)\in H^{d}(M,\ZZ)$ for the cohomology class
of $L$. We use $[L]$ for the class of $L$ in $K_0(\Two(\Wscr\Fscr(M))$.
We have 
\begin{enumerate}
\item every element of $H^{d}(M,\ZZ)$ is of the form $c(L)$ for an  exact conical Lagrangian $L$ in $M$;
\item if $L_1$, $L_2$ are exact conical Lagrangians in $M$ then equality $c(L_1)=c(L_2)$ in $H^{d}(M,\ZZ)$ implies equality $[L_1]=[L_2]$ in $ K_0(\Two(\Wscr\Fscr(M))))$.
\end{enumerate}
In other words there is a well-defined surjective map 
\begin{equation}
\label{eq:lazarev}
\operatorname{L}:H^{d}(M,\ZZ)\r K_0(\Two(\Wscr\Fscr(M))):c(L)\mapsto [L] \qedhere
\end{equation}
\end{theorem}
\begin{proof} This result for a Liouville manifold is stated in loc.\ cit.\ in the case that $M$ is Weinstein.\footnote{Theorem in loc.cit. is stated for $\ZZ_2$-coefficients and $\ZZ_2$-grading. However, in the paragraph after Theorem 1.8 in loc.cit. it is stated that it holds also in the case of $\ZZ$-coefficients and $\ZZ$-grading.} 
However it is clear that (1,2) are preserved under isomorphisms of Liouville manifolds. The corresponding result for Weinstein sectors follows from \cite[Proposition 2.15]{Lazarev}.
\end{proof}
\begin{theorem}
Assume $(M,F)$ is a Liouville pair such that $\hat{M}$, $\hat{F}$ are both Weinstein with $\dim M=2d$. Then there is a commutative diagram
\begin{equation}
\label{eq:lazarev1}
\begin{tikzcd}
H^{d-1}(\hat{F},\ZZ)\arrow[d,"\operatorname{L}"']\arrow[r,"\partial"] &H^{d}(\tilde{M},\hat{F},\ZZ)\arrow[d,"\operatorname{L}"]\\
K_0(\Two(\Wscr\Fscr(\hat{F})))\arrow[r,"\cup"']& K_0(\Two(\Wscr\Fscr(\tilde{M})))
\end{tikzcd}
\end{equation}
where $\tilde{M}$ has been defined in (F2).
\end{theorem}
\begin{proof}
$\hat{M}$ is the convex completion of $\tilde{M}$ (see \cite[\S2.7]{GPS1}). It is obtained by replacing $V\times \hat{F}\subset \tilde{M}$ (described in (F3)) by $\tilde{V}\times \hat{F}$ where
$\tilde{V}:=\CC_{\Re \le 0}\cup V$.
In particular the inclusion $i:\tilde{M}\r \hat{M}$ is a homotopy equivalence. It is therefore sufficient
to prove that the following diagram is commutative
\begin{equation}
\label{biggerdiagram}
\begin{tikzcd}
H^{d-1}(\hat{F},\ZZ)\arrow[d,"\operatorname{L}"']\arrow[r,"\partial"] &H^{d}(\hat{M},\hat{F},\ZZ)\arrow[d,"\operatorname{L}\circ i^\ast"]\\
K_0(\Two(\Wscr\Fscr(\hat{F})))\arrow[r,"\cup"']& K_0(\Two(\Wscr\Fscr(\tilde{M})))
\end{tikzcd}
\end{equation}
Let $\ell \in H^{d-1}(\hat{F},\ZZ)$. By Theorem \ref{th:lazarev}(1) $\ell=c(L)$ for some exact conical Lagrangian $L$ in $\hat{F}$. 
So in particular by \eqref{eq:lazarev} $\operatorname{L}(\ell)=\operatorname{L}(c(L))=[L]$.
Then by (F3) we have $(\cup\circ \operatorname{L})(\ell)=[l\tilde{\times} L]
=
\operatorname{L}(c(l\tilde{\times} L))=
\operatorname{L}(c(l\times L))=
(\operatorname{L}\circ i^\ast)(c(l\times L))$ where in
the rightmost expression we consider $l$ to be contained in $\tilde{V}$ via the inclusion $V\subset \tilde{V}$. Hence to verify the commutativity of \eqref{biggerdiagram} it is sufficient to check
$
\partial(c(L))=c(l\times L)
$ for exact conic Lagrangians $L$ in $\hat{F}$.

It follows from the well-known description of $\partial$ that  $\partial(c(L))=c(S^1\times L)$ where $S^1$ is a small circle around the origin in $\tilde{V}$ and $S^1\times L$ is embedded in $\hat{M}$ via the inclusion
$\tilde{V}\times\hat{F}\subset \hat{M}$. So we need to prove $c(l\times L)=c(S^1\times L)$. 
To this end we choose an open subset $j_0:U\hookrightarrow V$ as in \eqref{Uchoice} below and we let
$\hat{M}_U\subset \hat{M}$ be obtained from $\hat{M}$ by replacing $\operatorname{int}(\tilde{V})\times \hat{F}
\subset \hat{M}$ by $U\times \hat{F}$. Since the inclusion $j:\hat{M}_U\r \hat{M}$ is a homotopy equivalence 
it is sufficient to show that $j^*(c(l\times L))=j^\ast(c(S^1\times L))$. Since
cohomology classes commute with pullback, this amounts to showing that
$c(j_0^{-1}(l)\times L)=c(j_0^{-1}(S^1)\times L)$. This is true because 
$j_0^{-1}(l)$ and $j_0^{-1}(S^1)$ are homotopy equivalent (as can be seen from \eqref{Uchoice}).

\begin{equation}
\label{Uchoice}
\charclass
\end{equation}
\end{proof}

\subsection{Liouville pairs in the toric setting}
\label{sec:liouvilletoric}
With the notation of \S\ref{sec:toricprelims} we put $D_\alpha=f^{-1}_\alpha(0)$ for $\alpha\not\in V(A)$.
If we are in the setting of \S\ref{sec:setting} then we also put $D^{\tl}:=f^{-1}_{t,1}(0)$. In that case we also assume
that $1\gg \epsilon_2>\epsilon_1>0$ and $t\gg 0$ so that all prior results apply.

According to Proposition \ref{prop:GSpair} the pair $((\CC^\ast)^k,D^{\tl})$ contains a Liouville pair $(M,F)$. 
Note that $\hat{M}=(\CC^{\ast})^k$.
We write $\Wscr\Fscr((\CC^\ast)^k,D^{\tl}):=\Wscr\Fscr((\CC^{\ast})^k,F_\infty)$.

According to Theorem \ref{thm:Zhou} we have an isomorphism of pairs
\begin{equation}
\label{eq:pairs}
((\CC^{\ast})^k,\Skel(F)_\infty)\cong (T^\ast(S_{N^\ast}),\Lambda_\Sigma^\infty)
\end{equation}
where $\Lambda_\Sigma^\infty$ is  the FLTZ-skeleton (see \S\ref{sec:fltz}). By Lemma \ref{lem:canonicalgrading} this isomorphism is compatible with gradings.
\subsection{Homological mirror symmetry for pairs}
Let $\Sigma$ be a toric fan. Combining \cite[Theorem 1.1]{GPS3}  with \cite[Theorem 1.2]{Kuwagaki} yields an equivalence of categories
\[
\HMS_1:\Dscr(\Wscr\Fscr(T^\ast(S_{N^\ast}),\Lambda_\Sigma^\infty)_\CC)\cong D^b(\coh(X_\Sigma)).
\]
 If $\Sigma$ is a smooth fan with $X_\Sigma$ quasi-projective then an alternative construction of such an equivalence 
has been given in \cite[Corollary D]{HanlonHicks}.\footnote{\cite{HanlonHicks} work mostly in the setting of projective varieties, however the referred result is stated in a greater generality of quasi-projective varieties.} 
We denote it $\HMS_2$. 
We will use properties of both constructions  but we will not use that they are the same as this seems to be unknown (see \cite[Remark 1.2]{HanlonHicks}).

If we are in the setting of \S\ref{sec:setting} then there is a commutative diagram \cite[Figure 6]{GammageShende}
\begin{equation}\label{eq:commutative_diagram}
\begin{tikzcd}
\Tw(\Wscr\Fscr(D^{\tl})_\CC)\arrow[d,"{\HMS}","\cong"']\arrow[r,"\cup"]& \Tw(\Wscr\Fscr((\CC^*)^k,D^{\tl})_\CC)
\arrow[d,"{\HMS_1}","\cong"']\\
D^b(\coh(\partial X))\arrow[r]&D^b(\coh(X)).
\end{tikzcd}
\end{equation}
The left  arrow is the one described in Theorem \ref{th:mirror_}. The right arrow is constructed via the isomorphism of pairs \eqref{eq:pairs}.
The top arrow is obtained from (F4). The bottom arrow is the pushforward.

\subsection{A result on idempotent completion}
Let us start with a conjecture.
\begin{conjecture}
\label{conj:idcomplete}
The canonical map
\[
\Two(\Wscr\Fscr(D_\alpha))\r \Dscr(\Wscr\Fscr(D_\alpha))
\]
is an equivalence. In other words $\Two(\Wscr\Fscr(D_\alpha))$ is idempotent complete.
\end{conjecture}
We can prove this conjecture in a particular case.
\begin{proposition} \label{prop:idcomplete}
Assume that we are in the setting of  \S\ref{sec:setting} and that
 $A$ consists of all lattice points of $P=\conv(A)$.
Then  Conjecture \ref{conj:idcomplete} is true.
\end{proposition}
\begin{proof}
Although this is a purely
symplectic statement we have no intrinsic proof for it, and our proof relies on homological mirror symmetry.

Given that $A$ is the set of vertices of the triangulation $\Tscr$ of
$P$, with all elements of $A$, except possibly zero, being contained
in $\partial P$, this implies that the cones in $\Sigma$ are smooth.
This implies that $X:=X_\Sigma$ is a smooth scheme and by Remark \ref{rem:quasi-projective} it is quasi-projective so that we can use
the results from \cite{HanlonHicks}.

By Theorem \ref{prop:lastmile} it suffices to prove Conjecture  \ref{conj:idcomplete}  with $D_\alpha$ replaced by $D^{\tl}$.
Let $\Dscr$ be the full subcategory of
\[
\Tw(\Wscr\Fscr((\CC^*)^k,D^{\tl})_\CC)\cong \Tw(\Wscr\Fscr(T^\ast(S_{N^\ast}),\Lambda^\infty_\Sigma)_\CC)
\] (cfr \eqref{eq:pairs})
 generated by the linking
  disks. 
  An inductive application of the proof of \cite[Proposition
  6.8]{HanlonHicks} shows that $\HMS_2(\Dscr)\subset D^b(\coh(X))$ is generated by the set of line bundles on the irreducible 
divisors in the toric boundary of $X$. The latter generate the category $D^b(\coh_{\partial X}(X))$ so that we obtain an equivalence 
$
\Dscr\cong D^b(\coh\nolimits_{\partial X}(X))
$ which yields that $\Dscr$ is 
  idempotent complete.

The Liouville manifold $D^{\tl}$ is Weinstein by Theorem \ref{th:weinstein} below. Moreover, the cocores are properly embedded by Lemma \ref{lem:properly_embedded} below. 
Hence, by (F4)(F5) the objects of the image of $\Two(\Wscr\Fscr(D^{\tl})_\CC)$ by the functor $\cup$ generate  $\Dscr$.

Let $\Dscr'$ be the subcategory of $\Tw(\Wscr\Fscr((\CC^*)^k,D^{\tl})_\CC)$ generated by the image of $\Dscr(\Wscr\Fscr(D^{\tl})_\CC)$ under $\cup$.
Since $\Two(\Wscr\Fscr((\CC^*)^k,D^{\tl})_\CC)\subset \Tw(\Wscr\Fscr((\CC^*)^k,D^{\tl})_\CC)$ we obtain from the previous paragraph that $\Dscr\subset \Dscr'$.

On the other hand $\Tw(\Wscr\Fscr((\CC^*)^k,D^{\tl})_\CC)$ is also the idempotent completion of $\Two(\Wscr\Fscr((\CC^*)^k,D^{\tl})_\CC)$ so  $\Dscr'$ is contained in the idempotent
completion of $\Dscr$. But since $\Dscr$ is idempotent complete we have in fact $\Dscr'\subset \Dscr$. Thus we conclude $\Dscr=\Dscr'$.
We obtain a commutative diagram
\[
\begin{tikzcd}
\Two(\Wscr\Fscr(D^{\tl})_\CC)\arrow[rd]\arrow[d,hook]&\\
\Tw(\Wscr\Fscr(D^{\tl})_\CC)\arrow[d,"\HMS","\cong"']\arrow[r,"\cup"]& \Dscr'\ar[d,"{\HMS_1}","\cong"']\\
D^b(\coh(\partial X))\arrow[r]&D^b(\coh_{\partial X}(X)).
\end{tikzcd}
\]
where the lower square is obtained from \eqref{eq:commutative_diagram}, together with the fact that the image of $D^b(\coh(\partial X))$ in $D^b(\coh(X))$ generates $D^b(\coh_{\partial X}(X))$.

Applying $K_0(-)$ to this diagram, noting $K_0(\partial X)\cong K_0(\coh_{\partial X}(X))$, we get
\[
\begin{tikzcd}
K_0(\Two(\Wscr\Fscr(D^{\tl}))_\CC)\arrow[rd,twoheadrightarrow]\arrow[d,hook]&\\
K_0(\Tw(\Wscr\Fscr(D^{\tl}))_\CC)\arrow[d,"\cong"]\arrow[r]& K_0(\Dscr')\arrow[d,"\cong"]\\
K_0(\partial X)\arrow[r,"\cong"]&K_0(\coh_{\partial X}(X)).
\end{tikzcd}
\]
We conclude that all arrows in the latter diagram are isomorphisms. In particular we get
\[
K_0(\Two(\Wscr\Fscr(D^{\tl}))_\CC)\cong K_0(\Tw(\Wscr\Fscr(D^{\tl}))_\CC)
\]
which implies Conjecture \ref{conj:idcomplete} by \cite[Theorem 2.1]{Thomason}.
\end{proof}
\subsection{More on the Lazarev map}
\label{sec:lazarev2}
Note that on the nose we have a surjective map
\[
K_0(\Two(\Wscr\Fscr(M))) \r  K_0(\Two(\Wscr\Fscr(M))_\CC)
\]
and we denote the composition with $\operatorname{L}$ (see \S\ref{sec:lazarev}) by $\LC$. This map is still surjective.
We state the following conjecture.
\begin{conjecture} \label{conj:lazarev} Claim (1,2) in the statement of Theorem \ref{th:lazarev} hold for $D_\alpha$.  Moreover the resulting map $\operatorname{L}_\CC$ (and hence
$\operatorname{L}$) is an isomorphism.
\end{conjecture}
\begin{proposition}\label{prop:LC}
Assume that we are in the setting of \S\ref{sec:setting} and that
 $A$ consists of all lattice points of $P=\conv(A)$. 
Then Conjecture \ref{conj:lazarev} is true.
\end{proposition}
\begin{proof} By Theorem \ref{prop:lastmile}  $D_\alpha\cong D^{\tl}$. Moreover the Liouville manifold $D^{\tl}$ is Weinstein by Theorem \ref{th:weinstein} below. Hence 
(1,2) hold, i.e. $\LC$ is a well-defined surjective map. It remains to show that $\LC$ is an isomorphism for $M=D^{\tl}$. Note that as in the proof of Proposition \ref{prop:idcomplete}, $\Sigma$ is smooth.

To prove that $\LC$ is an isomorphism it suffices to prove that source and target are free and have the same rank.
\begin{itemize}
\item $H^{k-1}(D^{\tl},\ZZ)\cong H^{k-1}(D_\alpha,\ZZ)\cong \ZZ^{k!\operatorname{vol}(P)+k-1}$ by Proposition \ref{prop:GKZ1_}(\ref{it:D1},\ref{it:D6}).
\item From Proposition \ref{prop:idcomplete} we get $K_0(\Two(\Wscr\Fscr(D^{\tl}))_\CC)\cong 
K_0(\Dscr(\Wscr\Fscr(D^{\tl}))_\CC)$.
From Theorem \ref{th:mirror_} we obtain a corresponding isomorphism of Grothendieck groups 
$K_0(\Tw(\Wscr\Fscr(D^{\tl})_\CC)) \cong K_0(\coh(\partial X))$.
Since the cones in $\Sigma$ are smooth the simplices spanned by the generators of the one dimensional cones have volume $1/k!$.
By Theorem \ref{thm:BHH} we obtain that $\rk K_0(X)=k! \vol(P)$ (since $c(\sigma)=1$). 
It follows from \eqref{eq:K0ses}  that  $ K_0(\coh(\partial X))$ is free of rank $k!{\vol}(P)+k-1$. This finishes the proof. \qedhere
\end{itemize}
\end{proof}
\subsection{Summary}
Combining everything, we have proved the following result.
\begin{proposition}\label{prop:Lazarev}
Assume that we are in the setting of \S\ref{sec:setting} and that
 $A$ consists of all lattice points of $P=\conv(A)$. 
 There is an isomorphism of abelian groups $\LC:H^{k-1}(D_\alpha,\ZZ) \allowbreak\to K_0(\Tw(\Wscr\Fscr(D_\alpha)_\CC))$
which is compatible with the action of  $\pi_1(\Kscr_A,\alpha)$ on $H^{k-1}(D_\alpha,\CC)$ described in Proposition \ref{prop:GKZ1_}\eqref{it:D7}
and the action of  $\pi_1(\Kscr_A,\alpha)$ on  $K_0(\Tw(\Wscr\Fscr(D_\alpha)_\CC))$
obtained from Theorem \ref{th:mainth1}.
\end{proposition}
\begin{proof} The isomorphism follows from Proposition \ref{prop:LC}. That it is compatible with the action boils down to the fact that the family $(D_\alpha)_\alpha$ is locally trivial (by Theorem \ref{prop:Dlocally} and Theorem \ref{lem:moser2}), which is used in the construction of the action in  Proposition \ref{prop:GKZ1_}\eqref{it:D7} and Theorem \ref{th:mainth1}.
\end{proof}

\subsection{The Gauss-Manin connection and the GKZ system}
In this section we assume that we are in the setting of \S\ref{sec:setting} and that
 $A$ consists of all lattice points of $P=\conv(A)$. 

\begin{proposition}\label{prop:seqsiso}
There is an isomorphism of  exact sequences 
\begin{equation}
\label{eq:maindiagram}
\footnotesize{
\begin{tikzcd}
0\arrow[r]& H^{k-1}(\CC^{*k},\ZZ)\arrow[r]\arrow[d]&H^{k-1}(D_\alpha,\ZZ)\arrow[r]\arrow[d]& H^k(\CC^{*k},D_\alpha,\ZZ)\arrow[r]\arrow[d, dotted]&H^k(\CC^{*k},\ZZ)\arrow[d]\arrow[r]&0\\
0\arrow[r]&X(T)\arrow[r]&K_0(\coh(\partial X))\arrow[r]&K_0(X)\arrow[r]&\ZZ\arrow[r]&0.
\end{tikzcd}
}
\end{equation} 
The dotted arrow will be constructed in the proof.
The second vertical is the composition of the natural map from Proposition \ref{prop:Lazarev} with the HMS isomorphism, 
 the top row is a part of the l.e.s. associated to $D_\alpha\subset (\CC^*)^k$  and the bottom row  is \eqref{eq:K0ses}. 
\end{proposition}

\begin{proof}
We first construct an isotopy $D_\alpha\r D^{\tl}$ like in the proof of Theorem \ref{prop:lastmile}. This extends to an isotopy $((\CC^*)^k,D_\alpha)\r ((\CC^*)^k,D^{\tl})$ as in \cite[Proposition 4.9]{Abouzaid}, which yields a commuting square
\begin{equation}
\label{eq:square1}
\begin{tikzcd}
H^{k-1}(D_\alpha,\ZZ)\arrow[r]\arrow[d,"\cong"]& H^k((\CC^*)^k,D_\alpha,\ZZ)\arrow[d,"\cong"]\\
H^{k-1}(D^{\tl},\ZZ)\arrow[r]& H^k((\CC^*)^k,D^{\tl},\ZZ).
\end{tikzcd}
\end{equation}
There is a commutative diagram 
\begin{equation}
\label{eq:square2}
\begin{tikzcd}
H^{k-1}(D^{\tl},\ZZ)\arrow[r]\arrow[d,"\operatorname{L}"']& H^k((\CC^*)^k,D^{\tl},\ZZ)\arrow[d,"\operatorname{L}"]\\
K_0(\Tw(\Wscr\Fscr(D^{\tl})_\CC))\arrow[r]&K_0(\Tw(\Wscr\Fscr((\CC^*)^k,D^{\tl})_\CC)).
\end{tikzcd}
\end{equation}
This diagram is a translation of \eqref{eq:lazarev1} using $(M,F)$ as in \S\ref{sec:liouvilletoric} where we observe that
\[
H^\ast(\tilde{M},\hat{F},\ZZ)\cong H^\ast(\hat{M},\tilde{F},\ZZ)\cong H^\ast(\hat{M},F,\ZZ)\cong H^\ast((\CC^*)^k,D^{\tl},\ZZ)
\]
where all the isomorphisms are obtained by invoking suitable homotopy equivalences.

We claim that the vertical maps are isomorphisms. The left most  map is an isomorphism by Proposition
\ref{prop:Lazarev}.
We claim that the right most map is also a isomorphism.

By \cite[Corollaries D, 6.9]{HanlonHicks}
$\Two(\Wscr\Fscr((\CC^*)^k,D^{\tl})_\CC) \cong D^b(\coh(X))$. It follows that  $\Two(\Wscr\Fscr((\CC^*)^k,D^{\tl})_\CC)$ is idempotent complete
and hence $\Tw(\Wscr\Fscr((\CC^*)^k,D^{\tl})_\CC)\allowbreak  \cong \Two(\Wscr\Fscr((\CC^*)^k,D^{\tl})_\CC)\cong D^b(\coh(X))$.
The right most map in \eqref{eq:square2} is  surjective with $\Tw$ replaced by $\Two$ (see \S\ref{sec:lazarev}). So is sufficient to prove that source and target have the same rank and that the source is a free abelian group.
By  Theorem \ref{thm:BHH}, $\rk(K_0(X))=k!{{\rm vol}(P)}$.
 Moreover,
$H^k((\CC^*)^k,D^{\tl},\ZZ)\allowbreak \cong H^k((\CC^*)^k,D_\alpha,\ZZ)\cong
\ZZ^{k!{\rm vol}(P)}$ by Proposition \ref{prop:GKZ1_}(\ref{it:D1},\ref{it:D5}), 
which finishes the proof of the claim.

Finally we repeat the commutative square \eqref{eq:commutative_diagram} which was constructed in \cite[Figure 6]{GammageShende}
\begin{equation}
\label{eq:square3}
\begin{tikzcd}
\Tw(\Wscr\Fscr(D^{\tl})_\CC)\arrow[d,"{\HMS}","\cong"']\arrow[r,"\cup"]& \Tw(\Wscr\Fscr((\CC^*)^k,D^{\tl})_\CC)
\arrow[d,"{\HMS_1}","\cong"']\\
D^b(\coh(\partial X))\arrow[r]&D^b(\coh(X)).
\end{tikzcd}
\end{equation}
The construction of the dotted arrow in \eqref{eq:maindiagram} as well as the commutativity of the middle square follows by composing the diagrams \eqref{eq:square1}\eqref{eq:square2}\eqref{eq:square3}.
\end{proof}

It follows from the proposition that we may define an induced action of $\pi_1(\Kscr_X)$ on $K_0(X)$ since the boundary terms in the cohomology sequence are preserved by the action (the action is actually trivial on those two terms).  In fact, this sequence of $\pi_1(\Kscr_X)$-representations is well understood by the work of Reichelt, Stienstra as mentioned in the introduction, cf. Proposition \ref{prop:GKZ1}.  

\begin{corollary}\label{cor:GKZ}
The induced action on $K_0(X)_\CC$ is obtained from the GKZ-system corresponding to the matrix $A$ with parameters $\beta=0$. 
\end{corollary}

\section{Another induced action  on the Grothendieck group of $X$}\label{sec:another_action} 
{\it We assume we are in the setting of \S\ref{sec:setting}. Moreover, we assume that $X_\Sigma$ is a scheme and ${\rm conv}(A)$ is a strongly convex cone.
}
\begin{corollary}\label{cor:cor} Assume that $X_\Sigma$ is a scheme. 
There is a natural action of $\pi_1(\Kscr_{X_\Sigma})$ on $K_0(X_\Sigma)$ obtained from the action of $\pi_1(\Kscr_{X_\Sigma})$ on $\Perf(\partial X)$.
\end{corollary}

\begin{remark} We expect the hypothesis that $X$ is a scheme in Corollary \ref{cor:cor} to be superfluous.
\end{remark}
\begin{proof}[Proof of Corollary \ref{cor:cor}]
By Corollary \ref{cor:maincor}  we have a natural action of $\pi_1(\Kscr_{X_\Sigma})$ on  $D^b(\coh(\partial X_\Sigma))$ as an object in $\Ho(A_\infty)$. By the standard Lemma \ref{lem:perf} below this implies that
$\pi_1(\Kscr_{X_\Sigma})$ acts on  $\Perf(\partial X_\Sigma)$ as an object in $\Ho_\CC(A_\infty)$. Hence $\pi_1(\Kscr_{X_\Sigma})$ acts on the topological K-theory of $\partial X_\Sigma$  by
\cite[Theorem 1.1.b]{Blanc}.

By Lemma \ref{lem:coh} below it follows that $K^\ast_{\operatorname{top}}(\partial X_\Sigma)\cong K^\ast_{\operatorname{top}}(X_\Sigma)$.  Both $K_0$ and $K_{\operatorname{top}}^0$ are additive functors (for
the latter see \cite[Theorem 1.1.c]{Blanc}) and they agree on affine spaces. Moreover there exists a natural transformation $K_0\r K^0_{\operatorname{top}}$ by \cite[Theorem 1.1.d]{Blanc}. It follows from Lemma \ref{lem:chern} below together with \cite[Theorem 2.7]{MR3717985} that
the natural morphism $K_0(X_\Sigma)\r K^0_{\operatorname{top}}(X_\Sigma)$ is in fact an isomorphism. Hence via the combined isomorphism $K_0(X_\Sigma)\cong K^0_{\operatorname{top}}(\partial X_\Sigma)$ we obtain
the asserted action.
\end{proof}
\begin{lemma}
\label{lem:perf}
Let $M$ be a noetherian scheme. Then $F \in D^b(\coh(M))$ is perfect if and only if for all $G\in D^b(\coh(M))$, $\RHom_M(F,G)$ has bounded cohomology.
\end{lemma}
\begin{proof} We prove the nonobvious direction. By letting $G$ be equal to the residue field of $\Oscr_{M,m}$ for all $m$ in $M$ we reduce to the following classical statement: a 
bounded complex $F$ of finitely generated $R$-modules over
a noetherian local ring $R$ with residue field $K$ is perfect if and only if $\RHom_R(F,K)$ has bounded cohomology.
\end{proof}
\begin{lemma} \label{lem:chern} Assume that $X_\Sigma$ is a scheme. Then $X_\Sigma$ has a stratification by affine spaces.
\end{lemma}
\begin{proof} 
Let $\sigma$ be the cone spanned by $P$. Let $\lambda$ be a sufficiently
general element of $\ZZ^k\cap \sigma$. Then $\lambda$ defines a $G_m$-action on $X_\Sigma$. By looking at the open subsets $\Spec \ZZ[\tau^\vee]$ of $X_\Sigma$ for $\tau\in \Sigma$ one
sees that the fix points for the $\lambda$-action on $X_\Sigma$ coincide with the torus fixed points and hence they are finite in number. Moreover if $x\in X_\Sigma$ then $\lim_{t\r 0} \lambda(t) \pi(x)$ for $\pi:X_\Sigma\r X_\sigma$ exists as it is the unique fixed point in $X_\sigma$. Since $\pi$
 is proper it follows that $\lim_{t\r 0} \lambda(t) x$ also exists. Now use the Bia\l ynicki-Birula decomposition of $X_\Sigma$ associated to the $G_m$-action given by $\lambda$.
\end{proof}
\begin{lemma}\label{lem:coh}
Assume that $X_\Sigma$ is a scheme. Then $\partial X_\Sigma$ is a strong deformation retract of $X_\Sigma$.
\end{lemma}

\begin{proof}
As above let $\sigma$ be the cone spanned by $P$.
  Let $X_{\bar{\Sigma}}$ be a compactification of $X_\Sigma$  
determined by the complete fan $\bar{\Sigma}$ obtained from $\Sigma$ by adjoining
  a $1$-dimensional cone $\RR_+u$ such that $-u$ belongs to the
  interior of the cone $\sigma$ and the corresponding higher
  dimensional cones of the form $\operatorname{Span}(u,\tau)$ for
  all $\tau\in \partial \sigma$. The strictly convex function $\nu$ (see \S\ref{sec:polyhedral},\ref{sec:stacky}) on $\Sigma$ may be 
extended to a strictly convex piecewise linear function $\bar{\nu}$ on $\bar{\Sigma}$ by giving 
$\bar{\nu}(u)$ a large positive value.
Hence $X_{\bar{\Sigma}}$ is projective.
 Let $Q$ be a sufficiently large lattice normal polytope
  of  $\bar{\Sigma}$.

  Let $S_k$ be the $k$-dimensional compact real torus in $(\CC^*)^k$.  By
  \cite[(12.2.7)]{CoxLittleSchenck}, there is an $S_k$-equivariant
  homeomorphism
  $\phi:X_{\bar{\Sigma}}\xrightarrow{\cong} (S_k\times Q)/_{\sim}$ where $\sim$
  is a suitable equivalence relation on $S_k\times \partial Q$.
Moreover, by the proof
  of \cite[Theorem 12.2.5]{CoxLittleSchenck}
  $\phi(X_\Sigma)=(S_k\times Q^o)/_{\sim}$ where $Q^\circ=Q-F$ with
  $F$ being the facet perpendicular to $u$ and
  $\phi(\partial X_\Sigma)=(S_k\times \partial Q^o)/_{\sim}$ where
  $\partial Q^o=\partial Q-F$. Note that
  $\partial Q^o\to Q^o$ is a strong deformation retract, and thus so is
  $S_k\times \partial Q^o\to S_k\times Q^o$.
As it is compatible with
  $\sim$, we obtain a strong deformation retraction $\partial X_\Sigma\to X_\Sigma$. 
\end{proof}
\begin{remark} Note that it is crucial for this argument that $X_\Sigma$ is not complete. Indeed the result is trivially false in the complete case.
For example $\partial \PP^1$ consists of two points which is not a strong deformation retract of $\PP^1$.
\end{remark}

\section{Remarks on lifting the action to $D^b(X)$}
 \label{sec:SKMS} 
As before let $X=X_\Sigma$ for some toric fan $\Sigma$.
While it follows from \eqref{eq:maindiagram} that the action of $\pi_1(\Kscr_A)$ on $K_0(\coh(\partial X))$ can be lifted to an action on $K_0(\coh(X))$, it is not true in general
that the action of $\pi_1(\Kscr_A)$ on $D^b(\coh(\partial X))$ can be lifted to an action on $D^b(\coh(X))$. This is  not very surprising as in general $\partial X$ will have many more derived autoequivalences than~$X$. Nonetheless to make an actual
counterexample one needs to find an element of $\pi_1(\Kscr_A)$ such that the corresponding derived autoequivalence does not extend. This will be done in the next section.
The message is that $\Kscr_A$ can in general not be considered as the SKMS of $X$.

However, for crepant resolutions of Gorenstein affine toric varieties, we show that~$\Kscr_A$ coincides (modulo a copy of $\CC^*$) with the standard definition \cite{SegalDonovan,Kitethesis}. In this case a number of examples are known where $\pi_1(\Kscr_A)$ acts on $D^b(\coh(X))$ (see \cite{HLSam, Kitethesis}).

\begin{remark}
The difference between the general case and the Gorenstein case can be understood if one realizes that the definition of $\Kscr_A$ involves the discriminant locus of the Laurent
polynomial $f$ in the  Hori-Vafa LG-model $((\CC^\ast)^d,f)$ which is mirror to $X$. This discriminant locus refers to the behaviour of the zero-fiber $f^{-1}(0)$ of $f$. 
 In the Gorenstein case $f$ is homogeneous and hence the zero fiber is the worst possible fiber. However in the general
case, the zero fiber is not distinguished.
\end{remark}

\subsection{Example of non-lifting of the action of $\pi_1(\Kscr_A)$ on $D^b(\coh(X))$}
Here we give an example of $X$ where the action of $\pi_1(\Kscr_A)$ on $D^b(\coh(\partial X))$ does not lift to $D^b(\coh(X))$. We describe this somewhat heuristically, ignoring the differences of various incarnations of Fukaya categories. 

We take $X=\PP^2$. Then $A=\{(0,0),(1,0),(0,1),(-1,-1)\}$. By rescaling we may assume that $f_\alpha=\alpha_0+x+y+\alpha_1 x^{-1}y^{-1}$. Let $\Kscr'_A$ be the subspace of $\Kscr_A$ obtained by setting $\alpha_1=1$. Then $\Kscr'_A=\CC^*\setminus\{1,-3\}$. Let $f_a=a+x+y+x^{-1}y^{-1}$. The family $(f_a^{-1}(0))_{a\in \Kscr'_A}$ is a family of elliptic curves without three points. Varying $a$ corresponds to varying fibers of the potential $f=x+y+x^{-1}y^{-1}$, and $f_a^{-1}(0)$ for $a$ in the complement of  $\Kscr'_A$ in $\CC$ corresponds to singular fibers. 

The monodromy for loops based at $a\in \Kscr'_A$ around the singular fibers 
is given by an action on   
$\WF(f_a^{-1}(0))$\footnote{We are being a bit sloppy here. Seidel is using a version of Fukaya category in a slightly different setting.}  by spherical twists corresponding to vanishing cycles \cite{Seidel01,SeidelBook} (see \cite[Theorems 2.3.1,2.3.4]{ZihongChen}). These vanishing cycles are restrictions of Lefshetz thimbles that correspond by mirror symmetry for the LG model $((\CC^*)^2,f)$ \cite[\S3.2,p.13]{ZihongChen}  to 
 the $\Oscr_{\PP^2}(-1),\Omega_{\PP^2}^1(1),\Oscr_{\PP^2}$ on $\PP^2$. 
 As the restriction functor $\WF((\CC^*)^2,f)\to \WF(f_a^{-1}(0))$  corresponds to the restriction $D^b(\coh(\PP^2))\to D^b(\coh(\partial \PP^2))$ \cite[Proposition 3.2.1]{GammageJeffs}\footnote{Again, the Fukaya category in loc.cit. is yet another incarnation of the Fukaya category associated to the LG model.}, the action of $\pi_1(\Kscr'_A)$ on $D^b(\coh(\partial X))$ is given by spherical twists corresponding to the restrictions to $\partial \PP^2$ of the above sheaves on $\PP^2$.  

We claim that the spherical twist with $\Oscr_{\partial \PP^2}$ on $D^b(\coh(\partial \PP^2))$ does not arise from an autoequivalence of $D^b(\coh(\PP^2))$. First recall that the latter are generated by tensoring with line bundles, automorphisms of $\PP^2$ and the shifts \cite[Theorem 3.1]{BondalOrlov}. 

The spherical twist corresponding to $\Oscr_{\partial \PP^2}$ is given by
\[
\Phi_{\Oscr_{\partial \PP^2}}(\Fscr)={\rm cone}(\RHom(\Oscr_{\partial \PP^2},\Fscr)\otimes \Oscr_{\partial \PP^2}\to \Fscr)={\rm cone}(R\Gamma(\Fscr)\otimes \Oscr_{\partial \PP^2}\to \Fscr).
\]
We apply it with, for example,  $\Fscr=\Oscr_{\PP^1}(1)\oplus\Oscr_{\PP^1}(-2)$. Then $\Phi_{\Oscr_{\partial \PP^2}}(\Fscr)$ has nonzero cohomology in degrees $-1,0$. However, none of the derived autoequivalences of $\PP^2$ can change the number of the non-vanishing cohomology degrees.

\begin{remark}
Note that this is not a contradiction with the (decategorified) action of $\pi_1(\Kscr_A)$ on $K_0(\coh(\partial X))$ inducing an action on $K_0(X)$ as one may check that the spherical twists preserve the kernel of $K_0(\coh(\partial X))\to K_0(X)$ which is generated by the differences of structure sheaves of the irreducible boundary divisors. 
\end{remark}

\begin{remark}
In \cite{Hanlon} Hanlon constructs an action on the (monomially admissible) Fukaya category of the LG model corresponding to varying $\alpha_1$, 
by taking a loop around $0$. This is not in contradiction with the above, as sending $\alpha_1$ to $0$ would correspond to sending $\alpha_0$  to $\infty$. 
\end{remark}

\subsection{SKMS and crepant resolutions}
In this section we connect the setting $A\subset \{1\} \times \ZZ^{k-1}$ of Introduction to the setting of other sections.  The relevant notation for this section is introduced in \S\ref{sec:prindet}.

As above $P=\conv(A)$. Let $\hat{A}=A\cup \{0\}$ and  $\hat{P}=\conv(\hat{A})$
as in \S\ref{sec:adjoin} below. 
A triangulation of $P$ defines a star-shaped triangulation $\Tscr$ (see \S\ref{sec:stacky}) of $\hat{P}$ in the obvious way and we let $\Sigma$ be the fan spanned by $\Tscr$. The following lemma shows that working with $A$ is equivalent to working with $\hat{A}$.
\begin{lemma} The projection $\CC^{\hat{A}}\r \CC^A$ defines an isomorphism between $\Kscr_{\hat{A}}$ and $\CC^*\times\Kscr_A$.
\end{lemma}
\begin{proof} This follows from Lemma \ref{lem:EtildeA} below.
\end{proof}
\subsubsection{Adjoining zero}
\label{sec:adjoin}
We discuss a  technical result which we have used. 
For $A\subset \ZZ^{k-1}\times \{1\}$ we put $a_0=0^k$ and $\hat{A}=\{a_0\}\cup A$. We write the elements $\hat{\alpha}\in \CC^{\hat A}$ as $(\alpha_0,\alpha)$ with $\alpha\in\CC^{A}$.
We put $\hat f_{\hat \alpha}=\alpha_0+f_\alpha$. 
Let $p:\CC^{\hat A}\r \CC^A$ be the projection map which forgets
the first coordinate. 

\begin{lemma}\label{lem:EtildeA}
We have $V(\hat{A})=V(\alpha_0)\cup p^{-1}(V(A))$. 
\end{lemma}

\begin{proof}
Let $P$, $\hat{P}$ be the convex hulls of $(a_i)_{a_i\in A}$, $(a_i)_{a_i\in \hat{A}}$  and let $\hat{F}$ be a face of $\hat{P}$. There are three possibilities for $\hat{F}$.
\begin{enumerate}
\item  $\hat{F}=\{a_0\}$. In this case we have
\[
\nabla_{\hat{F}}=V(\alpha_0).
\]
\item $\hat{F}$ is a face $F$ of $P$. In this case
\[
\nabla_{\hat{F}}=p^{-1}(\nabla_F).
\]
\item $\hat{F}$ the convex hull of $\{a_0\}$ and a face $F$ of $P$. In this case we claim
\[
\nabla_{\hat{F}}=V(\alpha_0)\cap p^{-1}(\nabla_F).
\]
Indeed to calculate $\nabla_{\hat{F}}$ we have to verify when the singular locus
of $V(\alpha_0+f^F_\alpha(x))$ intersects $(\CC^\ast)^k$. This happens when $\alpha_0+f^F_\alpha(x)$ 
and $\partial_i f^F_\alpha(x)$, $1\leq i\leq k$, have a common zero in $(\CC^{\ast})^{ k}$. Now $\partial_k f^F_\alpha(x)=0$
for $x\in (\CC^\ast)^k$ is equivalent to $f^F_\alpha(x)=0$ by the fact that $A\subset \ZZ^{k-1}\times \{1\}$. Thus $\alpha_0+f^F_\alpha(x)$ 
and $\partial_i f^F_\alpha(x)$, $1\leq i\leq k$, have a common zero in $(\CC^{\ast})^{ k}$ if and only if $\alpha_0=0$ and $f^F_\alpha(x)$ and $\partial_i f^F_\alpha(x)$, $1\leq i\leq k$, have a common zero in $(\CC^{\ast})^{ k}$. 
This proves the claim.
\end{enumerate}
The lemma follows by combining (1)(2)(3).
\end{proof}

\subsection{Dependence of the fundamental group of $\Kscr_A$ on $A$}
If $\sigma$ is the cone spanned by $P$ then $X_\sigma$ is a  Gorenstein affine toric variety (see e.g. [CLS11,
Proposition 11.4.12]). The $X_\Sigma$ obtained from triangulations of $P$ are the crepant resolutions of $X_\sigma$.\footnote{$X_\Sigma$ is crepant \cite[Proposition 8.2.7]{CoxLittleSchenck} (see \cite[Example 11.2.6]{CoxLittleSchenck}, \cite[Lemma A.2]{SVdB5}). Every crepant resolution appears in this way as any crepant resolution of a toric variety is toric and \cite[Lemma 11.4.10]{CoxLittleSchenck}. This is a consequence of a minimal model problem fact that the number of crepant resolutions of an algebraic variety is finite \cite{BCHM}.} 
  By \cite{Kaw} all crepant resolutions of $X_\sigma$ are derived
  equivalent. In keeping with the general philosophy of homological mirror symmetry one might
  therefore hope $\Kscr_A$ to be an invariant of $X_\sigma$. 

  This is true if we only consider crepant resolutions that are schemes
  since it that case $A=P\cap \ZZ^k$ which does not depend on $X_\Sigma$. However it is false, for obvious dimension reasons, if we also allow crepant resolutions by DM-stacks since these correspond to smaller $A$.
In general there does not seem to be an obvious relation between the different $\Kscr_A$.

If $X_\sigma$ is the GIT quotient for a quasi-symmetric representation of a torus then, if we take for $A$ the vertices of $P$, $\Kscr_A$ is the complement of a toric hyperplane
arrangement \cite{Kite}. One might hope that this would remain true for the other $\Kscr_A$ but in Example \ref{ex:dependsonA} below we show that this is false. Furthermore, even the fundamental groups are not isomorphic.

\begin{example}
 \label{ex:dependsonA}
We consider the case where $X_\sigma$ is the GIT quotient of a $1$-di\-mensional torus acting with weights $-2,-1,1,2$ on a 4-dimensional representation. Using Gale duality \cite[\S 14.2]{CoxLittleSchenck} one sees that $P$ is the convex hull of
\[
  (0,0,1),(1,0,1),(1,1,1),(0,2,1).
\]
A picture shows that there are only two possible choices for $A$, one given by the vertices of $P$, and one given by $P\cap \ZZ^2$.
If we choose
\[A_1=\{(0,0,1),(1,0,1),(1,1,1),(0,2,1)\},\] 
(i.e.\ the set of vertices of $P$) 
with corresponding Laurent polynomial
\[
f=\alpha_1z+\alpha_2xz+\alpha_3xyz+\alpha_4y^2z
\]
then one may check that $V(A_1)=V(\alpha_1\alpha_2\alpha_3\alpha_4(\alpha_1\alpha_3^2+\alpha_2^2\alpha_4))$. By replacing $x$, $y$, $z$ by scalar multiples we may  make the non-zero $\alpha_1$, $\alpha_2$, $\alpha_3$ equal to 1 (in a unique way), which
 yields $\Kscr_{A_1}=\CC\setminus (V(\alpha_4)\cup V(1+\alpha_4))=\CC\setminus \{0,-1\}$. So $\Kscr_{A_1}$ is the complement of three distinct points in $\PP_\CC^1$ and $\pi_1(\Kscr_A)$ is a free
group on two generators. On the other hand for 
 \[A_2=\{(0,0,1),(1,0,1),(1,1,1),(0,2,1),(0,1,1)\}\] 
(all lattice points in $P$) with corresponding Laurent polynomial
\[
f=\alpha_1z+\alpha_2xz+\alpha_3xyz+\alpha_4y^2z+\alpha_5 yz
\]
we obtain 
 $V(A_2)=V(\alpha_1\alpha_2\alpha_3\alpha_4(\alpha_1\alpha_3^2+\alpha_2^2\alpha_4-\alpha_2\alpha_3\alpha_5)(\alpha_5^2-4\alpha_1\alpha_4))$. Again scaling away $\alpha_1$, $\alpha_2$, $\alpha_3$ we obtain
 $\Kscr_{A_2}=(\CC^*\times \CC)\setminus (V(1+\alpha_4-\alpha_5)\cup V(\alpha_5^2-4\alpha_4))$. If we put $u=\alpha_4$, $v=1+\alpha_4-\alpha_5$ then we obtain 
\begin{equation}
\label{eq:octahedron_vgit}
\begin{aligned}
\Kscr_{A_2}&=\{(u,v)\in (\CC^\ast)^2\mid u^2+v^2+1-2u-2v-2uv\neq 0\}\\
&=\{(U:V:W)\in \PP^2_\CC\mid U^2+V^2+W^2-2UV-2UW-2VW\neq 0, \\
&\qquad\qquad\qquad \qquad\qquad \qquad \hspace*{3cm} U\neq 0, V\neq 0, W\neq 0\}.
\end{aligned}
\end{equation}
So $\Kscr_{A_2}=\PP^2_\CC-(Q\cup L_1\cup L_2\cup L_3)$ where $Q$ is a smooth conic and $(L_i)_i$ are three distinct tangent lines of $Q$. Such a configuration is unique up to isomorphism.

Remarkably the variety \eqref{eq:octahedron_vgit} also appears in
\cite[\S7.2]{Kitethesis} in a very different toric setting. As far as
we can tell this is a coincidence. It means however that we can use the
results in loc.\ cit.\ to compute $\pi_1(\Kscr_{A_2})$ and see that it
is  different from $\pi_1(\Kscr_{A_1})$. To show this directly
 let us quickly verify that $\pi_1(\Kscr_{A_2})$ is not free. Let $\{p_{12}\}=L_1\cap L_2$. Then $p_{12}$ has a neighbourhood $B$ in $\Kscr_{A_2}$ which is the product of two punctured disks.
 Hence if we take a base point $b\in B$ then the loops around $L_1$, $L_2$ in $B$ commute in $\pi_1(B)$, and hence in $\pi_1(\Kscr_{A_2})$ and they do
not generate a cyclic subgroup, which one can check by looking at their image in $\pi_1(\PP^2_\CC-(L_1\cup L_2\cup L_3),b)\cong \pi_1(\CC^{\ast 2})\cong\ZZ^2$. Hence  $\pi_1(\Kscr_{A_2})$ is not free.
\end{example}

\appendix
\section{Linear algebra exercises}
\def\|{|}
\def\ddd{||}
\subsection{Reminder on Hermitian vector spaces}
\label{sec:hermitian}
Below $V$ is a finite dimensional complex vector space and $h:V\times V\r \CC$ is a non-degenerate Hermitian form.
We write $g:=\Re h$, resp.\ $\omega:=-\Im h$, for the corresponding Riemannian metric and  symplectic form. From
$h(X,iY)=-ih(X,Y)$ we obtain $g(X,iY)=-\omega(X,Y)$. In other words $h$ is uniquely determined by  $g$ and by $\omega$.

We write $V^\ast$ for $\Hom_\CC(V,\CC)$.
The hermitian metric yields a $\CC$-anti-linear identification
\begin{equation}
\label{eq:identification}
V\r V^\ast:X\mapsto h(-,X).
\end{equation}
Below we write $\|X\|=\sqrt{g(X,X)}$. If $a:V\r \CC$ is a $\CC$-linear map then we put $\|a\|:=\|X\|$ where $X$ is such that
$a=h(-,X)$. If $a:V\r \CC$ is $\CC$-anti-linear then we put $\|a\|:=\|\bar{a}\|$. If $a:V\r \CC$ is an $\RR$-linear map then
we write $a=a'+a''$ for the unique decomposition of $a$ as a sum of a $\CC$-linear and a $\CC$-anti-linear map. 

If $X,Y\in V$ then the \emph{Hermitian angle} $\tau\in [0,\pi/2]$ between $X$ and $Y$ is defined by $\cos\tau=|h(X,Y)|/(\|X\|\|Y\|)$. Hermitian angles between vectors in $V$ and $V^\ast$ or between
vectors in $V^\ast$ are
computed using the identification \eqref{eq:identification}.
\begin{remark} The Hermitian angle  between $X$ and $Y$ is the angle between the planes $\CC X$ and $\CC Y$ for the Riemannian metric $g$.
\end{remark}
\begin{remark} \label{rem:normproportionality} If we consider $V$ as a real vector space with a
  Riemannian metric then we may still put a norm on
  $\Hom_{\RR}(V,\RR)$ as follows: if $a\in \Hom_{\RR}(V,\RR)$ then
  $a=g(X,-)$ for $X\in V$ and we put $\ddd a\ddd =|X|$. We may extend
  this norm to $\Hom_{\RR}(V,\CC)$ by putting
  $\ddd a_1+ia_2\ddd _g^2=\ddd a_1\ddd _g^2+\ddd a_2\ddd _g^2$ for
  $a_1,a_2\in \Hom_\RR(V,\RR)$. However one needs to be careful. If
  $a:V\r \CC$ is $\CC$-linear or $\CC$-anti-linear so that $|a|$ is defined via $h$, as above then 
one computes
\begin{equation}
\label{eq:normproportionality}
\ddd a\ddd =\sqrt{2} |a|.
\end{equation}
\end{remark}

\subsection{Some lemmas by Donaldson}
We recall some lemmas which were stated  without proof in \cite{Donaldson}.
\begin{lemma}[{\cite[p669]{Donaldson}}]\label{lem:don1} An $\RR$-linear map $a:V\r \CC$ is surjective unless there exists $\alpha\in \RR$ such that $\overline{a''}=e^{i\alpha} a'$.
\end{lemma}
\begin{proof} Identifying $V\cong \CC^n$ we may assume that $a'=\sum_j a_j z_j$, $a''=\sum_j b_j \bar{z}_j$
for $(a_j)_j\in \CC$, $(b_j)_j\in \CC$. Assume $a$ is not surjective. Then there exists $u\in \CC$
such that $a_j z_j+b_j \bar{z_j}\in \RR u$ for all $j$ and all $z_j\in \CC$. Putting $z_j=1,i$ we get 
\begin{align*}
a_j+b_j\in \RR u\\
a_j-b_j\in i\RR u
\end{align*}
so that we get for suitable $\lambda,\mu \in \RR$:
\begin{align*}
a_j=(\lambda+ i\mu)u\\
b_j=(\lambda-i\mu)u.
\end{align*}
So we may take $e^{i\alpha}=\bar{u}/u$. The converse is a direct verification.
\end{proof}
Assume $a:V\r \CC$ is an $\RR$-linear map. Let $A',A''\in V$ be such that $a'=h(-,A')$,
$a''=\overline{h(-,A'')}$.
\begin{lemma} We have
\[
\ker a=\{X\in V\mid g(X,A'+A'')= g(X,i(A'-A''))=0\}.
\]
\end{lemma}
\begin{proof} Let $X\in V$. Then $a(X)=0$ if and only if $\Re a(X)=0$, $\Re i a(X)=0$.
We have 
\begin{align*}
\Re a(X)=0&\iff \Re h(X,A'+A'')=0\\
&\iff g(X,A'+A'')=0,
\end{align*}
\begin{align*}
\Re i a(X)=0&\iff \Re i h(X,A'-A'')=0\\
&\iff g(X,i(A'-A''))=0.\qedhere
\end{align*}
\end{proof}
\begin{corollary}\label{cor:perp}
We have
\begin{align*}
(\ker a)^{\perp_g}&=\RR(A'+A'') + i \RR(A'-A''),\\
(\ker a)^{\perp_\omega}&=\RR(A'-A'')+i\RR(A'+A'').
\end{align*}
\end{corollary}
\begin{proof} The second equality follows from $(\ker a)^{\perp_\omega}=i(\ker a)^{\perp_g}$.
\end{proof}
\begin{lemma}[{\cite[Proposition 3]{Donaldson}}] \label{lem:donaldson1} If $\|a'\|\neq \|a''\|$ then the restriction of $\omega$ to $\ker a$ is
a symplectic form.
\end{lemma}
\begin{proof} We need to prove 
$(\ker a)^{\perp \omega }\cap \ker a=0$. In other words by Corollary \ref{cor:perp} we should have $a(i(A'+A''))\neq 0$ or $a(A'-A'')\neq 0$. We have
\begin{align*}
a(A'-A'')&=h(A'-A'',A')+\overline{h(A'-A'',A'')}\\
&=h(A',A')-h(A'',A')-h(A'',A'')+\overline{h(A',A'')}\\
&=g(A',A')-g(A'',A'').
\end{align*}
Likewise
\begin{align*}
a(i(A'+A''))&=h(i(A'+A''),A')+\overline{h(i(A'+A''),A'')}\\
&=i h(A',A')+ih(A'',A')-ih(A'',A'')-i\overline{h(A',A'')}\\
&=i(g(A',A')-g(A'',A'')).\qedhere
\end{align*}
\end{proof}
\subsection{Formula for a symplectic projection}
\label{sec:simpproj}
Assume $\|a'\|\neq \|a''\|$. Let $Z\in V$ and let $Z_1$ be the symplectic projection of $Z$ on $\ker a$. So we have $Z=Z_1+Z_2$ where $Z_2\perp_\omega \ker a$.
Thus by Corollary \ref{cor:perp}
\[
Z_2=\alpha(A' - A'' ) + i\beta (A' + A'' )=\lambda A'-\bar{\lambda} A''
\]
for $\alpha,\beta\in \RR$, $\lambda=\alpha+i \beta$. To determine $\alpha,\beta$ we must use $a(Z-Z_2)=0$ or $a(Z_2)=a(Z)$.
\begin{align*}
a(Z_2)&=h(Z_2,A')+\overline{h(Z_2,A'')}\\
&=h(\lambda A'-\bar{\lambda} A'',A')+\overline{h(\lambda A'-\bar{\lambda} A'',A'')}\\
&=\lambda h(A',A')-\bar{\lambda} h(A'',A')+\bar{\lambda}\overline{ h(A',A'')}-\lambda \overline{h(A'',A'')}\\
&=\lambda(\|a'\|^2-\|a''\|^2).
\end{align*}
Thus we get the formulas
\[
\lambda=\frac{a(Z)}{\|a'\|^2-\|a''\|^2}
\]
\[
 Z_2=\frac{a(Z)A'-\overline{a(Z)}A''}{\|a'\|^2-\|a''\|^2}
\]
\begin{equation}
\label{eq:Z1}
Z_1=Z-\frac{a(Z)A'-\overline{a(Z)}A''}{\|a'\|^2-\|a''\|^2}.
\end{equation}
\subsection{Bounding the length of a symplectic projection}
\begin{proposition} \label{prop:projectionbound}
Let $a:V\r \CC$ be an $\RR$-linear map such that $\|a''\|<\|a'\|$. Let $Z\in V$ and let $Z_1$ be the symplectic projection of $Z$ on $\ker a$.
Put $r=\|a''\|/\|a'\|$ and $\gamma=2/(1-r)$. Then $\|Z_1\|\le \gamma \|Z\|$.
\end{proposition}
\begin{proof}
We compute
\begin{align*}
\|Z_1\|&=\left\|Z-\frac{a(Z)A'-\overline{a(Z)}A''}{\|a'\|^2-\|a''\|^2}\right\|\\
&\le \|Z\|+\frac{|a(Z)| \|a'\|+|a(Z)|\|a''\|}{\|a'\|^2-\|a''\|^2}\\
&\le\|Z\|+\frac{(\|a'\|+\|a''\|) \|Z\| \|a'\|+(|a'\|+\|a''\|) \|Z\|\|a''\|}{\|a'\|^2-\|a''\|^2}.
\end{align*}
So we get
\[
\frac{\|Z_1\|}{\|Z\|}\le 1+\frac{(1+r)^2}{1-r^2}=\frac{2}{1-r}.\qedhere
\]
\end{proof}
\begin{remark} With a more precise computation one may reduce $\gamma$ to
 $(1+r^2)/(1-r^2)$. However the actual bound is not important for us.
\end{remark}
\subsection{Another bound}
\label{sec:anotherbound}
\begin{proposition} \label{prop:anotherbound}
 Let $a:V\r \CC$ be an $\RR$-linear map such that $\|a''\|<\|a'\|$. Let $0\neq Z\in V$ and let $Z_1$ be the symplectic projection of $Z$ on $\ker a$.
Let $\tau'\in [0,\pi/2]$ be the \emph{Hermitian angle} between $a'$ and $Z$, defined via $\cos \tau'=|a'(Z)|/(\|a'\|\|Z\|)$. If
\[
\sigma:=\sin \tau'- \frac{\|a''\|}{\|a'\|} > 0
\]
then 
\[
g(Z,Z_1)\ge M|Z|^2
\]
where $M>0$ is a constant depending on $\sigma$ and $\tau'$. In particular $g(Z,Z_1)>0$.
\end{proposition}
\begin{proof}
We have $Z=Z_1+Z_2$ with  $Z_1\in \ker a$ and
$Z_2\in (\ker a)^{\perp_\omega}$. Hence
 $\omega(Z,Z_1)=\omega(Z_1+Z_2,Z_1)=\omega(Z_2,Z_1)=0$ as. So we get $g(Z,Z_1)=h(Z,Z_1)$. By \eqref{eq:Z1}:
\begin{equation}
\label{eq:calcul}
\begin{aligned}
g(Z,Z_1)&=\|Z\|^2-\frac{\Re(\overline{a(Z)}a'(Z)-a(Z)\overline{a''(Z)})}{\|a'\|^2-\|a''\|^2}\\
&=\|Z\|^2-\frac{|a'(Z)|^2-|a''(Z)|^2}{\|a'\|^2-\|a''\|^2}.
\end{aligned}
\end{equation}
Let $\tau''$ be the Hermitian angle between  $\overline{a''}$ and $Z$ (i.e. $\cos \tau''=|a''(Z)|/(\|a''\|\|Z\|)$, $\tau''\in [0,\pi/2]$). Then
from \eqref{eq:calcul} we get
\begin{align*}
g(Z,Z_1)&=\frac{\|Z\|^2}{\|a'\|^2-\|a''\|^2}( \|a'\|^2 \sin^2\tau'-\|a''\|^2 \sin^2\tau'')\\
&\ge\frac{\|Z\|^2}{\|a'\|^2}( \|a'\|^2 \sin^2\tau'-\|a''\|^2 )\\
&=\|Z\|^2(\sin^2\tau'-(\sin\tau'-\sigma)^2 )
\end{align*}
which proves what we want using the fact that $0<\sigma\le \sin\tau'$.
\end{proof}
\begin{corollary} \label{cor:gradient} Under the hypotheses of Proposition \ref{prop:anotherbound} we have
\[
g(Z,Z_1)\ge \delta(|Z|^2+|Z_1|^2)
\]
for suitable $\delta>0$ depending on $\sigma$ and $\tau'$.
\begin{proof} With the notation of Proposition \ref{prop:projectionbound} we have $r=\sin\tau'-\sigma$. By that proposition we have
$|Z_1|\le N|Z|$ for a suitable constant $N$ depending on $r$ and hence on $\tau'$ and $\sigma$. Hence by Proposition \ref{prop:anotherbound}
\begin{align*}
|Z|^2+|Z_1|^2&\le (1+N^2)|Z|^2\\
&=\frac{1+N^2}{M} g(Z,Z_1).
\end{align*}
Hence we can take $\delta=M/(1+N^2)$.
\end{proof}
\end{corollary}

\section{Reminder on the cohomology of hypersurfaces in tori}\label{sec:remH}
We assume that we are in the setting of \S\ref{sec:prindet}. For $\alpha\not\in V(A)$ we put $D_\alpha=f^{-1}_\alpha(0)$. 
 To avoid trivialities we assume $|A|>1$ so that $f^{-1}_\alpha(0)$ is non-empty by Lemma \ref{lem:notzero}. 
The following proposition summarizes what we know about the various cohomology groups attached to $D_\alpha$.
\begin{proposition}[{\cite{Oka}, \cite{Khovansky}, \cite[Remark 2.14]{Reichelt}, \cite{Stienstra}}] \label{prop:GKZ1_}

\rule{0pt}{0pt}

\begin{enumerate}
\item \label{it:D1} The cohomology groups $H^i(D_\alpha,\ZZ)$, $H^i(\CC^{\ast k},\ZZ)$ and 
$H^i(\CC^{\ast k}, D_\alpha,\ZZ)$ are free for all~$i$.
\item \label{it:D2} $H^i(D_\alpha,\ZZ)=0$ for $i>k-1$.
\item \label{it:D3} The restriction map $H^{i}(\CC^{\ast k},\ZZ)\r H^{i}(D_\alpha,\ZZ)$ is an isomorphism for $i<k-1$ and a split monomorphism for $i=k-1$.
\item  \label{it:D4} $H^{i}(\CC^{\ast k},D_\alpha,\ZZ)=0$, unless $i=k$.
\item \label{it:D5} $\rk H^{k}(\CC^{\ast k},D_\alpha,\ZZ)=k!\operatorname{vol}(P)$ (where the volume is computed with respect to the lattice $\ZZ^k\subset \RR^k$). 
\item \label{it:D6} $\rk H^{k}(D_\alpha,\ZZ)=k!\operatorname{vol}(P)+k-1$.
\item \label{it:D7}
The long exact sequence for relative cohomology for the pair $(\CC^{\ast k},D_\alpha)$ is $\pi_1(\Kscr_A,\alpha)$-equivariant.
\item \label{it:D8} The $\pi_1(\Kscr_A,\alpha)$ action on $H^i(\CC^{\ast k},\ZZ)\cong\ZZ^{k \choose i}$ is trivial. The same holds for $H^i(D_\alpha,\ZZ)$ for $i<k-1$.
\item \label{it:D9} The induced action of $\pi_1(\CC^A-V(A))$ 
on $H^{k-1}(\CC^{\ast k},D,\CC)$  (via the quotient map $\pi_1(\CC^A-V(A))\r \pi_1(\Kscr_A)$) 
is given by the
GKZ system  corresponding to the matrix $A$, with trivial parameters.
\end{enumerate}
\end{proposition}
\begin{proof}
We follow the mentioned references.
\begin{enumerate}
\item[\eqref{it:D2}] By \cite[Corollary 1.1.2]{Oka}, $D_\alpha$ has the homotopy type of a $(k-1)$-dimensional CW-complex.\footnote{This can be seen from the fact  $D_\alpha$ is a smooth affine variety 
combined with the Andreotti-Frankel theorem \cite{MR0177422}.}
 This proves \eqref{it:D2}.
\item[\eqref{it:D3}] By \cite[Corollary 1.1.1]{Oka} the natural map $j_i:H_{i}(D_\alpha,\ZZ)\r H_{i}(\CC^{\ast k},\ZZ)$ is an isomorphism for $i<k-1$ and a surjection for $i=k-1$. 
Since $H_{i}(\CC^{\ast k},\ZZ)$ is free for all $i$ we get in particular that $H_i(D_\alpha,\ZZ)$ is free for $i<k-1$. 
The
universal coefficient theorem gives an exact sequence for $i\le k-1$:
\begin{equation}
\label{eq:uct}
0\r \Ext^1_{\ZZ}(H_{i-1}(D_\alpha,\ZZ),\ZZ)\r H^{i}(D_\alpha,\ZZ)\r  \Hom_{\ZZ}(H_{i}(D_\alpha,\ZZ),\ZZ)\r 0.
\end{equation}
Hence we obtain for $i\le k-1$
\begin{equation}
\label{eq:uct2}
 H^{i}(D_\alpha,\ZZ)\cong  \Hom_{\ZZ}(H_{i}(D_\alpha,\ZZ),\ZZ).
\end{equation}
Hence the $j_i$ dualizes to the required isomorphisms/ split monomorphism (note that since $H_{k-1}(\CC^\ast,\ZZ)$ is free, $j_{k-1}$ is split).
\item[\eqref{it:D4}] From the long exact sequence for cohomology of pairs we get \eqref{it:D4} using \eqref{it:D2}\eqref{it:D3} and $H^i(\CC^{* k},\ZZ)=0$ for $i>k$.
\item[\eqref{it:D1}] The freeness of $H^i(\CC^{*k},\ZZ)$ is a basic fact. The freeness of $H^i(D_\alpha,\ZZ)$ follows from \eqref{it:D2} and \eqref{eq:uct2}. 
It follows from \eqref{it:D4} that it remains to show that $H^k(\CC^{*k},D_\alpha,\ZZ)$ is free. We use the corresponding part of the long exact sequence 
\begin{equation}\label{eq:sesH_}
0\r H^{k-1}(\CC^{\ast k},\ZZ)\r H^{k-1}(D_\alpha,\ZZ)\r H^{k}(\CC^{\ast k},D_\alpha,\ZZ)\r H^{k}(\CC^{\ast k},\ZZ)\r 0.
\end{equation}
The fact that the first map is a split monomorphism by \eqref{it:D3} finishes the proof.
\item[\eqref{it:D6}] This is \cite[theorem 4.6]{batyrev2}.
\item[\eqref{it:D5}] This follows from \eqref{it:D6} and \eqref{eq:sesH_}.
\item[\eqref{it:D7}] We will now discuss \eqref{it:D7}.
First consider the commutative diagram
\begin{equation}
\label{eq:familydiagram}
\begin{tikzcd}
\Dscr \ar[r,hookrightarrow]\ar[dr,"p_0"'] & \CC^{\ast k} \times (\CC^A-V(A))\ar[d,"p"]\\
& \CC^A-V(A)
\end{tikzcd}
\end{equation}
where $\Dscr$ is the family $(f^{-1}_\alpha(0))_{\alpha\in \CC^A-V(A)}$ and $p$ is the projection. Put  $B=\CC^A-V(A)$, $M=\CC^{\ast k} \times B$.

As in \cite[(IV.17)]{Bredon} we get a long exact sequence of sheaves on $B$
\begin{equation}
\label{eq:relative}
\cdots \r R^{i-1}p_\ast\underline{\ZZ}_M\r R^{i-1}p_{0,\ast}\underline{\ZZ}_{\Dscr}\r R^i(p,p_{0})_\ast\underline{\ZZ}_M\r R^{i}p_\ast\underline{\ZZ}_M\r  R^{i}p_{0,\ast}\underline{\ZZ}_{\Dscr}\r\cdots.
\end{equation}
Here $R^i(p,p_{0})_\ast\underline{\ZZ}_M$ is relative sheaf
cohomology for pairs. In a similar way as relative cohomology for
spaces, it is the sheaf associated to the presheaf
$U\mapsto H^i(p^{-1}(U),p_0^{-1}(U), \ZZ)$ (see \cite[Definition
IV.4.1]{Bredon}).
The family $p_0:\Dscr\r B$ is locally trivial by Theorem \ref{lem:moser2} and Theorem \ref{prop:Dlocally}. Moreover $p:M\r B$ is trivial by definition. It follows that
$R^{i}p_\ast\underline{\ZZ}_M$,   $R^{i}p_{0,\ast}\underline{\ZZ}_{\Dscr}$ are locally constant with fibers in $\alpha\in B$, given by $H^i(\CC^{\ast k},\ZZ)$ 
and $H^i(D_\alpha,\ZZ)$ respectively. 
If follows from \eqref{eq:relative} that  $R^i(p,p_{0})_\ast\underline{\ZZ}_M$ is locally free as well and from the five-lemma we obtain that the restriction map
$(R^i(p,p_{0})_\ast\underline{\ZZ}_M)_\alpha \r H^i(\CC^{\ast k},D_\alpha, \ZZ)$ for $\alpha\in B$ is also an isomorphism.

So taking fibers we obtain that 
\begin{multline}
\label{eq:absolute}
\cdots\r H^{i-1}(\CC^{\ast k},\ZZ)\r H^{i-1}(D_\alpha,\ZZ)\r H^{i}(\CC^{\ast k},D_\alpha,\ZZ)
\\\r H^{i}(\CC^{\ast k},\ZZ)\r  H^{i}(D_\alpha,\ZZ)  \r \cdots.
\end{multline}
is $\pi_1(B,\alpha)$-equivariant.  
 To prove that the $\pi_1(B,\alpha)$-action descends to an
action of $\pi_1(\Kscr_A,\alpha)$ we have to prove that
$\pi_1(\CC^{\ast k},1)$ acts trivially (see \eqref{eq:presentation}).
Let $\gamma:S^1\r \CC^{\ast k}$ be an element in
$\pi_1(\CC^{\ast k},1)$. This defines a path
$\Gamma:S^1\r B:t\mapsto \gamma(t)\alpha$ and we have to prove that the pullback of
\eqref{eq:relative} under $\Gamma$ consists of constant sheaves. But this
follows from the easily verified fact that the pullback of
\eqref{eq:familydiagram} under $\Gamma$ is trivial, using the fact that \eqref{eq:familydiagram} is $\CC^{\ast k}$-equivariant.
\item[\eqref{it:D8}] Since $M\r B$ is trivial (with the notation introduced in the previous item), the action of  $\pi_1(B,\alpha)$ on  $H^{i}(\CC^{\ast k},\ZZ)$ is trivial. This, combined with \eqref{it:D3} gives \eqref{it:D8}.
\item[\eqref{it:D9}] This follows from \cite[Remark 2.14]{Reichelt}, \cite{Stienstra}. \qedhere
\end{enumerate}
\end{proof}

\section{Star-shaped triangulations and Weinstein structure} 
We place ourselves in the context of \S\ref{sec:setting}. In particular the triangulation $\Delta_\nu$ is star-shaped, and so Convention \ref{conv:starshaped} applies.  
In this section we show that $f_{t,1}^{-1}(0)$ admits a (Morse-Bott) Weinstein structure\footnote{This result is already claimed without proof in  \cite[Theorem 2.19]{Gammage}. 
Moreover, the proof of the Morse-Bott property (\S\ref{sec:morsebott}) goes along the same lines as the claim at the end of \cite[p.18]{Zhou} (Zhou, private communication).} 
which is necessary
to complete the proof of Theorem \ref{th:mirror_} (recall that the proof of that theorem relies on the homological mirror symmetry results of \cite{GammageShende,GPS2,GPS3}).
\begin{theorem} \label{th:weinstein}
Assume $\psi:(\CC^\ast)^k\r \RR$ is a potential which is homogeneous of degree two outside a compact set and which is adapted to $C_0$ (see Definition \ref{def:adapted}).
One may choose $1\gg \epsilon_2>\epsilon_1>0$ such that for $t\gg 0$ one has
\begin{enumerate}
\item \label{it:weinstein1_} Let $\psi_1$ be the restriction of $\psi$ to $f^{-1}_{t,s}(0)$ and let $Z_1$ be the Liouville vector field on $f^{-1}_{t,s}(0)$
(see Theorem \ref{prop:tropical}).
Then $Z_1$ is gradient like for $\psi_1$.
\item \label{it:weinstein2_}
Let $s=1$. The critical points of $\psi_1$ are Morse-Bott.
\end{enumerate}
\end{theorem} 
Theorem \ref{th:weinstein}\eqref{it:weinstein1_} will be proved in \S\ref{sec:gradientlike} and Theorem \ref{th:weinstein}\eqref{it:weinstein2_} will be proved in \S\ref{sec:morsebott}.
\subsection{Gradient-like property}
\label{sec:gradientlike}
The proof of Theorem \ref{th:weinstein}\eqref{it:weinstein1_} consists of a number of steps. We assume that $\psi$ is a K\"ahler potential on $(\CC^\ast)^k$ which in logarithmic  polar
coordinates $(w_j=\rho_j+i\eta_j)_j$ depends only on $(\rho_i)_i\in \RR^k$. We use then notations $g,Z,\omega,\theta$ associated with $\psi$ (as in \S\ref{sec:zhou}). 
We will
also view $g$ as a Riemannian metric on $\RR^k$ (with matrix entries
 $\partial_{\rho_i}\partial_{\rho_j}\psi$, see \eqref{eq:zhoumetric}).  We put $\bar{Z}(\Log_t z)=d(\Log_t)(Z(z))$ which yields 
\begin{equation}
\label{eq:Zscaling}
\bar{Z}(\Log_t z)=Z(z)/\log t.
\end{equation}
It will be convenient to define $\bar{\psi}(q)=\psi(q \log t)/(\log t)^2$ and
to let $\bar{g}$ be the associated metric. In that case one calculates using \eqref{eq:gradg100} 
\begin{equation}
\label{eq:bargrad}
\bar{Z}=\grad_{\bar{g}}(\bar{\psi})\,.
\end{equation}
 
By $|{-}|$ we denote the norm for the Euclidean metric $\sum_i (d\rho_i)^2$ on $\RR^k$ and we let $d(-,-)$ be the corresponding distance.
We do not assume that $\Delta_\nu$ is star-shaped, except in Lemma \ref{lem:simpler}.
We will consider some hypotheses.
\def\thehypothesis{A}\begin{hypothesis} \label{hyp:equivalent} $\psi$ is homogeneous of degree two outside a compact set. 
\end{hypothesis}
\begin{remark} \label{rem:confusion}
It follows from Hypotheses A that $\bar{\psi}=\psi$ outside a compact set. Moreover from \eqref{eq:gradg1}\eqref{eq:bargrad} it follows that $\bar{Z}=\sum_i \rho_i\partial_{\rho_i}$ outside a compact set (i.e.\ the same expression as for $Z$) and this is how it will be used. 
Note that the compact set shrinks to zero if $t\r \infty$ although this fact will not be used.
\end{remark}

\begin{notation}\label{notations:gradientlikesection}
For $\tau$  a face of $\Delta_\nu$ we put $\rho_\tau=\rho_{L_\tau}$ (\S\ref{sec:polyhedral}, Lemma \ref{lem:extremal}). For $p\in \RR^k$ we let $L_{\tau,p}$ 
be the affine space which is the translate of $L_\tau$
passing through $p$.  We let $\rho_{\tau,p}$ be the point of $L_{\tau,p}$ where $\psi$ attains its minimum (see Lemma \ref{lem:extremal}).
We also put $\tau_p=A_0(p)$, $L_p=L_{\tau_p,p}$, $\rho_p=\rho_{\tau_p,p}$.
\end{notation}
We make the following hypothesis.  
\def\thehypothesis{B}\begin{hypothesis} \label{hyp:rhoF} There exists $\epsilon>0$ such that 
if $A_1(p)\neq\emptyset$ then $d(p,\rho_p)> \epsilon$.
\end{hypothesis}
The following result will be shown below.
\begin{proposition} \label{prop:main_taming} Let $Z_1$ be the Liouville vector field on $f^{-1}_{t,s}(0)$ corresponding to the restriction of 
$\psi$. Assume Hypotheses \ref{hyp:equivalent},\ref{hyp:rhoF} hold. Then for $t\gg 0$ one has that $Z_1$ is gradient like for the restriction of $\psi$ to $f^{-1}_{t,s}(0)$ (cf. \cite[(9.9)]{CiEl}); i.e. there is a constant $\delta>0$ such that on $f^{-1}_{t,s}(0)$ we have 
\[
Z_1\psi\ge \delta(|d\psi_1|^2_g+|Z_1|_g^2)
\]
where $\psi_1$ is  the restriction of $\psi$ to $f^{-1}_{t,s}(0)$.
\end{proposition}
Next we give an easier to verify criterion to check Hypothesis \ref{hyp:rhoF}.
\begin{lemma} \label{lem:epsilon_choice}
Assume that for every face $C_{\tau}$ in $\Pi_\nu$ we have $\rho_{\tau}\not\in \partial C_\tau$.
Then we may choose $\epsilon$, $\epsilon_1$, $\epsilon_2$ in such a way that Hypothesis \ref{hyp:rhoF} holds.
\end{lemma}
This criterion becomes even simpler to check in the case of a star-shaped triangulation.
\begin{lemma} \label{lem:simpler} Assume $0\in A$ and $\nu(0)=0$ is
  the unique minimal value of $\nu(a_i)_i$. Assume that $\Delta_\nu$
  is a star-shaped 
triangulation of $P$ (see
  \S\ref{sec:stacky},\S\ref{sec:setting}). Assume
  that for all faces $C_\tau$ of $\partial C_0$ we have
  $\rho_\tau\in \relint C_\tau$.  Then, if $C_{\tau}$ is a face of
  $\Pi_\nu$ which is not in $\partial C_0$, one has
  $\rho_{\tau}\not\in C_{\tau}$. In particular the hypotheses for
  Lemma \ref{lem:epsilon_choice} hold. 
\end{lemma}
Combining Proposition \ref{prop:main_taming} with Lemmas \ref{lem:epsilon_choice},\ref{lem:simpler} and Definition \ref{def:adapted} yields the proof
of Theorem  \ref{th:weinstein}\eqref{it:weinstein1_}.
\subsubsection{Proofs}
For an affine subspace $L\subset \RR^k$ we  let $\bar{Z}^\perp_L$ be the normal vector field on $L$ 
such that for $q\in L$, $\bar{Z}^\perp_L(q)$ is the projection of $\bar{Z}(q)$ on the normal space for $g$
to $L$ at $q$.
\begin{lemma} \label{lem:R}
\begin{enumerate}
\item Let $q\in L$.
We have $|\bar{Z}^\perp_L(q)|_g=|\bar{Z}(q)|_g$ if and only if $q=\rho_L$. 
\item If Hypothesis \ref{hyp:equivalent} holds then
\[
\forall a>0:\forall \epsilon>0:\exists R>0:\forall L:d(0,L)<a:\forall q\in L: d(0,q)>R\Rightarrow \frac{|\bar{Z}^\perp_L(q)|_g}{|\bar{Z}(q)|_g}<\epsilon.
\]
\end{enumerate}
\end{lemma}
\begin{proof} We have $|\bar{Z}^\perp_L(q)|_g=|\bar{Z}(q)|_g$ if and only if $\bar{Z}(q)\perp_g L$. In other words: if and only
if for all $v\in T_q(L)$: $g(\bar{Z}(q),v)=0$. Since $\bar{Z}=\grad_{\bar{g}} \bar{\psi}$ (see \eqref{eq:bargrad})
we have $\bar{g}(\bar{Z}(q),v)=(v\bar{\psi})(q)$. So we have $|\bar{Z}^\perp_L(q)|=|\bar{Z}(q)|$ if and only if $q$ is an extremal point
for $\psi$ on $L$. By Lemma \ref{lem:extremal} this is only possible if $q=\rho_L$.

To prove the second claim fix $a> 0$, $\epsilon>0$. We note that we have
\begin{equation}
\label{eq:angle_limit}
 \frac{|\bar{Z}^\perp_L(q)|_g}{|\bar{Z}(q)|_g}=\min_{v\in T_q(L)} \frac{|\bar{Z}(q)-v|_g}{|\bar{Z}(q)|_g}.
\end{equation}
Recall that if Hypothesis \ref{hyp:equivalent} holds
then by Remark \ref{rem:confusion} outside a compact set we have $\bar{Z}=\sum_i \rho_i\partial_{\rho_i}$.
By inspection one sees
that a suitable $R$ exists if we replace $|{-}|_g$ by $|{-}|$.
Combining \eqref{eq:angle_limit} with Proposition \ref{prop:bounds} finishes
the proof. 
\end{proof}

\begin{proof}[Proof of Lemma \ref{lem:epsilon_choice}] We put 
\[3\epsilon = \min_{|\tau|\ge 2} d(\partial C_{\tau}, \rho_{\tau}).
\]

We choose $0<\epsilon_2\ll 1$ in such a way that for faces $C_{\tau'}$ in $\Pi_\nu$ and for all 
$p\in C_{\tau',\epsilon_2}$ (see \eqref{eq:tau_epsilon}) we have
$d(p,C_{\tau'})<\epsilon$. This is possible by \eqref{eq:compact1}.

We choose $0<\epsilon_1<\epsilon_2$ in such a way that for all faces $C_\tau$ in $\Pi_\nu$ and for all $p\in C_{\tau,\epsilon_1}$
we have $d(\rho_p,\rho_\tau)<\epsilon$. This is possible by combining \eqref{eq:compact1}  with Lemma \ref{lem:extremal}.

Assume $A_1(p)\neq \emptyset$. Let $\tau'=A_0(p)\cup A_1(p)$, $\tau=A_0(p)$. By definition $p\in C_{\tau',\epsilon_2}\cap C_{\tau,\epsilon_1}$.
We obtain
\begin{align*}
d(p,\rho_p)&\ge d(p,\rho_\tau)-d(\rho_p,\rho_\tau)\\
&\ge d(C_{\tau'},\rho_\tau)-d(p,C_{\tau'})-d(\rho_p,\rho_\tau)\\
&> 3\epsilon-\epsilon-\epsilon\\
&=\epsilon.\qedhere
\end{align*}
\end{proof}

\begin{proof}[Proof of Lemma \ref{lem:simpler}] Assume $C_\tau$ is not in $\partial C_0$ and yet $\rho_\tau\in C_\tau$.  Since $C_\tau$ is not in $\partial C_0$ we have
  $0\not\in \tau$. We follow Convention \ref{conv:starshaped}; i.e. $a_1=0$ and
  $\tau=\{a_2,\ldots,a_s\}$ ($s\ge 3$). Put
  $\hat{\tau}=\tau\cup \{0\}$ and $\hat{\tau}_{i}=\hat{\tau}-\{a_i\}$
  for $i=1,\ldots,s$. Then $C_{\hat{\tau}_i}$ are the neighbouring
  faces of $C_{\hat{\tau}}=C_\tau\cap \partial C_0$. As
  $0\in \hat{\tau}_i$ for $i\neq 1$ we have that $C_{\hat{\tau}_i}$
  for $i\neq 1$ are the neighbouring faces of $C_{\hat{\tau}}$ in
  $\partial C_0$. 
By hypothesis $\rho_{\hat{\tau}}\in  \relint C_{\hat{\tau}}$ and $\rho_{\hat{\tau}_i}\in \relint C_{\hat{\tau}_i}$ for $i=2,\ldots,s$.
By Lemma
  \ref{lem:extremal} we have $\psi(\rho_{\hat{\tau}_i})<\psi(\rho_{\hat{\tau}})$ for all
  $i=2,\ldots,s$.

Let $M$ be the convex hull of $(\rho_{\hat{\tau}_i})_{i=2,\ldots,s}$.
By Lemma \ref{lem:single_point} below (with $\tau=\hat{\tau}$, $\rho_i=\rho_{\hat{\tau}_i}$) $L_\tau(=L_{\hat{\tau}_1})$ and $M$ intersect in a single point $m$ and
moreover the interval $[\rho_\tau,m]$ intersects $L_{\hat{\tau}}$ in a single point $k$. By Lemma \ref{lem:maximum}, $\psi(m)\le \max_{i=2,\ldots,s} \psi(\rho_{\hat{\tau}_i})<\psi(\rho_{\hat{\tau}})\le \psi(k)$. So we get $\psi(m)<\psi(k)\ge \psi(\rho_{\hat{\tau}})\geq \psi(\rho_{\hat{\tau}_1})$. By the strict convexity of $\psi{\mid} L_{\tau}$
this in only possible if $k=\rho_\tau$. But then $\psi(m)<\psi(\rho_\tau)$ which contradicts
the definition of $\rho_\tau$.
\end{proof}
We have used the following lemma which does not depend on the triangulation being star-shaped.
\begin{lemma} \label{lem:single_point} Let $\tau=\{a_{i_1},\ldots,a_{i_s}\}$ be a face of $\Delta_\nu$ and put $\tau_j=\tau-\{a_{i_j}\}$. For $i=2,\ldots,s$ let $\rho_i$ be points in the relative interior of 
$C_{\tau_i}$ and $\rho_1\in C_{\tau_1}$. Let $M$ be the convex hull of $(\rho_i)_{i=2,\ldots,s}$. Then $L_{\tau_1}\cap M$ is a single point $m$ and moreover the interval $[\rho_1,m]$ intersects
$L_\tau$ in a single point.
\end{lemma}
\begin{proof} After rearranging the elements of $A$ we may assume $\tau=\{a_1,\ldots,a_s\}$. We put $\Hscr_i=\Hscr(a_i)$ (see \S\ref{sec:polyhedral}).
After performing an affine coordinate transformation and replacing $\Hscr_i$ by $\Hscr_i-\Hscr_1$ we may assume that $\Hscr_1$ is the zero
  function and $\Hscr_i(p)=p_i$, for $i=2,\ldots,s$, where $p_i$ is the $i$'th coordinate
  of $p$. According to \eqref{eq:Cequations},\eqref{eq:Ltau} we find that $L_\tau$ (resp. $C_\tau$) are given by
\begin{gather*}
0=p_2=\cdots=p_s\\
(\text{resp. } 0=p_2=\cdots=p_s, \qquad p_j\ge \Hscr_{k}(p)\text{ for $j\le s$, $k>s$}),
\end{gather*}
$L_{\tau_i}$ (resp. $C_{\tau_i}$) for $i=2,\ldots,s$ are given by the equations
\begin{gather*}
0=p_2=\cdots=p_{i-1}=p_{i+1}=\cdots=p_s\\
(\text{resp. } 0=p_2=\cdots=p_{i-1}=p_{i+1}=\cdots=p_s, \quad p_i\le 0, \\\quad p_j\ge \Hscr_{k}(p) \text{ for $j\le s$, $j\neq i$,  $k>s$}),
\end{gather*}
whereas $L_{\tau_1}$ (resp. $C_{\tau_1}$) are given by
\begin{gather*}
p_2=\cdots=p_s\\
(\text{resp. } p_2=\cdots=p_s\ge 0, \quad p_j\ge \Hscr_{k}(p)\text{ for $j\le s$, $s\neq 1$,  $k>s$}).
\end{gather*}
So we have for $j=2,\ldots,s$.
\[
\rho_j=(\underbrace{0,\ldots,0,p^{(j)},0,\ldots,0}_{s\text{ coordinates}},\text{unknown})
\]
where $p^{(j)}<0$ and
\[
\rho_1=(?,p^{(1)},\ldots,p^{(1)},\text{unknown})
\]
$p^{(1)}\ge 0$.
To compute $L_{\tau_1}\cap M$ we have to find the
$(x_i)_{i=2,\ldots,s}$, $x_i\ge 0$, $\sum_i x_i=1$ such that
$ \sum_{j=2}^s x_j \rho_j$ has its $2$ to $s$ coordinates equal.  It is
clear that there is a unique solution given by
$x_j=1/p^{(j)}/(\sum_{j=2}^s 1/p^{(j)})$. So we get
\[
m=(0,s,\ldots,s,\text{unknown})
\]
with $s=1/(\sum_{j=2}^s 1/p^{(j)})<0$.  We find that the unique point in $[\rho_1,m]\cap L_\tau$
is given by $((1/p^{(1)})\rho_1-(1/s) m)/(1/p^{(1)}-1/s)$ if $p^{(1)}\neq 0$ and $\rho_1$ if $p^{(1)}=0$.
\end{proof}

We need the following lemma for the proof of Proposition \ref{prop:main_taming}. 
\begin{lemma} \label{lem:perpbound} Assume that Hypotheses \ref{hyp:equivalent},\ref{hyp:rhoF} hold. Then there exists $h<1$ such that for 
all $p\in \RR^k$ such that $A_1(p)\neq \emptyset$ we have
\begin{equation}
\label{eq:perpbound}
\frac{|\bar{Z}^\perp_{L_p}(p)|_g}{|\bar{Z}(p)|_g}\le h
\end{equation}
\end{lemma}
\begin{proof}

We fix a face $\tau$ in $\Delta_\nu$. 
The map $\RR^k\r \RR^k:q\mapsto \rho_{\tau,q}$ (see Notation \ref{notations:gradientlikesection}) is continuous by Lemma \ref{lem:extremal}.
If follows that the map $d:\RR^k\r \RR:q\mapsto d(q,\rho_{\tau,q})$ is continuous as well. 
Put  $O_\tau=\coprod_{q\in \RR^k} (B(\rho_{\tau,q},\epsilon)\cap L_{\tau,q})$ where $\epsilon$ is as in the statement of Hypothesis \ref{hyp:rhoF}. Since  $O_\tau=d^{-1}(]-\infty,\epsilon[)$ we obtain that $O_\tau$ is open.

Now fix $a>0$ and choose $R$ as in Lemma \ref{lem:R} with $\epsilon=1/2$ 
(here we use Hypothesis \ref{hyp:equivalent}). Set
\[
h_{\tau,a}=\max\left(\sup_{q\in \overline{B(0,R)}\setminus O_\tau} \frac{|\bar{Z}^\perp_{L_{\tau,q}}(q)|_g}{|\bar{Z}(q)|_g},1/2\right)
\]
(if $\overline{B(0,R)}\setminus O_\tau=\emptyset$ then $h_{\tau,a}=1/2$). Since $\overline{B(0,R)}\setminus O_\tau$ is compact, the supremum is a maximum and hence we have $1/2\le h_{\tau,a}<1$, where the strict inequality follows from the first claim in Lemma \ref{lem:R}. 
By the choice of $R$ we get for all $q\in \RR^k-O_\tau$ such that $d(0,L_{\tau,q})<a$:
\[
\frac{|\bar{Z}^\perp_{L_{\tau,q}}(q)|_g}{|\bar{Z}(q)|_g}\le h_{\tau,a}.
\]
Now note that by \eqref{eq:compact1}, $p$ is close to $C_{\tau_p}$ and in particular it is close to $L_{\tau_p}$ (the distance can be bounded by a constant independent of $p$).  Since there are only a finite number of possibilities
for $L_{\tau_p}$ we obtain that  there exists $A>0$ such that for all $p\in \RR^k$: $A> d(0,L_{\tau_p,p})$. Hence we obtain if $p\not\in O_{\tau_p}$:
\[
\frac{|\bar{Z}^\perp_{L_{\tau_p,p}}(p)|_g}{|\bar{Z}(p)|_g}\le h_{\tau_p,A}.
\]

Hypothesis \ref{hyp:rhoF} asserts that $A_1(p)\neq \emptyset$, thus $p\not\in O_{\tau_p}$. We conclude
by putting $h=\max_\tau h_{\tau,A}$.
\end{proof}

\begin{proof}[Proof of Proposition \ref{prop:main_taming}]
Let $z\in f^{-1}_{t,s}(0)$. If $f_{t,s}$ is holomorphic at $z$ then $Z_1=\grad_g \psi_1$ and hence $Z_1$ is gradient like.
So assume that $f_{t,s}$ is not holomorphic at~$z$. In particular for $p=\Log(z)/\log t$ we have $A_1(p)\neq\emptyset$.

Throughout we assume $t\gg 0$. 
First assume $Z(z)\neq 0$ and put:
\[
q=\frac{|(\partial f_{t,s}(z))(Z(z))|}{|\partial f_{t,s}(z)|_g |Z(z)|_g}
\]
(thus $\cos q$ is the Hermitian angle between $\partial f_{t,s}(z)$ and $Z(z)$ as defined in \S\ref{sec:hermitian}). Pick $r>0$. 
We first show that it is sufficient to prove $q<1$. If this is the case then we pick $0<r<1-q^2$.
By Lemma \ref{lem:Cbound} we have for $t\gg 0$  and for all $z\in f^{-1}_{t,s}(0)$:
$|\bar{\partial} f_{t,s}(z)|_g< r|{\partial} f_{t,s}(z)|_g$. By Corollary \ref{cor:gradient} we see that $1-q^2>r$ implies the existence of $\delta>0$
such that
\begin{equation}
\label{eq:gformula}
g(Z,Z_1)\ge \delta(|Z|^2_g+|Z_1|^2_g).
\end{equation}
Since $Z=\grad \psi$ this may be rewritten as 
\[
Z_1\psi\ge \delta(|d\psi|^2_g+|Z_1|_g^2)\ge \delta(|d\psi_1|^2_g+|Z_1|^2_g)
\]
which is what we want. If $Z=0$ then \eqref{eq:gformula} is true for any $\delta$ so there is nothing to prove. So we continue to assume $Z(z)\neq 0$.

Fix $h<h'<1$ with $h$ as in Lemma \ref{lem:perpbound}.
We will prove $q<h'$ which implies $q<1$.  For use below put:
\begin{align*}
q'&=\frac{|(\partial f_{t,s}^{\text{core}}(z))(Z(z))|}{|\partial f^{\text{core}}_{t,s}(z)|_g |Z(z)|_g}\\
q''&=\frac{|(\partial f_{t,s}^{\text{core}}(z))(Z^\perp_{L_p}(z))|}{|\partial f^{\text{core}}_{t,s}(z)|_g |Z(z)|_g}.
\end{align*}
Put
\[
c=\frac{|\partial f^{\text{small}}_{t,s}(z)|_g}{|\partial f^{\text{core}}_{t,s}(z)|_g},\qquad c'=\frac{|f^{\text{small}}_{t,s}(z)|}{|\partial f^{\text{core}}_{t,s}(z)|_g}.
\]
Recall that by \eqref{eq:parbound}\eqref{eq:negligible}\eqref{eq:mainbound}, $c$, $c'$ become arbitrarily small (uniformly in $z\in f^{-1}_{t,s}(0)$ for $t\gg 0$).
We get 
\begin{align*}
q&\le \frac{|(\partial f_{t,s}^{\text{core}}(z))(Z(z))|+|(\partial f_{t,s}^{\text{small}}(z))(Z(z))|}{(|\partial f^{\text{core}}_{t,s}(z)|_g- |\partial f^{\text{small}}_{t,s}(z)|_g)|Z(z)|_g}\\
&\le \frac{|(\partial f_{t,s}^{\text{core}}(z))(Z(z))|+|(\partial f_{t,s}^{\text{small}}(z))|_g|Z(z)|_g}{(|\partial f^{\text{core}}_{t,s}(z)|_g- |\partial f^{\text{small}}_{t,s}(z)|_g)|Z(z)|_g}\\
&= \frac{|(\partial f_{t,s}^{\text{core}}(z))(Z(z))|+c|(\partial f_{t,s}^{\text{core}}(z))|_g|Z(z)|_g}{(1-c)|\partial f^{\text{core}}_{t,s}(z)|_g|Z(z)|_g}\\
&= \frac{q'+c}{1-c}.
\end{align*}
Write $Z(p)=Z^\perp_{L_p}(p)+Z^{\quot\!}(p)$ and consider the corresponding pullback under $\log t$ scaling
$Z(z)=Z^\perp_{L_p}(z)+Z^{\quot\!}(z)$ (see \eqref{eq:Zscaling})
We have
\[
(\partial f_{t,s}^{\text{core}})(Z^{\quot\!})=Z^{\quot} (f_{t,s}^{\text{core}}).
\]
Put $\tau=A_0(p)$. Because $Z^{\quot\!}$ is parallel with $L_{p}$, we have $(a_i-a_j)\cdot Z^{\quot}=0$ for $a_i,a_j\in \tau$.
Hence if $a\in \tau$ then $Z^{\quot\!}(z^{-a} f^{\text{core}}_{t,s})=0$. In other words  
\begin{align*}
Z^{\quot\!}(f^{\text{core}}_{t,s})&= Z^{\quot\!}(z^{a}) z^{-a} f^{\text{core}}_{t,s}\\
&=z^{-a} \partial (z^{a})(Z^{\quot}) f^{\text{core}}_{t,s}.
\end{align*}
We obtain
\begin{align*}
|Z^{\quot\!}(f^{\text{core}}_{t,s})|\le |z|^{-a} |\partial (z^{a})|_g|Z^{\quot}|_g| f^{\text{core}}_{t,s}|
&\le C|Z|_g| f^{\text{small}}_{t,s}|\\
&\le Cc'|Z|_g| \partial f^{\text{core}}_{t,s}|_g
\end{align*}
where $C$ is a constant depending only on $A$ and $\psi$. In the second inequality we used $0=f_{t,s}(z)=f^{\text{core}}_{t,s}(z)+f^{\text{small}}_{t,s}(z)$.
We deduce
\begin{align*}
q'&\leq q''+
\frac{|(\partial f_{t,s}^{\text{core}}(z))(Z^{\quot\!}(z))|}{|\partial f^{\text{core}}_{t,s}(z)|_g |Z(z)|_g}\le q''+Cc'
\end{align*}
and hence
\[
q\le \frac{q''+Cc'+c}{1-c}.
\]
Taking into account that $c,c'$ can be made arbitrarily small it follows that is sufficient to prove $q''<h'$.
We get
\begin{align*}
q''&=\frac{|(\partial f_{t,s}^{\text{core}}(z))(Z^\perp_{L_p}(z))|}{|\partial f^{\text{core}}_{t,s}(z)|_g |Z(z)|_g}\\
&\le\frac{|\partial f_{t,s}^{\text{core}}(z)|_g|Z^\perp_{L_p}(z)|_g}{|\partial f^{\text{core}}_{t,s}(z)|_g |Z(z)|_g}\\
&=\frac{|Z^\perp_{L_p}(z)|_g}{ |Z(z)|_g}\\
&=\frac{|\bar{Z}^\perp_{L_p}(p)|_g}{ |\bar{Z}(p)|_g}\\
&\le h<h'.\qedhere
\end{align*}
\end{proof}
\subsection{Morse-Bott property}
\label{sec:morsebott}
The following lemma proves Theorem \ref{th:weinstein}\eqref{it:weinstein2_}. 

Recall that a Morse–Bott function is a smooth function  whose critical set is a closed submanifold and such that the kernel of the Hessian at a critical point equals the tangent space to the critical submanifold.
\begin{lemma}\label{lem:MB} 
Let $\psi:\RR^k-\{0\}\r \RR$ be a potential which is homogeneous of degree two outside a compact set and which is adapted to $C_0$. 
The potential $\psi\circ \Log:f_{t,1}^{-1}(0)\to \RR$ is Morse-Bott for $t\gg 0$. 
\end{lemma}
\begin{proof}
Zhou describes the critical manifolds of $f^{-1}_{t,1}(0)$ \cite[Theorem 1,Proposition
5.9]{Zhou}. They are labelled by the faces $\tau$ of $\partial\Tscr$ which are the faces of $\Tscr$ that do not contain the origin. There are critical points $\tilde{\rho}_\tau$ for $\Log_t$
on the
boundary of the amoeba of $f^{-1}_{t,1}(0)$ (i.e. the image of $f^{-1}_{t,1}(0)$ by $\Log_t$), such that the preimages by
$\Log$ are the critical manifolds
$T_{\tau}\times\{\tilde{\rho}_\tau\} \subset S_{N^*}\times N^*_\RR$ where 
$T_\tau=\{\theta\in S_{N^*}\mid \langle\theta,a_i\rangle =0 \text{ for
  each vertex $a_i \in \tau$}\}$ \cite[Proposition 5.5]{Zhou}.
 By \cite[Proposition 4.2]{Zhou}, we may choose for every $\epsilon >0$ a $t\gg 0$ such  $|\tilde{\rho}_\tau-\rho_\tau|<\epsilon$ where $\rho_\tau$ is the minimum of $\psi$ on the face of $C_0$ dual to $\tau$.

We first change the coordinates. Let $\tau'$ be a facet of $\Tscr$ with a face $\tau$. We may relabel $A$ and assume that the vertices of $\tau$ are given by $a_1,\dots,a_j$ and the vertices of $\tau'$ by $a_1,\dots, a_k$.

In a neighbourhood of the critical manifold containing $(y_i)_i:=(t^{(\tilde{\rho}_\tau)_i})_i$,   
$f^{-1}_{t,1}(0)$ is defined by
\begin{equation}
\label{eq:transpants0}
\sum_{i=1}^j t^{-\nu_i}z^{a_i}=1
\end{equation}
(see \cite[proof of Proposition 5.5]{Zhou}).  
 The dependence of $(y_i)_{i=1,\ldots,k}$  on $t$
will be suppressed in the notations for now. Using the \'etale map $\iota:\CC^{\ast k}\r \CC^{\ast k}$ from \eqref{eq:etalemappants} we may reduce to the
case $a_i=e_i$ were $e_i$ is the $i$'th basis vector of $\ZZ^k$. We do this from now so that the equation in a neighbourhood of the critical manifold of $y$ becomes
\begin{equation}
\label{eq:transpants}
\sum_{i=1}^j t^{-\nu_i}z_i=1.
\end{equation}
The critical manifold containing $y$ is the $T_\tau=(S^1)^{k-j}$-orbit of $y$. It is given by
\[
\{z\mid |z_i|=|y_i|\text{ for $i>j$}\}.
\]
In particular it has dimension
$k-j$. We want to prove that for $t\gg 0$, $\psi\circ\Log$ restricted
to \eqref{eq:transpants} is Morse-Bott in a neighbourhood of this
critical manifold. By the $(S^1)^{k-j}$-action we only have to consider a neighbourhood of $y$.
Thus in a neighbourhood of $y$, $\psi\circ \Log$
subject to \eqref{eq:transpants} should be described by a quadratic
function of rank $2k-2-(k-j)=k+j-2$. 

In a neighbourhood of $\Log y=\log(t) \tilde{\rho}_\tau$, $\psi$ will be homogeneous of degree two for $t\gg 0$. So we may pretend that $\psi$ is homogeneous of
degree everywhere, except in $\rho=0$. We do this from now.

Before we embark on the proof we note that $\psi$ is still adapted to $C_0$ when computed for \eqref{eq:transpants0} and \eqref{eq:transpants} (for \eqref{eq:transpants0} this follows
from Lemma \ref{lem:extremal} and the passage to  \eqref{eq:transpants} amounts to a linear coordinate change).  Note that $C_0$ computed for \eqref{eq:transpants}  is the region
$\rho_i\le \nu_i$ for $i=1,\ldots,j$ and moreover $(\rho_\tau)_i=\nu_i$ for  $i=1,\ldots,j$. By looking at the coordinate half lines we obtain from the definition of adaptedness
\begin{equation}
\label{eq:psipositivity}
\frac{\partial \psi}{\partial \rho_i}(\rho_\tau)>0.
\end{equation}
To continue we first need to understand the variation of $\Log$. Since
\[
\log |z_i|=\frac{1}{2}\left( \log z_i+\log \bar{z}_i\right)
\]
(for suitable branches of $\log$) we obtain
\[
\log |z_i+\delta_i|=\log|z_i|+ \frac{1}{2}\left(\frac{\delta_i}{z_i}+
\frac{\bar{\delta}_i}{\bar{z}_i}\right)
-\frac{1}{4}\left(\frac{\delta_i^2}{z_i^2}+\frac{\bar{\delta}_i^2}{\bar{z}_i^2}\right)+\cdots
\]
Also
\begin{align*}
\psi(\Log(z+\delta))&=\psi\left(\left(\log|z_i|+\frac{1}{2}\left(\frac{\delta_i}{z_i}+
\frac{\bar{\delta}_i}{\bar{z}_i}\right)
-\frac{1}{4}\left(\frac{\delta_i^2}{z_i^2}+\frac{\bar{\delta}_i^2}{\bar{z}_i^2}\right)+\cdots\right)_{i=1}^k\right)\\
&=\psi(\Log(z))
+\frac{1}{2}\sum_{i=1}^k \frac{\partial \psi}{\partial \rho_i}(\Log(z))\left(\frac{\delta_i}{z_i}+
\frac{\bar{\delta}_i}{\bar{z}_i}\right)
\\
&\qquad-\frac{1}{4}\sum_{i=1}^k \frac{\partial \psi}{\partial \rho_i} (\Log(z))\left(\frac{\delta_i^2}{z_i^2}+\frac{\bar{\delta}_i^2}{\bar{z}_i^2}\right)
\\
&\qquad\qquad+\frac{1}{8}\sum_{i,l=1}^k \frac{\partial^2\psi}{\partial\rho_i\partial\rho_l}(\Log(z))  \left(\frac{\delta_i}{z_i}+
\frac{\bar{\delta}_i}{\bar{z}_i}\right)\left(\frac{\delta_l}{z_l}+
\frac{\bar{\delta}_l}{\bar{z}_l}\right)
+\cdots
\end{align*}
We have assumed that $y$ is a critical point of $\psi\circ \Log$, restricted to \eqref{eq:transpants}. In other words (writing $\rho=\Log y=\log(t) \tilde{\rho}_\tau$)
\begin{equation}
\label{eq:critical}
\sum_{i=1}^j t^{-\nu_i} \delta_i=0\Rightarrow  \sum_{i=1}^k \frac{\partial \psi}{\partial \rho_i}(\rho)\left(\frac{\delta_i}{y_i}+
\frac{\bar{\delta}_i}{\bar{y}_i}\right)=0
\end{equation}
which implies the existence of $\lambda\in \CC$ such that 
\begin{align}
\frac{1}{y_i}\frac{\partial\psi}{\partial \rho_i}(\rho)& =\lambda t^{-\nu_i}&&(\text{for $i=1,\ldots,j$})\label{eq:lambdadef}\\
\frac{\partial\psi}{\partial \rho_i}(\rho)&=0&&(\text{for $i=j+1,\ldots,k$}).
\end{align}
Since $\rho$ is not a global critical point for $\psi$ we have $\lambda\neq 0$. Multiplying \eqref{eq:lambdadef} with $y_i$ and summing over $i$ yields by \eqref{eq:transpants}
\[
\lambda=\sum_{i=1}^j \frac{\partial\psi}{\partial \rho_i}(\rho)
\]
so that
\begin{equation}
\label{eq:ytdependence}
y_i=t^{\nu_i} \frac{\displaystyle\frac{\partial\psi}{\partial \rho_i}(\rho)}{\displaystyle\sum_{l=1}^j \frac{\partial\psi}{\partial \rho_l}(\rho)}\qquad
\text{(for $i=1,\ldots,j$)}.
\end{equation}
Assuming $\sum_{i=1}^j t^{-\nu_i} \delta_i=0$ we obtain by \eqref{eq:critical}:
\begin{align*}
\psi(\Log(y+\delta))&=\psi(\rho)
-\frac{1}{4}\sum_{i=1}^j \frac{\partial \psi}{\partial \rho_i}(\rho) \left(\frac{\delta_i^2}{y_i^2}+\frac{\bar{\delta}_i^2}{\bar{y}_i^2}\right)
\\
&\qquad +\frac{1}{8}\sum_{i,l=1}^k \frac{\partial^2\psi}{\partial\rho_i\partial\rho_l}(\rho)  \left(\frac{\delta_i}{y_i}+
\frac{\bar{\delta}_i}{\bar{y}_i}\right)\left(\frac{\delta_l}{y_l}+
\frac{\bar{\delta}_l}{\bar{y}_l}\right)+\cdots
\end{align*}
If we put $\gamma_i=\delta_i/y_i$ then this becomes
\begin{align*}
\psi(\Log(y+\delta))&=\psi(\rho)-\frac{1}{4}\sum_{i=1}^j \frac{\partial \psi}{\partial \rho_i}(\rho) \left(\gamma_i^2+\bar{\gamma}^2_i\right)
\\
&\qquad+\frac{1}{8}\sum_{i,j=1}^k \frac{\partial^2\psi}{\partial\rho_i\partial\rho_l}(\rho)  \left(\gamma_i+
\bar{\gamma}_i\right)\left(\gamma_l+
\bar{\gamma_l}\right)+\cdots\\
&=\psi(\rho)-\frac{1}{2} \sum_{i=1}^j\frac{\partial \psi}{\partial \rho_i}(\rho)((\Re \gamma_i)^2-
(\Im\gamma_i)^2)\\
&\qquad+\frac{1}{2}\sum_{i,l=1}^k\frac{\partial^2\psi}{\partial\rho_i\partial\rho_l}(\rho)  \left(\Re\gamma_i\right)\left(\Re\gamma_l\right)+\cdots
\end{align*}
under the conditions
\begin{equation}
\label{eq:gammaconstraint}
\begin{aligned}
\sum_{i=1}^j t^{-\nu_i}y_i\Re\gamma_i&=0\\
\sum_{i=1}^j t^{-\nu_i}y_i\Im\gamma_i&=0.
\end{aligned}
\end{equation}
So we need to compute the rank of the quadratic form
\[
-\frac{1}{2} \sum_{i=1}^j\frac{\partial \psi}{\partial \rho_i}(\rho)((\Re \gamma_i)^2-
(\Im\gamma_i)^2)+\frac{1}{2}\sum_{i,l=1}^k\frac{\partial^2\psi}{\partial\rho_i\partial\rho_l}(\rho)  \left(\Re\gamma_i\right)\left(\Re\gamma_l\right)
\]
under the constraints \eqref{eq:gammaconstraint}.
This is equivalent to computing the ranks of
\begin{equation}
\label{eq:firstform}
\sum_{i=1}^j\frac{\partial \psi}{\partial \rho_i}(\rho)(\Im \gamma_i)^2 \quad\text{assuming}\quad \sum_{i=1}^j t^{-\nu_i}y_i\Im\gamma_i=0,
\end{equation}
\begin{multline}
\label{eq:secondform}
-\frac{1}{2} \sum_{i=1}^j\frac{\partial \psi}{\partial \rho_i}(\rho)(\Re \gamma_i)^2+\frac{1}{2}\sum_{i,l=1}^k\frac{\partial^2\psi}{\partial\rho_i\partial\rho_l}(\rho)  \left(\Re\gamma_i\right)\left(\Re\gamma_l\right) \\ \quad\text{assuming}\quad \sum_{i=1}^j t^{-\nu_i}y_i\Re\gamma_i=0.
\end{multline}
To continue we note that 
\begin{equation}
\label{eq:taulim}
\lim_{t\r\infty} (\rho/\log t)=\lim_{t\r\infty} \tilde{\rho}_\tau=\rho_\tau.
\end{equation}
Since $\partial\psi/\partial \rho_i(\rho)>0$ for $t\gg 0$, by \eqref{eq:psipositivity} and \eqref{eq:taulim}, the form in 
\eqref{eq:firstform} is positive definite and hence so is its restriction. A positive definite
quadratic form has maximal rank so we deduce that \eqref{eq:firstform} has rank $j-1$.

We rewrite \eqref{eq:secondform} as 
\begin{equation}
\label{eq:secondformbis}
-\frac{1}{2} \log t\sum_{i=1}^j\frac{\partial \psi}{\partial \rho_i}(\rho/\log t)(\Re \gamma_i)^2+\frac{1}{2}\sum_{i,l=1}^k\frac{\partial^2\psi}{\partial\rho_i\partial\rho_l}(\rho/\log t)  \left(\Re\gamma_i\right)\left(\Re\gamma_l\right).
\end{equation}
Now
we eliminate $\Re \gamma_j$ from \eqref{eq:secondformbis} using the fact that $\sum_{i=1}^j t^{-\nu_i}y_i \Re \gamma_i=0$ (denoting the
resulting vector by $\Re \hat{\gamma}$) so that \eqref{eq:secondformbis}
can be written as 
\[
(\Re \hat{\gamma})^t(\log(t) A(t)+B(t))(\Re \hat{\gamma})
\]
where $A(t),B(t)$ are symmetric $k-1\times k-1$-matrices where the former has rank $j-1$ (we argue as for  \eqref{eq:firstform}) and the latter is positive definite (using
the fact that $(\partial^2\psi/\partial\rho_i\partial\rho_l)(\rho)$ is positive definite).

We now let $t\r \infty$. It follows from \eqref{eq:ytdependence}\eqref{eq:taulim}\eqref{eq:psipositivity} that for $t\r\infty$ the constraints
on $(\Re\gamma_i)_i$ becomes 
\[
\sum_{i=1}^j C_i\Re\gamma_i=0
\]
for suitable constants $C_i>0$. From this one deduces that  $A(\infty):=\lim_{t\r\infty} A(t)$, $B(\infty):=\lim_{t\r\infty} B(t)$ exist and have
the same properties as $A(t),B(t)$. It now remains to show that $\det(\log(t) A(t)+B(t))\neq 0$ for $t\gg 0$. This follows from Lemma \ref{lem:corank2} below with $s=1/\log t$.
\end{proof}

We used Lemma \ref{lem:corank2} below for which we first need the following lemma.
\begin{lemma} \label{lem:corank}
Let $C$, $D$ be symmetric real $l\times l$-matrices with $D$ positive definite. Then the order of the zero of $\det(C+sD)$ in $s=0$ is the corank of $C$.
\end{lemma}
\begin{proof} We write $D=E^t\Sigma E$ where $\Sigma$ is a diagonal matrix with positive entries and $E$ is real and orthogonal. 
Since $D$ is positive definite, $\Sigma$ has strictly positive entries. Put $E'=\Sigma^{1/2} E$ so that $D=E^{\prime t}E'$. We get
$\det(C+sD)=\det(\Sigma) \det (E^{\prime -t} C E'+sI)$. Hence without loss of generality we may and we will assume $D=I$. Now write $C=F^t \Gamma F$ where $\Gamma$ is diagonal and $F$ is real
and orthogonal. Then we get $\det(C+sI)=\det(\Gamma+s I)$. Hence we may assume that $C$ is diagonal and this case is clear.
\end{proof}
\begin{lemma}
\label{lem:corank2}
Let $C(s)$, $D(s)$ be symmetric $l\times l$-matrices depending continuously on $s\in [0,\epsilon[$ for some $\epsilon>0$. Assume that $D(s)$ is positive definite and that $\rk C(s)$ is  constant.
Then $\det(C(s)+sD(s))\neq 0$ for $s\in ]0,\epsilon'[$ for suitable $\epsilon'>0$.
\end{lemma}
\begin{proof} Since $\det(C(0)+s D(0))$ is a non-trivial polynomial (as $\det D(0)\neq 0$) it has at most $l$ roots. Hence there is some $\epsilon\ge \epsilon''>0$ such that
$\det(C(0)+s D(0))$ has no roots in $]0,\epsilon''[$. 
Since the multiplicity of the zero $s=0$ in the polynomial $\det(C(t)+sD(t))$ for $s,t\in [0,\epsilon[$ is independent of $t$ by the hypotheses and Lemma \ref{lem:corank}
  there exists $\delta>0$ such that if $t\in [0,\delta[$ then $\det(C(t)+sD(t))$ has no zeroes in $]0,\epsilon''/2[$. Then $\epsilon'=\min(\delta,\epsilon''/2)$ has the required property.
\end{proof}

\subsubsection{Cocores}
We record here a lemma that we need for a generation of $\WF(D^{\tl})$.
\begin{lemma}\label{lem:properly_embedded}
The unstable manifolds (cocores) of $D^{\tl}$ are properly embedded in $D^{\tl}$.
\end{lemma}

\begin{proof}
We use the notation as in the first paragraph of the proof of Lemma \ref{lem:MB}. 
Let $T_\tau\times W_\tau$ be the stable manifold \cite[Proposition 5.7]{Zhou}\footnote{We use the opposite terminology as in loc.cit., i.e. for us the stable manifold is downward flowing.}, $t$ a generic point of $T_\tau$, $U_{t,\tau}$  the unstable manifold of $(t,\tilde{\rho}_\tau)$.

If $U_{t,\tau}$ is not properly embedded, then there will be some trajectory from $(t,\tilde{\rho}_\tau)$  to $(t',\tilde{\rho}_{\tau'})$ where $(t',\tilde{\rho}_{\tau'})$ is another critical point. 

A stable trajectory ending in $(t',\tilde{\rho}_{\tau'})$ will contain points of the form $(t',\rho)$ so that $t=t'$. 
Also, $\tilde{\rho}_{\tau}$ is in the closure of $W_{\tau'}$, hence $\tau$ is in the closure of $\tau'$ by \cite[Theorem 1]{Zhou}. However, a generic element of $T_\tau$ will not be contained in $T_{\tau'}$ by \cite[(3.2)]{Zhou} if $\tau' \neq \tau$. Thus, $\tau'=\tau$, a contradiction.
\end{proof}

\section{Moser's lemma}

In this section we give a proof of  Moser's lemma for families of Liouville manifolds, Theorem \ref{lem:moser2}. See Remark \ref{rem:moser_generality} for a comment about its generality. 

The basic objects that we use here are introduced in \S\ref{sec:liouville}.

\subsection{Families of Liouville domains and Liouville manifolds}

\subsubsection{Classifying generating Liouville domains}\label{sec:class_gen_Lio_dom}
 Let $N$ be a Liouville manifold with generating Liouville domain $M$. Recall that the \emph{skeleton} of $N$ is
 \begin{equation}
   \label{eq:capformula}
 \Skel(N)=\bigcap_{s\leq 0}\phi_{Z_\theta,s}M.
   \end{equation}
Since the skeleton can be equivalently described as the set of points of $N$ that do not escape under the flow of $Z_\theta$, we have
   \begin{equation}
     \label{eq:coprod}
  M=\Skel(N) \coprod \left(\bigcup_{s\le 0} \phi_{Z_\theta,s}(\partial M)\right).
\end{equation}
In particular we obtain:
\begin{lemma}
  \label{lem:uniqueboundary}
 A generating Liouville domain is uniquely determined by its boundary.
 \end{lemma}

Generating Liouville domains are easily classified.
\begin{lemma} \label{lem:generation} Let $(N,\theta)$ be a Liouville manifold.
\label{item:h} $\partial M_i$ for $i=1,2$ are boundaries of generating Liouville domains then there exists a unique function $h:N\setminus\Skel(N)\r \RR$
  invariant under the Liouville flow such that
  $\partial M_2=\phi_{Z_\theta,h}(\partial M_1)$. Moreover $h$ is smooth.
\end{lemma}
\begin{proof}
Since the skeleton is given by the orbits that do not escape to infinity we have
  diffeomorphisms
  \[
   \phi_i: \partial M_i\times \RR\r N\setminus \Skel(N):(m,s)\mapsto \phi_{Z_\theta,s}(m).
 \]
 The composition
 \[
  \partial M_1\times \RR \xrightarrow{\phi_2^{-1}\phi_1} \partial M_2\times \RR
\]
is a diffeomorphism invariant under translation in $\RR$. The function
$h$ is the composition
\[
 \partial M_1\times \RR
  \xrightarrow{\phi_2^{-1}\phi_1} \partial M_2\times \RR\xrightarrow{(m,s)\mapsto -s} \RR.\qedhere
  \]
\end{proof}
\subsubsection{Interpolating Liouville domains}
\label{sec:interpolating}
Let $(N,\theta)$ be a Liouville manifold. We will say that submanifold $K\subset N$ is a \emph{Liouville boundary} if it is the boundary
of a generating Liouville domain $M$. By \eqref{eq:coprod} $M$ is determined by $K$ via $M=\Skel(N)\cup\left(\bigcup_{s\le 0} \phi_{Z_\theta,s} K\right)$ and
by
Lemma \ref{lem:generation}
the
Liouville boundaries $K'$ in $N$ are in 1-1 correspondence with smooth functions $h:N-\Skel(N)\r \RR$, invariant under the Liouville flow, via
$K'=\phi_{Z_\theta,h}(K)$.

\medskip

Let
$(M_i)_{i=1,\ldots,t}$ be a collection of generating Liouville domains in $N$. We take 
$(a_i)_{i=1,\ldots,t}\in \RR$ be such that $\sum_i a_i=1$. Then we define
an \emph{interpolated Liouville domain} $M$
as follows. Pick $j\in \{1,\ldots,t\}$. Then there exist smooth
functions $h_{ji}$ on $N-\Skel(N)$, invariant under the Liouville flow, such that
$\partial M_i=\phi_{Z_\theta,h_{ji}}(\partial M_j)$ (in particular $h_{jj}=0$). Put $h_j=\sum_i a_i h_{ji}$. Then it is easy to see that
$K:=\phi_{Z_\theta,h_j}(\partial M_j)$ is independent  of $j$\footnote{Indeed, the uniqueness of $h_{ji}$ implies that $h_{ji}+h_{lj}+h_{il}=h_{ii}=0$ and $h_{il}=-h_{li}$, from which it follows that $h_j+h_{lj}=\sum_i a_i (h_{ji}+h_{lj})=\sum_ia_i h_{li}$, and hence $\phi_{Z_\theta,h_j}(\partial M_j)=\phi_{Z_\theta,h_j}\phi_{Z_\theta,h_{lj}}(\partial M_l)=\phi_{Z_\theta,h_l}(\partial M_l)$.}, and moreover it is a Liouville boundary as explained in the first paragraph. We let $M$ be the resulting Liouville domain.
We will use the ad hoc notation $M=\sum_i a_i M_i$.
\begin{remark} \label{rem:finite}
  If $a_j=0$ then it follows from the construction that $M=\sum_{i\neq j} a_i M_i$. So we can make sense of $\sum_{i\in I} a_i M_i$ for any collection
  of real numbers $(a_i)_{i\in I}$ provided that at most a finite number are non-zero.
\end{remark}

\subsubsection{Generating family of Liouville domains with submersive boundary}
\begin{definition}
Let $B$ be a manifold. A family of manifolds (with boundary) is a surjective submersion $M\r B$.
We will often  informally write such a family as $(M_t)_{t\in B}$.
We say that a family has \emph{submersive boundary} if
 $\partial_h M:=\coprod_{t\in B} \partial M_t$ is a surjective submersion as well. 
 If $B$ has a boundary (or even corners), then we also require the submersivity condition for the corresponding
 strata on $\partial_h M$.
\end{definition}
In this section we discuss families of Liouville domains $N\r B$. It turns out that the results of the previous sections for single Liouville domains generalize directly to families
if they contain a family of generating Liouville domains with submersive boundary. Unfortunately imposing submersivity on the boundary is too inconvenient in practice.
Luckily it turns out that the existence of a  family of generating Liouville domains with submersive boundary is true under seemingly much weaker
hypothesis.  The following result will be the main result of this section.

\begin{proposition} \label{prop:main} Let $N\r B$ be a family of Liouville manifolds containing locally over $B$ generating families of Liouville domains. Then $N$ contains a generating family
  of Liouville domains with submersive boundary.
\end{proposition}
\subsubsection{Reminder about connections}
\label{sec:connection}
  Recall that if $\gamma:M\r B$ is a family then a connection
  \[
    \nabla:\gamma^\ast (TB)\r TM
    \]
is a splitting of the natural map of vector bundles \cite[Definition 9.3]{MR1202431}
\[
  TM\r \gamma^\ast(TB)
\]
A connection always exists. Indeed it can be constructed locally by lifting sections and then be globalized using a partition of unity. A connection provides
a recipe for lifting vector fields on $B$ to vector fields on $M$.

If $M$ has a submersive boundary $\partial_h M$ then   one may check locally that the map\footnote{If $B$ has boundary then we should consider $T(M,\partial M)\r \gamma^*T(B,\partial B)$.} 
\begin{equation}
  \label{eq:connectionboundary}
    \gamma:T(M,\partial_hM)\r \gamma^\ast (TB)
  \end{equation}
  is surjective
  where $T(M,\partial_h M)$ is the vector bundle whose sections are vector fields on $M$ parallel to $\partial_h M$.
  A connection
  \[
    \nabla:\gamma^\ast(TB)\r T(M,\partial_h M)
  \]
  is now defined as a splitting of \eqref{eq:connectionboundary}. Such a connection is a recipe for lifting vector fields on $B$ to vector fields on $M$ parallel to $\partial_h M$.
 \subsubsection{Some useful results}
 For brevity we only give proofs in the case that $B$ has no boundary. The generalizations to the case where $B$ also has a boundary
   are straightforward.

 \begin{lemma}[Ehresmann with boundary]
  \label{lem:ehresmann} Let $\phi:M\r B$ be a proper family of manifolds with submersive boundary.
Then $\phi$ is a locally trivial.
\end{lemma}
\begin{proof}
  We choose a connection $\nabla:\phi^\ast(TB)\r T(M,\partial_h M)$ and then we use parallel transfer for this connection in the usual way.
  \end{proof}
\begin{lemma} \label{lem:completions}
  Let $(M,\theta)\r B$ be a  family of 
  Liouville domains with submersive boundary. Then
\[
\gamma:\partial_h M\times ]-\infty,0]\r M:(m,s)\mapsto \phi_{Z_{\theta},s}(m)
\]
is an injective immersion.
\end{lemma}
\begin{proof} $\gamma$ is clearly injective. So it is suffices to show that
  $d\gamma$ is invertible.
  Since $\gamma(m,s+r)=\phi_{Z_{\theta},r}(\gamma(m,s))$ it suffices
  to do this for a point of the form $(m,0)$ with $m\in \partial M_t$. It
  is then sufficient that $Z_{\theta_t}$ restricted to $\partial_h M$ is
  transversal to $\partial_h M$. Now $Z_{\theta_t}$ is parallel to $M_t$ and since
  $\partial_h M\r B$ is a submersion, $M_t$ is transversal to $\partial_h M$.
  Finally $Z_{\theta_t}$ is transversal to $\partial M_t$. This finishes the proof.
\end{proof}
\begin{remark} The hypothesis in Lemma \ref{lem:completions}
  that $\partial_h M\r B$ is a submersion is necessary for the conclusion to hold.
  Indeed if the conclusion holds then $\partial_h M\times ]-\infty,0]\r B$ is a submersion but this map factors through $\partial_h M\r B$. 
\end{remark}
\begin{definition}
  \label{def:completion}
If $M\r B$ is a family of Liouville domains then we define the completion $\hat{M}$ using the formula \eqref{eq:completionformula}.
\end{definition}

\subsubsection{Interpolation in families}
Let $B$ be a manifold, possibly with boundary. Let $(N,\theta)\r B$ be a family of Liouville manifolds.
A family of Liouville boundaries is an inclusion $K\hookrightarrow N$
such that
\begin{itemize}
\item $K\hookrightarrow N$ is a submanifold,
  \item $K_t$ is a Liouville boundary in $N_t$ for all $t\in B$.
  \end{itemize}
  If in addition $K\r B$ is a surjective submersion then we say that the family is submersive (with the usual understanding that if $B$ has a boundary or corners
  we also require the surjective submersion condition on the corresponding strata in $K$).

  \begin{lemma} $M\mapsto \partial_h M$ defines a 1-1 correspondence between families of generating Liouville domains with submersive boundary and
    submersive Liouville boundaries.
  \end{lemma}
  \begin{proof} This is proved like Lemma \ref{lem:uniqueboundary}.
 \end{proof}

  The skeleton of a family of Liouville manifolds $N\r B$ is defined fiberwise; i.e.
  \[
    \Skel(N)=\bigcup_{t\in B} \Skel(N_t).
    \]  
  \begin{lemma} \label{lem:completions2}
    Let $K\hookrightarrow N$ be a submersive family of Liouville boundaries.  Then we have an injective immersion
  \[
    \gamma:K\times \RR \r N:(m,s)\r \phi_{Z_\theta,s}(m)
  \]
  whose image is $N-\Skel(N)$ (in particular $\Skel(N)$ is closed).
\end{lemma}
\begin{proof} This is proved in a similar way as Lemma \ref{lem:completions}.
 \end{proof}

The following lemma is a family version of Lemma \ref{lem:generation}.

  \begin{lemma} \label{lem:radial}
    Assume that $K\hookrightarrow N\r B$, $K'\hookrightarrow N\r B$ are submersive families of Liouville boundaries.
    Then there exists a (unique) smooth function $h:N-\Skel(N)\r \RR$, invariant under the Liouville flow, such that $K'=\phi_{Z_\theta,h}(K)$.
  \end{lemma}
\begin{proof}
  By Lemma \ref{lem:completions2} we have a diffeomorphisms
  \begin{align*}
    \gamma:K\times \RR \r N-\Skel(N)&:(m,s)\r \phi_{Z_\theta,s}(m),\\
    \gamma:K'\times \RR \r N-\Skel(N)&:(m,s)\r \phi_{Z_\theta,s}(m).
  \end{align*}
We now proceed as in Lemma \ref{lem:generation}.
  \end{proof}
\begin{lemma} Liouville boundaries are stable under radial expansion. The same goes for submersive Liouville boundaries.
\end{lemma}
\begin{proof} Liouville boundaries are defined fiberwise. 
  So the only thing to check is submersivity. Let $K$ be submersive Liouville boundaries and let $K'$ be a radial expansion of $K$; i.e.\ there is a  smooth function $h$ on $N-\Skel(N)$ such that
$K'=\phi_{Z_\theta,h}(K)$. By Lemma \ref{lem:radial} there is a smooth function $h'$ on $N-\Skel(N)$ such that $K=\phi_{Z_\theta,h'}(K')$ 
and one checks that $\phi_{Z_\theta,h}$ and $\phi_{Z_\theta,h'}$ are each other's inverse so they are diffeomorphisms.

The map $K'\r B$ factors through the projection $\phi_{Z_\theta,h}:K'\r K$ and hence
is is a submersion.
\end{proof}
\begin{corollary} \label{cor:inter}
  The interpolation construction for Liouville domains exhibited in \S\ref{sec:class_gen_Lio_dom} works in families, with the $(a_i)_i$ being smooth functions on $B$, provided one restricts to families of Liouville domains with submersive boundary.
\end{corollary}
\subsubsection{Proof of Proposition \ref{prop:main}}

\begin{lemma} \label{lem:local}
  Let $\phi:N\r B$ be a family of Liouville manifolds containing locally (over $B$) families of generating Liouville domains.  
  Let $t\in B$. Then there is an open neighbourhood $t\in V\subset B$ and a family of Liouville
  domains $M'\r V$ in $N\times_B V$ with submersive boundary.
\end{lemma}
\begin{proof}  Choose a connection $\nabla$ on $N$ as in \S\ref{sec:connection}. Let $M\r B$ be a family of generating Liouville domains in $N$ and  $M'_t=\phi_{Z_\theta,s}(M_t)$ for an arbitrary $s>1$.
  This is a generating Liouville domain in $N_t$.
 Choose a smoothly contractible open neighbourhood $V$ of $t$ and let $M'\r V$ be obtained from $M'_t$ by parallel transport for $\nabla$ (a contracting homotopy defines a path $v\r t$ for every $v\in V$). By construction $M'\r V$, $\partial_h M' \r V$ are surjective and   have sections, locally on $M'$ and $\partial_h M'$.
  Hence $M'\r V$ is a family with submersive boundary by \cite[Theorem B.3]{book:fibers}.

  $M'$ is generating since by choosing $V$ small enough we may ensure that $M'$ contains
  $M\times_B V$. Moreover, again choosing $V$ small enough, we may ensure that $Z_\theta$ is outward pointing on $\partial_h M'$, since that too is stable under deformation.
\end{proof}
\begin{proof}[Proof of Proposition \ref{prop:main}] For every $t\in B$ choose an open neighbourhood of $V_t$ of $t$ and a family of generating Liouville domains $M_t\r V_t$ with submersive boundary (made possible
  by Lemma \ref{lem:local}). Now choose a partition of unity $\phi_t$ subordinate to the open cover $(V_t)_t$ of $B$ and  put
  \[
    M=\sum_{t\in B} \phi_t M_t.
  \]
  By Remark \ref{rem:finite} and Corollary \ref{cor:inter}
  this expression makes sense in a neighbourhood of any point. Hence it is globally defined. $M$ is a generating Liouville domain with submersive boundary,
  again because this is true in a neighbourhood of any point.
\end{proof}

 \subsection{Moser's lemma for families of Liouville manifolds}
 \begin{lemma}
   \label{lem:smooth}
   Let $(M,\theta_t)_{t\in B}$ be a family of Liouville domains (i.e.\ constant if we forget the Liouville forms).
   Let $\gamma:\hat{M}\r B$ be the corresponding family of Liouville manifolds (see Definition \ref{def:completion}).
   In particular $\hat{M}$ embeds the trivial families $M\times B$ and
   $\partial M\times [0,\infty[\times B$ (but the smooth structure varies in a neighbourhood of $\partial M$).
   Let $K$ be a compact set in the interior of $M$.
   Let $b\in B$. Then there is an open neighbourhood $V$ of $b$ and a diffeomorphism~$\psi:\hat{M}_b\times V \r \hat{M}\times_B V$ over $B$ such that
   \begin{enumerate}
     \item $\psi_b$ is the identity $\hat{M}_b\r \hat{M}_b$,
   \item $\psi$ restricts to the identity $K\times B\r K\times B$,
   \item $\psi$ restricts to the identity $\partial M\times [0,\infty[\times B
     \r \partial M\times [0,\infty[\times B$.
   \end{enumerate}
 \end{lemma}
 \begin{proof}
  We choose a connection $\nabla:\gamma^\ast(TB)\r TN$
   which is trivial on $K\times B$ and $M\times [0,\infty[\times B$ (i.e.\ they are the natural splittings). Choose
   a smoothly contractible open neighbourhood $V$ of $b$ and use parallel transport for $\nabla$ to make the identification $\psi:\hat{M}_b\times V \r \hat{M}\times_B V$.
\end{proof}
\begin{lemma}
  \label{lem:ultimate}
  Assume  $(N_t,\theta_t)_{t\in B}$ is a family of Liouville manifolds which locally has families of generating Liouville
  domains, such that $B$ is smoothly contractible and let $b\in B$.
  Then there exists a
  diffeomorphism $\psi:N_b\times B\r N$
  over $B$, inducing the identity $N_b\r N_b$  which commutes with the Liouville flow outside $K\times B$ for a compact set $K\subset N_b$ where $N_b\times B$ is equipped with the constant Liouville form $\theta_b$.
\end{lemma}
\begin{proof} Let $(M_t)_{t\in B}$ be a generating family of Liouville domains with submersive boundary for $(N_t,\theta_t)$ (see Proposition \ref{prop:main}). By a family version of Lemma \ref{lem:open}
  combined with Lemma \ref{lem:completions} we have a diffeomorphism of families $\psi_1:(\hat{M}_t,\hat{\theta}_t)_t \cong (N_t,\theta_t)$.
  By Lemma \ref{lem:ehresmann} and \cite[Theorem B.7]{book:fibers}
  we obtain a diffeomorphism $\psi_2:M_b\times B\r M$ over $B$. Let $(M_b,\theta'_t)$ be the family of Liouville domains
  obtained by pullback from $\psi_2$ and let $(M_b\times B)\,\hat{}\r B$ be the corresponding Liouville completion. We now obtain the diffeomorphism 
  $\psi_3:\hat{M}_b\times B\r (M_b\times B)\,\hat{}$\, from Lemma \ref{lem:smooth}. We put $\psi=\psi_1\hat{\psi}_2\psi_3\psi_{1,b}^{-1}$. For $K$ we may take $M_b$.
\end{proof}
\subsubsection{Proof of Moser's lemma}\label{subsubsec:proof_Moser}
We  first make more explicit the relation between a Liouville manifold $(N,\theta)$ and a generating Liouville domain $(M,\theta)$

There is an embedding $\partial M\times ]-\infty,\infty[_s\subset N:(m,s)\mapsto \phi_{Z_\theta,s}(m)$. It has the following properties.
\begin{itemize}
  \item $N=M\cup \partial M\times ]-\infty,\infty[_s$.
\item $M\cap \partial M\times ]-\infty,\infty[_s=\partial M\times ]-\infty,0]_s$.
\item If $\theta_\partial=\theta{|} \partial M $ then on  $\partial M\times ]-\infty,\infty[_s$ we have
  \begin{equation*}
    \label{eq:thetapartial}
\theta=    e^s\theta_\partial.
\end{equation*}
\item It follows that if $\omega_\partial=\omega_{\theta}{|} \partial M$ then on $\partial M\times ]-\infty,\infty[_s$
  \[
\omega_{\theta}=e^s ds \,\theta_\partial+e^s\omega_\partial.
\]
\item Finally on $\partial M\times ]-\infty,\infty[_s$ we also have
  \[
     Z_{\theta}=\partial/\partial s.
   \]
 \end{itemize}
\begin{remark}
If $\alpha$ is a 1-form on $N$ let $X$ be the corresponding dual vector field, i.e.    $i_X\omega_\theta=\alpha$.
Write on $\partial M\times ]-\infty,\infty[_s$
\[
  \alpha=\alpha_\partial+\alpha_s ds, \qquad X=X_\partial+X_s \partial/\partial s
  \]
  where $\alpha_\partial$ and $X_\partial$ are respectively an $s$-dependent 1-form and vector field on $\partial M$ and $\alpha_s$ and $X_s$ are smooth functions. We compute
  \[
    i_X \omega_\theta=-e^s i_{X_\partial} \theta_\partial\, ds+e^s i_{X_\partial } \omega_{\partial} +e^s X_s \theta_\partial
    \]
    so that we get the equations
    \begin{align}
      e^s i_{X_\partial } \omega_{\partial} +e^s X_s \theta_\partial&=\alpha_\partial\label{eq:alphapartial}\\
-e^s i_{X_\partial} \theta_\partial&=\alpha_s. \label{eq:alphas}
\end{align}
By contracting with the Reeb vector field\footnote{The Reeb vector field $R$ is the vector field on 
  $\partial M$ such that $i_R\omega_\partial=0$,
  $i_R\theta_\partial=1$. It exists because $\partial M$ is a contact manifold with contact form $\theta_\partial$.} 
     we see that first of these equations determines $X_s$ uniquely.
Given this, it is also clear that it determines $X_{\partial}$ up to a multiple of the
Reeb
vector field.
The second equation determines this multiple. 
\end{remark}

\begin{proof}[Proof of Theorem \ref{lem:moser2}]
We choose $\psi_1:N_b\times B\r N$ for $\psi$ as in  Lemma \ref{lem:ultimate}. 
The pull back under $\psi_2$ defines a family of Liouville forms $\theta'_t$ on $N_b$.
We will construct a diffeomorphism over $B$: $\psi_2:N_b\times B\r N_b\times B$, such that $\psi_{2,b}$ is the
identity and a function $f\in C^\infty(N_b\times B)$, supported on some $K\times B$, $K$ compact, satisfying $f_b=0$ such that
$\psi_{2,t}^\ast(\theta'_t)+df_t=\theta'_b$. The construction of $\psi_2$ will not depend on any more choices. Finally we put $\psi=\psi_1\psi_2$.

Without loss of generality we may assume that $\psi_1$ is the identity. So we assume $N=N_b\times B$ where $N_b\times \{t\}$ is equipped with $\theta_t$.

A choice of contracting homotopy $h_t:B\times [0,1]\r B$ with $h_1$ being the constant map $B\r \{b\}$ fixes for every $t$  a path $t\r b$. Since our construction will not depend on choices we may replace $B$ by $[0,1]$ and put $b=0$.

We will now obtain $\psi_t$ as the flow of a time dependent smooth vector field $(X_t)_{t\in [0,1]}$ on $N_0$.
Assume we have constructed $\psi_t$, $f_t$.
Differentiating
\begin{equation}
  \label{eq:defeq}
  \psi_t^\ast(\theta_t)+df_t=\theta_0\qquad (\psi_0=\Id, f_0=0)
\end{equation}
with respect to $t$ yields
\[
  \psi_t^\ast\left(L_{X_t}
    \theta_t+\frac{d\theta_t}{dt}\right)+d\frac{df_t}{dt}=0
\]
which is equivalent with (for $\omega_t=d\theta_t$)
\[
  \psi_t^\ast\left(i_{X_t}
    \omega_t+di_{X_t}\theta_t+\frac{d\theta_t}{dt}\right)+d\frac{df_t}{dt}=0
\]
or
\begin{equation}
  \label{eq:defeq2}
  \psi_t^\ast\left(i_{X_t}\omega_t+\frac{d\theta_t}{dt}\right)+d\left(\psi^\ast_t (i_{X_t} \theta_t)+\frac{df_t}{dt}\right)=0.
\end{equation}
We now construct a solution to \eqref{eq:defeq2}. If $X_t$ is integrable then this yields a solution to \eqref{eq:defeq}.

We get solutions of \eqref{eq:defeq2} from solutions of the following
two equations.
\begin{align}
  i_{X_t}\omega_t+\frac{d\theta_t}{dt}&=0\label{eq:defeq3},\\
  \psi^\ast_t (i_{X_t} \theta_t)+\frac{df_t}{dt}&=0\label{eq:defeq4}.
\end{align}
\eqref{eq:defeq3} fixes $X_t$ and, provided $X_t$ is integrable, \eqref{eq:defeq4} fixes $f_t$.

\medskip

We now show that the constructed $X_t$ is indeed integrable, and moreover that $f_t$ has compact support.
We let $M_0\subset N_0$ be a Liouville domain containing $K$ (as in Lemma \ref{lem:ultimate}) in its interior.
By Lemma \ref{lem:tautology} below $M_0$ is a common Liouville domain for the $\theta_t$ and moreover 
the restrictions of $\theta_t$ to $\partial M_0\times [0,\infty[_s$ are 
of the form $e^s \theta_{t,\partial}$.

On $\partial M_0\times [0,\infty[_s$ we get from (\ref{eq:alphapartial},\ref{eq:alphas})
\begin{align*}
  e^s i_{X_{t,\partial}}\omega_{t,\partial}+e^s X_{t,s}\theta_{t,\partial}&=-e^s \frac{d\theta_{\partial,t}}{dt}\\
  -e^s i_{X_{t,\partial}}\theta_{t,\partial}&=0
\end{align*}
or simplified
\begin{align}
   i_{X_{t,\partial}}\omega_{t,\partial}+ X_{t,s}\theta_{t,\partial}&=- \frac{d\theta_{\partial,t}}{dt}\\
   i_{X_{t,\partial}}\theta_{t,\partial}&=0.\label{eq:simplified}
\end{align}
Since $X_{t,\partial}$, $X_{t,s}$ are the unique solutions of a system of linear equations, independent of $s$, they must themselves be independent of $s$ (for $s\ge 0$).
Since $X_{t,s}=X_t$ is constant it follows from \cite[Theorem 2.2]{MR2839561} that $X_t$ is integrable.

From \eqref{eq:defeq4} we get
\[
  f_t=-\int_{0}^t \psi^\ast_\tau (i_{X_\tau}\theta_\tau) d\tau
\]
(since $f_0=0$). So for $s\ge 0$:
\[
    f_t=-e^s\int_{0}^t \psi^\ast_\tau (i_{X_{\tau,\partial}}\theta_{\tau,\partial}) d\tau=0
  \]
  using \eqref{eq:simplified}. Thus $f_t$ is supported inside $\coprod_t \psi^{-1}_t M_0$. This is a compact set since it is the image of $M_0\times [0,1]$
  under the continuous map $N\times [0,1]:(t,n)\mapsto \psi_t^{-1}(n)$.
 So the  support of $f_t$ is compact. 
\end{proof}

The following lemma used above is a tautology.
\begin{lemma} \label{lem:tautology}
  Let $(N_1,\theta_1), (N_2,\theta_2)$ be Liouville manifolds and let $\psi:N_1\r N_2$ be a diffeomorphism which commutes with the Liouville flow
  outside $M_1-\partial M_1\subset N_1$ for $M_1\subset N_1$ a generating Liouville domain. Then $\psi(M_1)$ is a generating Liouville domain in $(N_2,\theta_2)$.
\end{lemma} 

\bibliographystyle{amsalpha}

\end{document}